\documentclass[12pt,a4paper,reqno]{amsart}
\usepackage{amsmath,amssymb,amsthm,wasysym,calc,verbatim,enumitem,tikz,url,hyperref,mathrsfs,bbm,cite,fullpage}
\usetikzlibrary{shapes.misc,calc,intersections,patterns,decorations.pathreplacing, calligraphy}
\hypersetup{colorlinks=true,citecolor=blue,linktoc=page}
\hypersetup{%
    colorlinks,
    linkcolor={red!50!black},
    citecolor={green!50!black},
    urlcolor={blue!50!black}
}

\usepackage[abbrev,msc-links,backrefs]{amsrefs}
\usepackage{doi}

\renewcommand{\PrintDOI}[1]{\doi{#1}}

\AtBeginDocument{\def\MR#1{}}

\usepackage{tikz}
\usetikzlibrary{decorations.markings}
\usetikzlibrary{calc,positioning,decorations.pathmorphing,decorations.pathreplacing}

\usepackage{fullpage}
\usepackage{setspace}
\usepackage{datetime}
\usepackage[margin=1cm]{caption}
\usepackage{graphicx}

\newtheorem{theorem}{Theorem}
\newtheorem{lemma}[theorem]{Lemma}
\newtheorem{proposition}[theorem]{Proposition}

\newtheorem{fact}[theorem]{Fact}
\theoremstyle{definition}
\newtheorem{definition}[theorem]{Definition}
\newtheorem{example}[theorem]{Example}

\usepackage{enumitem}

\let\polishlcross=\l
\def\l{\ifmmode\ell\else\polishlcross\fi}

\makeatletter
\def\moverlay{\mathpalette\mov@rlay}
\def\mov@rlay#1#2{\leavevmode\vtop{    \baselineskip\z@skip\lineskiplimit-\maxdimen%
    \ialign{\hfil$\m@th#1##$\hfil\cr#2\crcr}}}
\newcommand{\charfusion}[3][\mathord]{
    #1{\ifx#1\mathop\vphantom{#2}\fi
        \mathpalette\mov@rlay{#2\cr#3}
      }
    \ifx#1\mathop\expandafter\displaylimits\fi}
\makeatother

\DeclareFontFamily{U}  {MnSymbolC}{}
\DeclareSymbolFont{MnSyC}         {U}  {MnSymbolC}{m}{n}
\DeclareFontShape{U}{MnSymbolC}{m}{n}{%
    <-6>  MnSymbolC5
   <6-7>  MnSymbolC6
   <7-8>  MnSymbolC7
   <8-9>  MnSymbolC8
   <9-10> MnSymbolC9
  <10-12> MnSymbolC10
  <12->   MnSymbolC12}{}
\DeclareMathSymbol{\powerset}{\mathord}{MnSyC}{180}

\makeatletter
\def\namedlabel#1#2{\begingroup
    #2%
    \def\@currentlabel{#2}%
    \phantomsection\label{#1}\endgroup
}
\makeatother

\numberwithin{theorem}{section}
\setlist[itemize]{leftmargin=1cm}
\setlist[enumerate]{leftmargin=1cm}

\everydisplay{\thickmuskip=7mu}
\renewcommand{\leq}{\leqslant}
\renewcommand{\geq}{\geqslant}

\let\epsilon\varepsilon%

\definecolor{green1}{rgb}{0.0, 0.5, 0.0}

\def\cC{\mathcal{C}}

\def\cI{\mathcal{I}}

\def\1{\mathbbm{1}}
\def\<{\langle}
\def\>{\rangle}

\let\theta=\vartheta%
\let\rho=\varrho%
\let\phi=\varphi%

\DeclareMathOperator{\smp}{smp}

\pgfdeclaredecoration{complete sines}{initial}
{
    \state{initial}[
        width=+0pt,
        next state=sine,
        persistent precomputation={\pgfmathsetmacro\matchinglength{
            \pgfdecoratedinputsegmentlength / int(\pgfdecoratedinputsegmentlength/\pgfdecorationsegmentlength)}
            \setlength{\pgfdecorationsegmentlength}{\matchinglength pt}
        }] {}
    \state{sine}[width=\pgfdecorationsegmentlength]{
        \pgfpathsine{\pgfpoint{0.25\pgfdecorationsegmentlength}{0.5\pgfdecorationsegmentamplitude}}
        \pgfpathcosine{\pgfpoint{0.25\pgfdecorationsegmentlength}{-0.5\pgfdecorationsegmentamplitude}}
        \pgfpathsine{\pgfpoint{0.25\pgfdecorationsegmentlength}{-0.5\pgfdecorationsegmentamplitude}}
        \pgfpathcosine{\pgfpoint{0.25\pgfdecorationsegmentlength}{0.5\pgfdecorationsegmentamplitude}}
}
    \state{final}{}
}

\usepackage{cases}
\usepackage{eqparbox}

 \usepackage{lineno}
 \newcommand*\patchAmsMathEnvironmentForLineno[1]{
 \expandafter\let\csname old#1\expandafter\endcsname\csname #1\endcsname
 \expandafter\let\csname oldend#1\expandafter\endcsname\csname end#1\endcsname
 \renewenvironment{#1}
 {\linenomath\csname old#1\endcsname}
 {\csname oldend#1\endcsname\endlinenomath}}
 \newcommand*\patchBothAmsMathEnvironmentsForLineno[1]{
 \patchAmsMathEnvironmentForLineno{#1}
 \patchAmsMathEnvironmentForLineno{#1*}}
 \AtBeginDocument{
 \patchBothAmsMathEnvironmentsForLineno{equation}
 \patchBothAmsMathEnvironmentsForLineno{align}
 \patchBothAmsMathEnvironmentsForLineno{flalign}
 \patchBothAmsMathEnvironmentsForLineno{alignat}
 \patchBothAmsMathEnvironmentsForLineno{gather}
 \patchBothAmsMathEnvironmentsForLineno{multline}
 \patchAmsMathEnvironmentForLineno{numcases}
 }

\begin{document}
\onehalfspace%
\shortdate%
\yyyymmdddate%
\settimeformat{ampmtime}
\footskip=28pt

\title{On the Ramsey numbers of daisies II}

\author{Marcelo Sales}

\thanks{The author was supported by NSF grant DMS 1764385 and US Air Force grant FA9550-23-1-0298.}

\address{Department of Mathematics, University of California, 
    Irvine, CA, USA}
\email{mtsales@uci.edu}

\begin{abstract}
A $(k+r)$-uniform hypergraph $H$ on $(k+m)$ vertices is an $(r,m,k)$-daisy if there exists a partition of the vertices $V(H)=K\cup M$ with $|K|=k$, $|M|=m$ such that the set of edges of $H$ is all the $(k+r)$-tuples $K\cup P$, where $P$ is an $r$-tuple of $M$. Complementing results in \cite{PRS}, we obtain an $(r-2)$-iterated exponential lower bound to the Ramsey number of an $(r,m,k)$-daisy for $2$-colors. This matches the order of magnitude of the best lower bounds for the Ramsey number of a complete $r$-graph.
\end{abstract}

\maketitle

\section{Introduction}

For a natural number $N$, we set $[N]=\{1,\ldots,N\}$. Given a set $X$, we denote by $X^{(r)}$ the set of $r$-tuples of $X$. For two sets $X$, $Y$ we say that $X<Y$ if $\max(X) < \min(Y)$. Unless stated otherwise, the elements of a set $X$ will be always displayed in increasing order. That is, if $X=\{x_1,\ldots,x_t\}$, then $x_1<\ldots<x_t$.

A $(k+r)$-uniform hypergraph $H$ on $k+m$ vertices is an \textit{$(r,m,k)$-daisy} if there exists a partition of the vertices $V(H)=K\cup M$ with $|K|=k$ and $|M|=m$ such that
\begin{align*}
    H=\{K\cup P: P \in M^{(r)}\}
\end{align*}
We say that the set $K$ is the kernel of $H$, the elements of $M^{(r)}$ are the petals of $H$ and $M$ is the universe of petals. We will often refer to an edge of $H$ by $X$ and its correspondent petal by $P$.

Daisies were first introduced by Bollob\'{a}s, Leader and Malvenuto in \cite{BLM11}. They were interested in Tur\'{a}n-type questions related to $(r,m,k)$-daisies, i.e., the maximum number of edges that an $(r+k)$-graph has with no copy of an $(r,m,k)$-daisy. In this paper we will study the Ramsey number $D_r(m,k)$ of an $(r,m,k)$-daisy. The number $D_r(m,k)$ is defined
as the minimum integer $N$ such that any $2$-coloring of the complete hypergraph $[N]^{(k+r)}$ contains a monochromatic $(r,m,k)$-daisy.

Those numbers were already studied in \cite{PRS}. Although the main focus of their paper is on daisies with kernel of non fixed size, they noted that
\begin{align}\label{eq:eq1}
R_{r-k}(\lceil m/(k+1) \rceil-k)\leq D_r(m,k)\leq R_r(m)+k,
\end{align}
where $R_r(m)$ is the Ramsey number of the complete graph $K_m^{(r)}$, i.e., the minimum integer $N$ such that any $2$-coloring of $[N]^{(r)}$ contains a monochromatic set $X$ of size $m$. 

A natural question raised in \cite{PRS} is whether $D_r(m,k)$ behaves similarly $R_r(m)$. Erd\H{o}s, Hajnal and Rado (see \cites{EHR65, GRS13}) and Conlon, Fox and Sudakov \cite{CFS13} showed that there exists absolute constants $c_1,c_2$ such that for sufficiently large $m$,
\begin{align}\label{eq:ramsey}
t_{r-2}(c_1m^2)\leq R_r(m)\leq t_{r-1}(c_2m),
\end{align}
where $t_i(x)$ is the tower function defined by $t_0(x)=x$ and $t_{i+1}(x)=2^{t_i(x)}$. In this paper we provide for $k\geq 1$ a lower bound of $D_r(m,k)$ in the same order of magnitude as the best current bounds of the Ramsey number $R_r(m)$ for sufficiently large $m$. We remark here that for $k=0$, the problem is equivalent to the Ramsey number, since an $(r,m,0)$-daisy is just the complete graph $K_m^{(r)}$.

\begin{theorem}\label{th:main}
Let $r\geq 3$ and $k\geq 1$ be integers. There exist integer $m_0=m_0(r,k)$ and absolute constant $c$ such that 
\begin{align*}
    D_r(m,k)\geq t_{r-2}(ck^{-2}m^{2^{4-r}})
\end{align*}
holds for $m\geq m_0$.
\end{theorem}

In order to prove Theorem \ref{th:main} we will actually study the Ramsey number of a subfamily of daisies. We say that a hypergraph $H$ is a \textit{simple $(r,m,k)$-daisy} if $H$ is an $(r,m,k)$-daisy and its kernel $K$ can be partitioned into $K=K_0\cup K_1$ such that $K_0<M<K_1$. We define the Ramsey number of simple $(r,m,k)$-daisies $D_r^{\smp}(m,k)$ as the minimum integer $N$ such that any $2$-coloring of the complete hypergraph $[N]^{(k+r)}$ yields a monochromatic copy of a simple $(r,m,k)$-daisy.

In \cite{PRS}, the authors observed that the Ramsey number of daisies can be bounded from below by the Ramsey number of simple daisies.

\begin{proposition}[\cite{PRS}, Proposition 5.3]\label{prop:pigeonhole}
$D_r(m,k)\geq D_r^{\smp}\left(\lceil m/(k+1) \rceil, k\right)$.
\end{proposition}

Our main technical result is an $(r-2)$-iterated exponential lower bound for the Ramsey number of simple $(r,m,k)$-daisies. Note that Theorem \ref{th:main} is a corollary from Proposition \ref{prop:pigeonhole} and Theorem \ref{thm:daisyramsey}.

\begin{theorem}\label{thm:daisyramsey}
Let $r \geq 3$ and $k\geq 1$ be integers. There exist integer $m_0=m_0(r,k)$ and absolute positive constant $c$ such that
\begin{align*}
D_r^{\smp}(m,k)\geq t_{r-2}(ck^{2^{4-r}-2}m^{2^{4-r}})
\end{align*}
holds for $m\geq m_0$.
\end{theorem}

Our proof is a variant of the stepping-up lemma of Erd\H{o}s, Hajnal and Rado \cites{EHR65, GRS13}. There are $k+1$ distinct simple $(r,m,k)$-daisies depending on the sizes of $K_0$ and $K_1$. While it is not hard to construct a coloring avoiding a monochromatic copy of one of these simple daisies, the main challenge is to define a coloring that avoids all $k+1$ simple $(r,m,k)$-daisies simultaneously. To this end, we will introduce in Section \ref{sec:auxtrees} some auxiliary trees using the vertices of our ground set. A big portion of the paper consists on the study of those trees and how to use them to obtain a stepping-up lemma.

The paper is organized as follows. We introduce some auxiliary trees and most of the terminology in Section \ref{sec:auxtrees}. Section \ref{sec:coloring} is devoted to give a general overview of the proof. We briefly describe the stepping-up lemma in \cites{EHR65, GRS13} with our setup and later describe the coloring of the variant. Sections \ref{sec:info} and \ref{sec:preprocessing} are the heart of the proof. We prove a key lemma (Lemma \ref{lem:preprocessing}) that allows us to identify an important auxiliary tree containing the petal of an edge and then show how to reduce the stepping-up argument to this tree. We finish the proof of the stepping-up lemma and Theorem \ref{thm:daisyramsey} in Section \ref{sec:proof}.

\section{Auxiliary trees}\label{sec:auxtrees}

Given an integer $N$, we construct a binary tree $T_{[2^N]}$ of height $N$ with $2^{N+1}-1$ vertices and identify its leaves with the set $[2^N]$. We also identify each level of the tree with the set $[N+1]$, where the root is at level $1$, while the leaves are at level $N+1$. For a vertex $u \in T_{[2^N]}$ we denote its level by $\pi(u)$.

\begin{figure}[h]
\centering
{\hfil \begin{tikzpicture}[scale=1.0]
    \coordinate (A) at (0,0);
    \coordinate (B) at (-2,1);
    \coordinate (C) at (2,1);
    \coordinate (D) at (-3,2);
    \coordinate (E) at (-1,2);
	\coordinate (F) at (1,2);
    \coordinate (G) at (3,2);
    \coordinate [label=above:$1$] (H1) at (-3.5,3);
    \coordinate [label=above:$2$] (H2) at (-2.5,3);
    \coordinate [label=above:$3$] (H3) at (-1.5,3);
    \coordinate [label=above:$4$] (H4) at (-0.5,3);
    \coordinate [label=above:$5$] (H5) at (0.5,3);
    \coordinate [label=above:$6$] (H6) at (1.5,3);
    \coordinate [label=above:$7$] (H7) at (2.5,3);
    \coordinate [label=above:$8$] (H8) at (3.5,3);
    
    \coordinate [label=left:$1$] (L1) at (-5,0);
    \coordinate [label=left:$2$] (L2) at (-5,1);
    \coordinate [label=left:$3$] (L3) at (-5,2);
    \coordinate [label=left:$4$] (L4) at (-5,3);
    \coordinate (R1) at (5,0);
    \coordinate (R2) at (5,1);
    \coordinate (R3) at (5,2);
    \coordinate (R4) at (5,3);

    \draw[line width=0.5] (C)--(A)--(B);
    \draw[line width=0.5] (D)--(B)--(E);
    \draw[line width=0.5] (F)--(C)--(G);
    \draw[line width=0.5] (H1)--(D)--(H2);
    \draw[line width=0.5] (H3)--(E)--(H4);
    \draw[line width=0.5] (H5)--(F)--(H6);
    \draw[line width=0.5] (H7)--(G)--(H8);
    \draw[line width=0.5] (L1)--(L2)--(L3)--(L4);
    \draw[dashed] (L1)--(R1);
    \draw[dashed] (L2)--(R2);
    \draw[dashed] (L3)--(R3);
    \draw[dashed] (L4)--(R4);
    \draw (L1) circle [radius=0.05];
 	\draw (L2) circle [radius=0.05];
 	\draw (L3) circle [radius=0.05];
 	\draw (L4) circle [radius=0.05];
 	\draw[fill] (H1) circle [radius=0.05];
 	\draw[fill] (H2) circle [radius=0.05];
 	\draw[fill] (H3) circle [radius=0.05];
 	\draw[fill] (H4) circle [radius=0.05];
 	\draw[fill] (H5) circle [radius=0.05];
 	\draw[fill] (H6) circle [radius=0.05];
 	\draw[fill] (H7) circle [radius=0.05];
 	\draw[fill] (H8) circle [radius=0.05];
 	\draw[fill] (A) circle [radius=0.05];
 	\draw[fill] (B) circle [radius=0.05];
 	\draw[fill] (C) circle [radius=0.05];
 	\draw[fill] (D) circle [radius=0.05];
 	\draw[fill] (E) circle [radius=0.05];
 	\draw[fill] (F) circle [radius=0.05];
 	\draw[fill] (G) circle [radius=0.05];
    
    \end{tikzpicture}\hfil}
\caption{An example of a binary tree $T_{[2^3]}$ with its $4$ levels}
\label{fig:binarytree}
\end{figure}
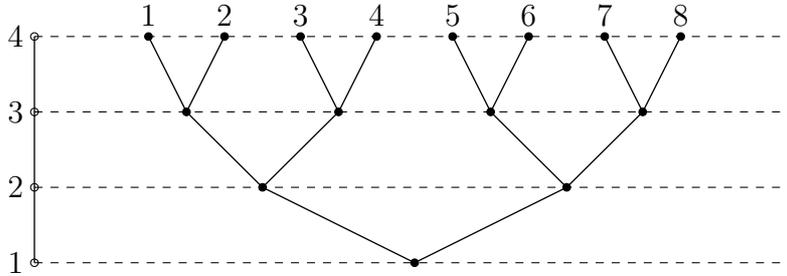

Given two vertices $u,v$ in $T_{[2^N]}$, we say that $u$ is an \textit{ancestor} of $v$ if $\pi(u)<\pi(v)$ and there is a path $u=x_1,x_2,\ldots,x_{\ell}=v$ in $T_{[2^N]}$ such that $\pi(x_i)\neq \pi(x_j)$ for every $1\leq i,j \leq \ell$. For two vertices $x,y \in [2^N]$ we define the $\textit{greatest common ancestor}$ $a(x,y)$ of $x$ and $y$ as the vertex of $T_{[2^N]}$ of highest level that is an ancestor of both $x$ and $y$. Also define
\begin{align*}
    \delta(x,y)=\pi(a(x,y)).
\end{align*}

Let $X=\{x_1,\ldots,x_t\} \subseteq [2^N]$ with $x_1<\ldots<x_t$ be a subset of the leaves of our binary tree. We define the \textit{auxiliary tree} $T_X$ of $X$ as the subtree of $T_{[2^N]}$ whose vertices are $X$ and all their common ancestors. That is,
\begin{align*}
    T_X=X \cup \{a(x_i,x_{i+1}):\:1\leq i \leq t-1\}.
\end{align*}

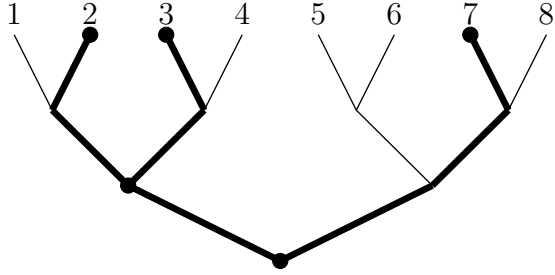
\begin{figure}[h]
\centering
{\hfil \begin{tikzpicture}[scale=1.0]
    \coordinate (A) at (0,0);
    \coordinate (B) at (-2,1);
    \coordinate (C) at (2,1);
    \coordinate (D) at (-3,2);
    \coordinate (E) at (-1,2);
	\coordinate (F) at (1,2);
    \coordinate (G) at (3,2);
    \coordinate [label=above:$1$] (H1) at (-3.5,3);
    \coordinate [label=above:$2$] (H2) at (-2.5,3);
    \coordinate [label=above:$3$] (H3) at (-1.5,3);
    \coordinate [label=above:$4$] (H4) at (-0.5,3);
    \coordinate [label=above:$5$] (H5) at (0.5,3);
    \coordinate [label=above:$6$] (H6) at (1.5,3);
    \coordinate [label=above:$7$] (H7) at (2.5,3);
    \coordinate [label=above:$8$] (H8) at (3.5,3);

    \draw[line width=2.5] (C)--(A)--(B);
    \draw[line width=2.5] (D)--(B)--(E);
    \draw[line width=0.5] (F)--(C);
    \draw[line width=2.5] (G)--(C);
    \draw[line width=0.5] (H1)--(D);
    \draw[line width=2.5] (H2)--(D);
    \draw[line width=2.5] (H3)--(E);
    \draw[line width=0.5] (H4)--(E);
    \draw[line width=0.5] (H5)--(F)--(H6);
    \draw[line width=2.5] (H7)--(G);
    \draw[line width=0.5] (H8)--(G);
 	\draw[fill] (H2) circle [radius=0.1];
 	\draw[fill] (H3) circle [radius=0.1];
 	\draw[fill] (H7) circle [radius=0.1];
 	\draw[fill] (B) circle [radius=0.1];
 	\draw[fill] (A) circle [radius=0.1];
    
    \end{tikzpicture}\hfil}
\caption{The auxiliary tree $T_X$ for $X=\{2,3,7\}$.}
\label{fig:auxtrees}
\end{figure}

Note that $T_X$ is a tree of $2t-1$ vertices. Moreover, we denote the set of non-leaves by $a(X)$ and its projection by $\delta(X)$, i.e.,
\begin{align*}
a(X)&=\{a(x_i,x_{i+1}):\:1\leq i \leq t-1\}  \\
\delta(X)&=\{\delta(x_i,x_{i+1}):\: 1\leq i \leq t-1\}.
\end{align*}
Since the auxiliary tree $T_X$ is uniquely determined by its ground set $X$, sometimes we will denote $T_X$ by $X$.

Given a vertex $u \in a(X)$, we can define the set $X(u)$ of \textit{descendants} of $u$ as the leaves of $T_X$ that have $u$ as an ancestor. That is,
\begin{align*}
    X(u)=\{x \in X:\: \text{$u$ is an ancestor of $x$}\}.
\end{align*}
The set of descendants of $u$ can be partitioned into the left descendants and right descendants as follows: Since $T_X$ is a binary tree, the vertex $u$ has two children $u^L$ and $u^R$. Let $u^L$ be the left children of $u$ and $u^R$ be the right children of $u$. Then we define the left descendants of $u$ by
\begin{align*}
    X_L(u)=\begin{cases}
        u^L &\quad \text{if $u^L \in X$},\\
        X(u^L) &\quad \text{if $u^L\in a(X)$},
    \end{cases}
\end{align*}
and the right descendants of $u$ by
\begin{align*}
   X_R(u)=\begin{cases}
        u^R &\quad \text{if $u^R \in X$},\\
        X(u^R) &\quad \text{if $u^R\in a(X)$},
    \end{cases}
\end{align*}
Note that by this definition $X_L(u), X_R(u)\neq \emptyset$ and $\max X_L(u)<\min X_R(u)$.

Although an auxiliary tree is not uniquely determined by its ancestors, we can at least determine the ``shape" of the tree $T_X$ by looking at $a(X)$. In a more precise way, the following can be proved by a simple induction.

\begin{fact}\label{fact:shape}
If $X$ and $Y$ are subsets of $[2^N]$ such that $a(X)=a(Y)$, then $|X|=|Y|$. Moreover, if $X=\{x_1,\ldots,x_t\}$ and $Y=\{y_1,\ldots,y_t\}$, then $a(x_i,x_{i+1})=a(y_i,y_{i+1})$ for every $1\leq i \leq t-1$.
\end{fact}

Now we devote the rest of the section on classifying our auxiliary trees. 

\begin{definition}\label{def:closed}
Given $X=\{x_1,\ldots,x_t\}\subseteq [2^N]$. We say that an interval $I=\{x_p,\ldots,x_q\} \subseteq X$  for some $1\leq p\leq q \leq t$ is \textit{closed} in $X$ if the following condition holds:

\begin{enumerate}
    \item[($\star$)] $I=X(a(x_p,x_q))$.
\end{enumerate}
\end{definition}

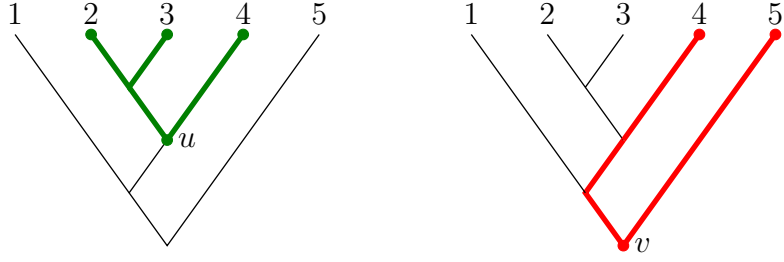
\begin{figure}[h]
\centering
{\hfil \begin{tikzpicture}[scale=1.0]
    \coordinate (A) at (-3,0);
    \coordinate (B) at (-3.5,0.7);
    \coordinate [label=right:$u$] (C) at (-3,1.4);
    \coordinate (D) at (-3.5,2.1);
    \coordinate [label=above:$1$] (E1) at (-5,2.8);
    \coordinate [label=above:$2$] (E2) at (-4,2.8);
    \coordinate [label=above:$3$] (E3) at (-3,2.8);
    \coordinate [label=above:$4$] (E4) at (-2,2.8);
    \coordinate [label=above:$5$] (E5) at (-1,2.8);
    
    \coordinate [label=right:$v$] (X) at (3,0);
    \coordinate (Y) at (2.5,0.7);
    \coordinate (Z) at (3,1.4);
    \coordinate (W) at (2.5,2.1);
    \coordinate [label=above:$1$] (U1) at (1,2.8);
    \coordinate [label=above:$2$] (U2) at (2,2.8);
    \coordinate [label=above:$3$] (U3) at (3,2.8);
    \coordinate [label=above:$4$] (U4) at (4,2.8);
    \coordinate [label=above:$5$] (U5) at (5,2.8);

    \draw[line width=0.5] (E5)--(A)--(B)--(C);
    \draw[line width=0.5] (B)--(E1);
    \draw[line width=2][green1] (D)--(C)--(E4);
    \draw[line width=2][green1] (E3)--(D)--(E2);
    
    \draw[line width=2][red] (U5)--(X)--(Y)--(Z);
    \draw[line width=0.5] (Z)--(W);
    \draw[line width=0.5] (Y)--(U1);
    \draw[line width=2][red] (Z)--(U4);
    \draw[line width=0.5] (U3)--(W)--(U2);
    
    \draw[fill][green1] (C) circle [radius=0.07];
 	\draw[fill][green1] (E2) circle [radius=0.07];
 	\draw[fill][green1] (E3) circle [radius=0.07];
 	\draw[fill][green1] (E4) circle [radius=0.07];
 	
 	\draw[fill][red] (X) circle [radius=0.07];
 	\draw[fill][red] (U4) circle [radius=0.07];
 	\draw[fill][red] (U5) circle [radius=0.07];
    
    \end{tikzpicture}\hfil}
\caption{The interval $\{2,3,4\}$ is closed, since $X(u)=\{2,3,4\}$ for $u=a(2,4)$. The interval $\{4,5\}$ is not closed, since $X(v)=\{1,2,3,4,5\}\neq \{4,5\}$ for $v=a(4,5)$.}
\label{fig:closed}
\end{figure}

Alternatively, one can replace ($\star$) by the useful equivalent condition:

\begin{enumerate}
    \item[($\star\star$)] For every vertex $y\in X\setminus I$, the vertex $a(x_p,x_q)$ is not an ancestor of $y$.
\end{enumerate}

The following proposition shows that closed intervals cannot have proper intersections.

\begin{proposition}\label{prop:closed}
Let $I_1$, $I_2$ be two intervals in $X$ with $|I_1|\leq |I_2|$. If $I_1$ and $I_2$ are closed, then either $I_1 \cap I_2=\emptyset$ or $I_1\subseteq I_2$.
\end{proposition}

\begin{proof}
Suppose that $I_1\cap I_2$ is a proper intersection. That is, $I_1\cap I_2 \neq \emptyset$, $I_1\setminus I_2\neq \emptyset$ and $I_2\setminus I_1\neq \emptyset$. Write $X=\{x_1,\ldots,x_t\}$ and $I_1=\{x_{p_1},x_{p_1+1},\ldots,x_{q_1}\}$, $I_2=\{x_{p_2},x_{p_2+1},\ldots,x_{q_2}\}$ for $1\leq p_1<p_2\leq q_1<q_2\leq t$. Let $u=a(x_{p_1},x_{q_1})$ and $v=a(x_{p_2},x_{q_2})$. We claim that either $u$ is an ancestor of $v$ or $v$ is an ancestor of $u$. Let $z\in I_1\cap I_2$. By definition, both $u$ and $v$ are ancestors of $z$. This means that there exists descending paths connecting $z$ to $u$ and $z$ to $v$ in $T_X$ with vertices in different levels. However, every vertex in $T_X$ has at most one father. Therefore, either the path $z$ to $u$ contains the path $z$ to $v$ or vice-versa. If the path $z$ to $u$ contains the path $z$ to $v$, then $u$ is an ancestor of $v$. Hence $u$ is an ancestor of $I_2\setminus I_1$, which contradicts the fact that $I_1$ is closed (Condition ($\star\star$) of Definition \ref{def:closed}). The other case is analogous.  
\end{proof}

We classify the closed intervals of $X$ by three classes: left combs, right combs and broken combs.

\begin{definition}\label{def:comb}
Given a closed interval $I$ in $X$ we say that
\begin{enumerate}
    \item[(a)] $I$ is a \textit{$\ell$-left comb} if $\ell$ is the least positive integer such that there exists a partition $I=A\cup B$ with $|A|=\ell$ and $B\neq \emptyset$ and
    \begin{enumerate}
        \item[(a1)] $A<B$.
        \item[(a2)] $A$ is a closed interval in $X$
        \item[(a3)] If $z=\max(A)$ and $B=\{b_1,\ldots,b_s\}$, then $\delta(z,b_1)>\delta(b_1,b_2)>\ldots>\delta(b_{s-1},b_s)$.
        \end{enumerate}
        
    \item[(b)] $I$ is a \textit{$\ell$-right comb} if $\ell$ is the least positive integer such that there exists a partition $I=A\cup B$ with $|A|=\ell$ and $B\neq \emptyset$ and
    \begin{enumerate}
        \item[(b1)] $B<A$.
        \item[(b2)] $A$ is a closed interval in $X$
        \item[(b3)] If $z=\min(A)$ and $B=\{b_1,\ldots,b_s\}$, then $\delta(b_1,b_2)<\ldots<\delta(b_{s-1},b_s)<\delta(b_t,z)$.
    \end{enumerate}
    
    \item[(c)] $I$ is a \textit{broken comb} if it is neither a left or right comb.
\end{enumerate}
\end{definition}

\begin{figure}[h]
\centering
{\hfil \begin{tikzpicture}[scale=0.7]
    
    \draw[red] (-9,4.5) ellipse (1 and 0.3);
   
    \coordinate (A) at (-10,4.5);
    \coordinate (B) at (-8,4.5);
    \coordinate (C) at (-9,3);
    \coordinate (D1) at (-8.66,2.5);
    \coordinate (D2) at (-8.33,2);
    \coordinate (D3) at (-8,1.5);
    \coordinate (D4) at (-7.66,1);
    \coordinate (D5) at (-7.33,0.5);
    \coordinate (D6) at (-7,0);
    \coordinate (E1) at (-7.33,4.5);
    \coordinate (E2) at (-6.66,4.5);
    \coordinate (E3) at (-6,4.5);
    \coordinate (E4) at (-5.33,4.5);
    \coordinate (E5) at (-4.66,4.5);
    \coordinate (E6) at (-4,4.5);
    \coordinate [label=above:$A$] (o1) at (-9,5);
    \coordinate [label=above:$B$] (o2) at (-5.66,5.4);
    
    \draw [decorate,
    decoration = {brace, amplitude=8pt}] (-7.2,4.9) -- (-4.1,4.9);
    
    \draw[line width=0.5][red] (D6)--(C)--(A);
    \draw[line width=0.5][red] (B)--(C);
    \draw[line width=0.5][red] (D1)--(E1);
    \draw[line width=0.5][red] (D2)--(E2);
    \draw[line width=0.5][red] (D3)--(E3);
    \draw[line width=0.5][red] (D4)--(E4);
    \draw[line width=0.5][red] (D5)--(E5);
    \draw[line width=0.5][red] (D6)--(E6);
    
 	\draw[blue] (2,4.5) ellipse (1 and 0.3);
 	
 	\coordinate (F) at (3,4.5);
 	\coordinate (G) at (1,4.5);
    \coordinate (H) at (2,3);
    \coordinate (M1) at (1.66,2.5);
    \coordinate (M2) at (1.33,2);
    \coordinate (M3) at (1,1.5);
    \coordinate (M4) at (0.66,1);
    \coordinate (M5) at (0.33,0.5);
    \coordinate (M6) at (0,0);
    \coordinate (N1) at (0.33,4.5);
    \coordinate (N2) at (-0.33,4.5);
    \coordinate (N3) at (-1,4.5);
    \coordinate (N4) at (-1.66,4.5);
    \coordinate (N5) at (-2.33,4.5);
    \coordinate (N6) at (-3,4.5);
    \coordinate [label=above:$A$] (o3) at (2,5);
    \coordinate [label=above:$B$] (o4) at (-1.33,5.4);
    
    \draw [decorate,
    decoration = {brace, amplitude=8pt}] (-2.9,4.9) --  (0.2,4.9);
    
    \draw[line width=0.5][blue] (M6)--(H)--(F);
    \draw[line width=0.5][blue] (H)--(G);
    \draw[line width=0.5][blue] (M1)--(N1);
    \draw[line width=0.5][blue] (M2)--(N2);
    \draw[line width=0.5][blue] (M3)--(N3);
    \draw[line width=0.5][blue] (M4)--(N4);
    \draw[line width=0.5][blue] (M5)--(N5);
    \draw[line width=0.5][blue] (M6)--(N6);
 	
 	
 	\draw[green1] (9,4.5) ellipse (1 and 0.3);
 	\draw[green1] (5.8,4.5) ellipse (1.8 and 0.3);
 	
 	\coordinate (X) at (4,4.5);
 	\coordinate (Y) at (7.6,4.5);
    \coordinate (Z) at (5.8,1.8);
    \coordinate (P) at (8,4.5);
 	\coordinate (Q) at (10,4.5);
    \coordinate (R) at (9,3);
    \coordinate (U) at (7,0);
    \coordinate [label=above:$A$] (o5) at (7,5.4);
    
    \draw [decorate,
    decoration = {brace, amplitude=8pt}] (4.2,5) --  (9.8,5);
 	
 	\draw[line width=0.5][green1] (X)--(Z)--(Y);
    \draw[line width=0.5][green1] (P)--(R)--(Q);
    \draw[line width=0.5][green1] (R)--(U)--(Z);

    \end{tikzpicture}\hfil}
    \caption{An example of a left, right and broken comb, respectively.}
    \label{fig:combs}
\end{figure}
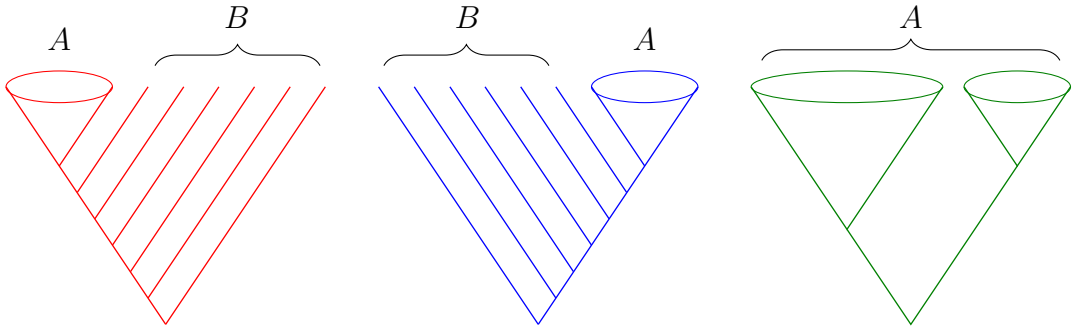

We will use the convention that an $\ell$-left/right comb will be described by its partition $I=A\cup B$ with $|A|=\ell$ that verifies the condition on Definition \ref{def:comb}. As we can see in the picture above, the set $A$ should be thought as the ``handle" of the comb, while the set $B$ should be thought as the ``teeth" of the comb. For broken combs we will adopt the same convention by assuming that $B= \emptyset$.

One may remove the use of the projection $\delta(b_i,b_{i+1})$ in conditions (a3) and (b3) of the right/left comb by using the following equivalent alternative conditions:
\begin{enumerate}
    \item[(a3*)] If $B=\{b_1,\ldots,b_s\}$, then the intervals $A\cup\{b_1,\ldots,b_i\}$ are closed in $X$ for every $1\leq i \leq s$
    \item[(b3*)] If $B=\{b_1,\ldots,b_s\}$, then the intervals $\{b_i,\ldots,b_s\}\cup A$ are closed in $X$ for every $1\leq i \leq s$.
\end{enumerate}
Those conditions have the advantage of describing a comb only using closed intervals. This will be useful later in the proof.

\begin{example}\label{ex:1comb}
A important type of comb in the stepping-up lemma \cites{EHR65, GRS13} is the $1$-left/right comb. Those are the combs $I=\{y_1,\ldots,y_t\}$ satisfying that the sequence $\{\delta(y_i,y_{i+1})\}_{1\leq i \leq t}$ is monotone. Indeed, the interval $I$ is a $1$-left comb if $\delta(y_1,y_2)>\ldots>\delta(y_{t-1},y_t)$, while it is a $1$-right comb if $\delta(y_1,y_2)<\ldots<\delta(y_{t-1},y_t)$.
\end{example}

\begin{figure}[h]
\centering
{\hfil \begin{tikzpicture}[scale=0.7]
    \coordinate (R1) at (-3,0);
    \coordinate (R2) at (-3.5,1);
    \coordinate (R3) at (-4,2);
    \coordinate (R4) at (-4.5,3);
    \coordinate (R5) at (-5,4);
	\coordinate (R6) at (-4,4);
    \coordinate (R7) at (-3,4);
    \coordinate (R8) at (-2,4);
    \coordinate (R9) at (-1,4);
    \coordinate (H1) at (3,0);
    \coordinate (H2) at (3.5,1);
    \coordinate (H3) at (4,2);
    \coordinate (H4) at (4.5,3);
    \coordinate (H5) at (5,4);
    \coordinate (H6) at (4,4);
    \coordinate (H7) at (3,4);
    \coordinate (H8) at (2,4);
    \coordinate (H9) at (1,4);

    \draw[line width=1] (R1)--(R2)--(R3)--(R4)--(R5);
    \draw[line width=1] (R1)--(R9);
    \draw[line width=1] (R2)--(R8);
    \draw[line width=1] (R3)--(R7);
    \draw[line width=1] (R4)--(R6);
    
    \draw[line width=1] (H1)--(H2)--(H3)--(H4)--(H5);
    \draw[line width=1] (H1)--(H9);
    \draw[line width=1] (H2)--(H8);
    \draw[line width=1] (H3)--(H7);
    \draw[line width=1] (H4)--(H6);

    \end{tikzpicture}\hfil}
\caption{A left and right $1$-comb}
\label{fig:1comb}
\end{figure}
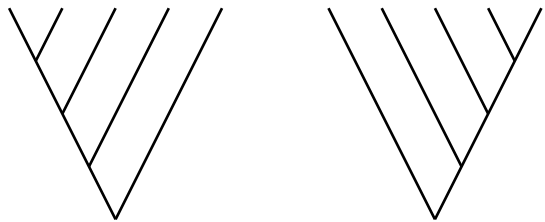

For the proof of Theorem \ref{thm:daisyramsey} we will be interested in maximal comb structures inside our auxiliary trees.

\begin{definition}\label{def:maximalcomb}
Given $X=\{x_1,\ldots,x_t\} \subseteq [2^N]$, a interval $I=\{x_p,\ldots,x_q\}$ is a
\begin{enumerate}
    \item[(a)] \textit{Maximal left comb} in $X$ if $I$ is a left comb and $I\cup\{x_{q+1}\}$ is not a closed interval in $X$.
    \item[(b)] \textit{Maximal right comb} in $X$ if $I$ is a right comb and $I\cup\{x_{p-1}\}$ is not a closed interval in $X$.
    \item[(c)] \textit{Maximal broken comb} in $X$ if $I$ is a broken comb and neither $I\cup\{x_{p-1}\}$ or $I\cup\{x_{q+1}\}$ are closed.
\end{enumerate}
\end{definition}

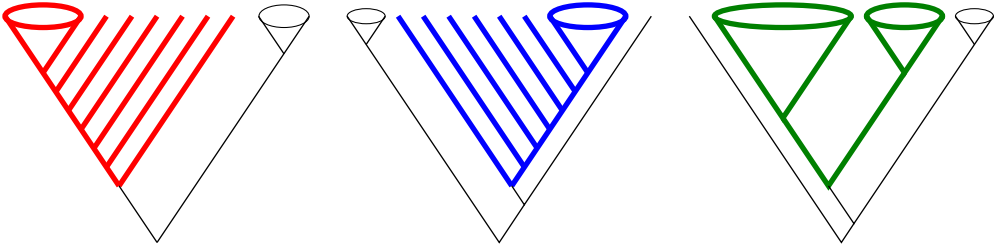
\begin{figure}[h]
\centering
{\hfil \begin{tikzpicture}[scale=0.5]
    
    \draw[red][line width=2] (-12,6) ellipse (1 and 0.3);
    \draw (-5.66,6) ellipse (0.66 and 0.3);
   
    \coordinate (A) at (-13,6);
    \coordinate (B) at (-11,6);
    \coordinate (C) at (-12,4.5);
    \coordinate (D1) at (-11.66,4);
    \coordinate (D2) at (-11.33,3.5);
    \coordinate (D3) at (-11,3);
    \coordinate (D4) at (-10.66,2.5);
    \coordinate (D5) at (-10.33,2);
    \coordinate (D6) at (-10,1.5);
    \coordinate (D7) at (-9,0);
    \coordinate (E1) at (-10.33,6);
    \coordinate (E2) at (-9.66,6);
    \coordinate (E3) at (-9,6);
    \coordinate (E4) at (-8.33,6);
    \coordinate (E5) at (-7.66,6);
    \coordinate (E6) at (-7,6);
    \coordinate (E7) at (-6.33,6);
    \coordinate (E8) at (-5,6);
    \coordinate (C2) at (-5.66,5);
    
    \draw[line width=2][red] (D6)--(C)--(A);
    \draw[line width=2][red] (B)--(C);
    \draw[line width=2][red] (D1)--(E1);
    \draw[line width=2][red] (D2)--(E2);
    \draw[line width=2][red] (D3)--(E3);
    \draw[line width=2][red] (D4)--(E4);
    \draw[line width=2][red] (D5)--(E5);
    \draw[line width=2][red] (D6)--(E6);
    \draw[line width=0.5] (D6)--(D7);
    \draw[line width=0.5] (D7)--(E8);
    \draw[line width=0.5] (C2)--(E7);
    
 	\draw[line width=2][blue] (2.33,6) ellipse (1 and 0.3);
 	\draw (-3.5,6) ellipse (0.5 and 0.2);
 	
 	\coordinate (F) at (3.33,6);
 	\coordinate (G) at (1.33,6);
    \coordinate (H) at (2.33,4.5);
    \coordinate (M1) at (2,4);
    \coordinate (M2) at (1.66,3.5);
    \coordinate (M3) at (1.33,3);
    \coordinate (M4) at (1,2.5);
    \coordinate (M5) at (0.66,2);
    \coordinate (M6) at (0.33,1.5);
    \coordinate (M7) at (0.66,1);
    \coordinate (M8) at (0,0);
    \coordinate (N1) at (0.66,6);
    \coordinate (N2) at (0,6);
    \coordinate (N3) at (-0.66,6);
    \coordinate (N4) at (-1.33,6);
    \coordinate (N5) at (-2,6);
    \coordinate (N6) at (-2.66,6);
    \coordinate (N7) at (-3,6);
    \coordinate (N8) at (-4,6);
    \coordinate (N9) at (4,6);
    \coordinate (H2) at (-3.5,5.25);
    
    \draw[line width=2][blue] (M6)--(H)--(F);
    \draw[line width=2][blue] (H)--(G);
    \draw[line width=2][blue] (M1)--(N1);
    \draw[line width=2][blue] (M2)--(N2);
    \draw[line width=2][blue] (M3)--(N3);
    \draw[line width=2][blue] (M4)--(N4);
    \draw[line width=2][blue] (M5)--(N5);
    \draw[line width=2][blue] (M6)--(N6);
    \draw[line width=0.5] (M7)--(M8)--(N8);
    \draw[line width=0.5] (M6)--(M7)--(N9);
    \draw[line width=0.5] (H2)--(N7);
 	
 	
 	\draw[line width=2][green1] (10.66,6) ellipse (1 and 0.3);
 	\draw[line width=2][green1] (7.46,6) ellipse (1.8 and 0.3);
 	\draw (12.5,6) ellipse (0.5 and 0.2);
 	
 	\coordinate (L) at (5,6);
 	\coordinate (X) at (5.66,6);
 	\coordinate (Y) at (9.26,6);
    \coordinate (Z) at (7.46,3.3);
    \coordinate (P) at (9.66,6);
 	\coordinate (Q) at (11.66,6);
    \coordinate (R) at (10.66,4.5);
    \coordinate (U) at (8.66,1.5);
    \coordinate (O) at (9,0);
    \coordinate (V1) at (12,6);
    \coordinate (V2) at (13,6);
    \coordinate (R1) at (9.33,0.5);
    \coordinate (R2) at (12.5,5.25);
 	
 	\draw[line width=2][green1] (X)--(Z)--(Y);
    \draw[line width=2][green1] (P)--(R)--(Q);
    \draw[line width=2][green1] (R)--(U)--(Z);
    \draw[line width=0.5] (L)--(O)--(R1);
    \draw[line width=0.5] (U)--(R1)--(V2);
    \draw[line width=0.5] (R2)--(V1);

    \end{tikzpicture}\hfil}
    \caption{An example of a maximal left, right and broken comb, respectively.}
    \label{fig:maxcombs}
\end{figure}

The next proposition shows that given two maximal combs they are either disjoint or one is contained in the ``handle" of the other.

\begin{proposition}\label{prop:maximalcomb}
Given a closed inteval $I_1$ and a maximal comb $I_2=A_2\cup B_2$ with $|I_1|\leq |I_2|$ in a set $X\subseteq [2^N]$, then one of the following holds:
\begin{enumerate}
    \item $I_1\cap I_2=\emptyset$
    \item $I_1\subseteq A_2$.
    \item $I_1=A_2\cup B_1$ for some initial segment $B_1\subseteq B_2$
\end{enumerate}
Moreover, condition (3) only holds if $I_1$ is not a maximal comb.
\end{proposition}

\begin{proof}
By Proposition \ref{prop:closed} we obtain that either $I_1\cap I_2=\emptyset$ or $I_1\subseteq I_2$. If the first case happens, then $I_1$ and $I_2$ satisfy condition (1) and we are done. Hence, we may assume that $I_1\subseteq I_2$. If $I_2$ is a broken comb, then by definition $A_2=I_2$. Thus in this case $I_1\subseteq A_2$, satsifying condition (2). Now suppose without loss of generality that $I_2=A_2\cup B_2$ is a left maximal comb and write $A_2=\{x_1,\ldots, x_{\ell}\}$, $B_2=\{y_1,\ldots, y_s\}$. If $I_1\cap B_2=\emptyset$, then $I_1\subseteq A_2$ and again condition (2) holds.

At last, it remains to deal with the case that $I_1\cap B_2\neq \emptyset$. Since $I_1$ is an interval of $X$ and $I_1\subseteq I_2$, then in particular $I_1$ is an interval of $I_2$. Write $I_1=\{x_p,\ldots,x_\ell\}\cup \{y_1,\ldots, y_q\}$ for $1\leq p \leq \ell$ and $1\leq q \leq s$. By condition (a3*) of Definition \ref{def:comb}, the set $A_2\cup\{y_1,\ldots,y_{q-1}\}$ is closed. Therefore for any $z \in A_2\cup\{y_1,\ldots,y_{q-1}\}$ the greatest ancestor $a(z,y_q)$ of $z$ and $y_q$ is the same as the greatest ancestor of $a(x_1,y_q)$. In particular, this implies that $a(x_p,y_q)$ is an ancestor for the entire set $A_2$. Hence $A_2\subseteq I_1$ and consequently $I_1=A_1\cup B_1$ is a left comb satisfying condition (3), because $A_1=A_2$ and $B_1$ is an initial segment of $B_2$. Note that $I_1$ is not maximal in this case, since the set $I_1\cup\{y_{q+1}\}$ is also a left comb. Thus if $I_1$ is a maximal comb, then it either satisfies (1) or (2).
\end{proof}

\section{Stepping-up lemma and our coloring}\label{sec:coloring}

\subsection{Erd\H{o}s--Hajnal--Rado stepping-up lemma}\label{subsec:ehrstep}

For instructional purposes, we will briefly go over the stepping-up lemma in \cites{EHR65, GRS13} using our notation. For $k\geq 4$, let $N=R_{k-1}((n-k+4)/2)-1$ and $\phi: [N]^{(k-1)} \rightarrow \{0,1\}$ be a coloring of the $(k-1)$-tuples in $[N]$ with no monochromatic subset of size $(n-k+4)/2$. Our goal is to find a coloring $\psi: [2^N]^{(k)}\rightarrow \{0,1\}$ with no monochromatic subset of size $n$. This will give us that $R_k(n)> 2^N=2^{R_{k-1}((n-k+4)/2)-1}$. 

Fix an edge $X=\{x_1,\ldots,x_k\} \in [2^N]^{(k)}$ and let $\delta_i=\delta(x_i,x_{i+1})$. We describe the coloring $\psi$ by the structure of $T_X$ and the coloring of the vertical projection $\phi$ of $[N]$ in the following way
\begin{align*}
    \psi(X)=\begin{cases}
    0,& \quad \text{if }\delta_{k-3}>\delta_{k-2}<\delta_{k-1}\\
    1,& \quad \text{if }\delta_{k-3}<\delta_{k-2}>\delta_{k-1}\\
    \phi(\{\delta_1,\ldots,\delta_{k-1}\}),& \quad \text{otherwise if $|\delta(X)|=k-1$}\\
    0,& \quad \text{otherwise if $|\delta(X)|<k-1$.}
    \end{cases}
\end{align*}

Suppose by contradiction that $\psi$ contains a monochromatic subset $Y\subset [2^N]$ of size $n$. We can use the structure of $Y$ to find a large $1$-comb.

\begin{proposition}\label{prop:findcomb}
There exists an interval $I$ of $Y$ with $|I|\geq (n-k+6)/2$ such that $I$ is a $1$-comb
\end{proposition}

\begin{proof}
We may assume without loss of generality that $Y$ is monochromatic of color $0$. Write $Y=\{y_1,\ldots,y_n\}$ and let $\delta^Y_i=\delta(y_i,y_{i+1})$ for $1\leq i \leq n-1$. Since all edges in $Y$ are of color $0$, then for any edge $X=\{x_1,\ldots,x_k\}\in Y^{(k)}$ we do not have that
\begin{align}\label{eq:stepup}
    \delta(x_{k-3},x_{k-2})<\delta(x_{k-2},x_{k-1})>\delta(x_{k-1},x_k).
\end{align}
In particular, by taking the edge $\{y_{\ell-k+1},\ldots,y_\ell\}$, inequality (\ref{eq:stepup}) implies that $\delta^Y_{\ell-3}<\delta^Y_{\ell-2}>\delta^Y_{\ell-1}$ does not hold fore every $k\leq \ell \leq n$. Hence, the sequence $\{\delta^Y_i\}_{i=k-3}^{n-1}$ has no local maximum.

A standard calculus argument says that between two local minimums there is always a local maximum. Therefore, the sequence $\{\delta^Y_i\}_{i=k-3}^{n-1}$ has at most one local minimum, which means that there exists an interval $[p,q]$ of size $(n-k+4)/2$ such that $\{\delta^Y_i\}_{i\in [p,q]}$ is monotone. Thus by definition the interval $I=\{x_p,x_{p+1},\ldots,x_{q+1}\}$ is a $1$-comb of size $(n-k+6)/2$.
\end{proof}

Let $I=\{z_1,\ldots,z_{t}\}$ be the $1$-comb of $Y$ obtained by Proposition \ref{prop:findcomb} and denote $\delta_i^I=\delta(z_i,z_{i+1})$ for $1\leq i \leq i-1$. Note that because $\{\delta_i^I\}_{i=1}^{t-1}$ is a monotone sequence, every edge $X\in I^{(k)}$ will be also a $1$-comb. Moreover, for every $(k-1)$-tuple $Z\in \delta(I)$ there exists an edge $X\in I^{(k)}$ such that $\delta(X)=Z$.

Finally, by the definition of the coloring $\psi$, if $X$ is a $1$-comb, then $\psi(X)=\phi(\delta(X))$. Thus if $I^{(k)}$ colored by $\psi$ is monochromatic, then $\delta(I)^{(k-1)}$ colored by $\phi$ is also monochromatic. This implies that $[N]$ has a monochromatic set of size $(n-k+4)/2$, which contradicts our assumption on $\phi$.

\subsection{Overview of the proof}\label{subsec:overview}

In order to obtain a lower bound for simple daisies, we will define a variant of the stepping-up lemma described in the previous subsection. Suppose for a moment that our goal is to avoid a monochromatic simple $(r,m,k)$-daisy with in $[2^N]$ with $|K_0|=k_0$ and $|K_1|=k_1$ fixed. Then for every edge $X=\{x_1,\ldots,x_{k+r}\}$ of the daisy, we know that the petal of size $r$ of $X$ is $P=\{x_{k_0+1},\ldots,x_{k_0+r}\}$. That is, we know the exact location of the petal prior defining the coloring in our stepping-up lemma. In this case a natural way to define the coloring would be to just assign for every edge $X$ with petal $P\subseteq X$ the color $\chi(X)=\psi(P)$, where $\psi(P)$ is exactly the stepping-up coloring defined in the previous subsection. Since the petal is the only part of the edge changing when we run through all edges, a similar proof as in the previous subsection works.

Unfortunately, in the original problem we want to avoid all possible monochromatic simple $(r,m,k)$-daisies, which means that we need to avoid simple daisies for all the values of $|K_0|$ and $|K_1|$. The obstruction now is that the location of the petal within the edge is no longer clear to us. To fix that we are going to pre-process our potentially monochromatic simple daisy (Lemma \ref{lem:preprocessing}) to satisfy the following property: Every petal $P$ of an edge $X$ is either a closed interval in $X$ or is in the ``teeth" of a maximal comb in $X$. This gives us partial information about the location of the petal. A good strategy then is to define an auxiliary coloring $\chi_0$ for every maximal comb in $X$ and use those colorings to define a coloring for $X$. This is the content of Section \ref{subsec:variant}.

Some technical challenges remain. By Proposition \ref{prop:maximalcomb}, the maximal combs in $X$ do not need to be disjoint. Therefore, it might happen that for the coloring $\chi$ different maximal combs interfere with each other. To solve that we need to construct a careful coloring taking the issue into consideration. In Section \ref{sec:info} we provide an analysis showing that distinct maximal combs do not interfere with each other in our coloring. Section \ref{sec:preprocessing} is devoted to the pre-processing described in the last paragraph. One of the consequences of the section is that for an edge $X$ the coloring $\chi(X)$ is essentially determined by a unique maximal comb inside of it. Finally, we finish the proof in Section \ref{sec:proof}, by showing, similarly as in Subsection \ref{subsec:ehrstep}, that a monochromatic simple daisy in $[2^N]$ corresponds to a monochromatic simple daisy in the vertical coloring of $[N]$.

\subsection{A variant of the stepping-up lemma}\label{subsec:variant}

Let $N=\min_{0\leq t \leq k-1}\left\{D_{r-1}^{\smp}(c_k\sqrt{m},t)-1\right\}$ for $r\geq 4$ and $c_k$ some constant depending on $k$ to be defined later and let
$\{\phi_i\}_{r-1\leq i \leq k+r-1}$ be a family of colorings such that $\phi_i:[N]^{(i)}\rightarrow\{0,1\}$ is a $2$-coloring of the $i$-tuples without a monochromatic simple $(r-1,c_k\sqrt{m},i-r+1)$-daisy. Note that by the choice of $N$ is always possible to find such a family.

Given an $(k+r)$-tuple $X \in [2^N]^{(k+r)}$ we define 
\begin{align*}
    \cI_X=\{I\subseteq X:\: \text{$I$ is a maximal comb in $X$}\}
\end{align*}
as the set of maximal combs of $X$. We will construct now an auxiliary coloring $\chi_0: \cI_X \rightarrow \{0,1\}$ depending on the structure of $T_X$ and in the family of colorings $\{\phi_t\}_{0\leq t \leq k}$. The coloring is divided in several cases depending on the type of the maximal comb $I$.

Remember that a maximal $\ell$-comb is always identified with the partition $I=A\cup B$, where $|A|=\ell$ is the handle and $B$ is the set of teeth of the comb. Also write $I=\{x_1,\ldots,x_s\}$ and let $\delta_{i}^I=\delta(x_i,x_{i+1})$ for $1\leq i \leq s-1$. Aiming to simplify the discussion, we will only describe $\chi_0$ for left and broken maximal combs. We define $\chi_0$ for right combs by symmetry.

\noindent \underline{Type 1:} $I$ is broken or left comb, $|I|=r$ and there is no maximal comb $I'=A'\cup B'$ such that $I=A'$.
\begin{align*}
    \chi_0(I)=\begin{cases}
    0,& \quad \text{if }\delta^I_{r-3}>\delta^I_{r-2}<\delta^I_{r-1}\\
    1,& \quad \text{if }\delta^I_{r-3}<\delta^I_{r-2}>\delta^I_{r-1}\\
    \phi_{|\delta(I)|}(\{\delta^I_1,\ldots,\delta^I_{r-1}\}),& \quad \text{otherwise if $|\delta(I)|=r-1$}\\
    0,& \quad \text{otherwise if $|\delta(I)|<r-1$}
    \end{cases}
\end{align*}

\vspace{0.2cm}

\noindent \underline{Type 2:} $I$ is left comb, $|I|=r$ and there exists a maximal comb $I'=A'\cup B'$ such that $I=A'$ 
\begin{align*}
    \chi_0(I)=0
\end{align*}

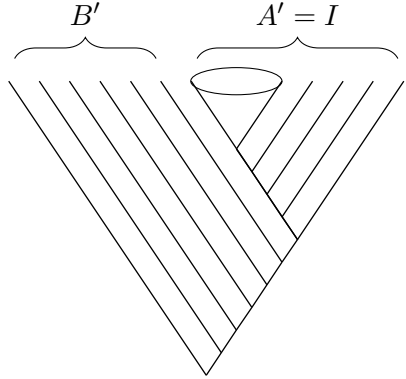
\begin{figure}[h]
\centering
{\hfil \begin{tikzpicture}[scale=0.6]
    
 	\draw (0.66,6.5) ellipse (1 and 0.3);
 	
 	\coordinate (T0) at (-4.33,6.5);
 	\coordinate (T1) at (-3.66,6.5);
 	\coordinate (T2) at (-3,6.5);
 	\coordinate (T3) at (-2.33,6.5);
 	\coordinate (T4) at (-1.66,6.5);
 	\coordinate (T5) at (-1,6.5);
 	\coordinate (T6) at (-0.33,6.5);
 	\coordinate (T10) at (1.66,6.5);
 	\coordinate (T9) at (2.33,6.5);
 	\coordinate (T8) at (3,6.5);
 	\coordinate (T7) at (3.66,6.5);
 	\coordinate (T11) at (4.33,6.5);
    \coordinate (R10) at (0.66,5);
    \coordinate (R9) at (1,4.5);
    \coordinate (R8) at (1.33,4);
    \coordinate (R7) at (1.66,3.5);
    \coordinate (R6) at (2,3);
    \coordinate (R5) at (1.66,2.5);
    \coordinate (R4) at (1.33,2);
    \coordinate (R3) at (1,1.5);
    \coordinate (R2) at (0.66,1);
    \coordinate (R1) at (0.33,0.5);
    \coordinate (R0) at (0,0);
    \node (q1) at (-2.65,7.5) [above,font=\small] {$B'$};
	\node (q2) at (2,7.5) [above,font=\small] {$A'=I$};
    
    \draw [decorate,
    decoration = {brace, amplitude=8pt}] (-4.2,7) --  (-1.1,7);
    \draw [decorate,
    decoration = {brace, amplitude=8pt}] (-0.2,7) --  (4.2,7);
    
    \draw[line width=0.5] (R0)--(R1)--(R2)--(R3)--(R4)--(R5);
    \draw[line width=0.5] (R5)--(R6)--(R7)--(R8)--(R9)--(R10);
    \draw[line width=0.5] (T0)--(R0);
    \draw[line width=0.5] (T1)--(R1);
    \draw[line width=0.5] (T2)--(R2);
    \draw[line width=0.5] (T3)--(R3);
    \draw[line width=0.5] (T4)--(R4);
    \draw[line width=0.5] (T5)--(R5);
    \draw[line width=0.5] (T6)--(R6);
    \draw[line width=0.5] (T7)--(R7);
    \draw[line width=0.5] (T8)--(R8);
    \draw[line width=0.5] (T9)--(R9);
    \draw[line width=0.5] (T10)--(R10);
    \draw[line width=0.5] (T11)--(R6);
    
    \end{tikzpicture}\hfil}
    \caption{An example of left comb of type 2.}
    \label{fig:type2}
\end{figure}

\vspace{0.2cm}

\noindent \underline{Type 3:} $I$ is left comb, $\ell=|A|\leq r$ and $r+1\leq |I|\leq 2r-2$.
\begin{align*}
    \chi_0(I)=\begin{cases}
    0,& \quad \text{if }\delta^I_{r-3}>\delta^I_{r-2}<\delta^I_{r-1}\\
    1,& \quad \text{if }\delta^I_{r-3}<\delta^I_{r-2}>\delta^I_{r-1}\\
    \phi_{|\delta(I)|}(\{\delta^I_1,\ldots,\delta^I_{s-1}\}),& \quad \text{otherwise if $|\delta(I)|\geq r-1$}\\
    0,& \quad \text{otherwise if $|\delta(I)|<r-1$}
    \end{cases}
\end{align*}

\vspace{0.2cm}

\noindent \underline{Type 4:} $I$ is left comb, $\ell=|A|\leq r$ and $|I|\geq 2r-1$.
\begin{align*}
    \chi_0(I)=\begin{cases}
    \phi_{s-r}(\{\delta_r^I,\ldots,\delta_{s-1}^I\}),& \quad \text{if }\delta^I_{r-3}>\delta^I_{r-2}<\delta^I_{r-1}\\
    1-\phi_{s-r}(\{\delta_r^I,\ldots,\delta_{s-1}^I\}),& \quad \text{if }\delta^I_{r-3}<\delta^I_{r-2}>\delta^I_{r-1}\\
    \phi_{|\delta(I)|}(\{\delta^I_1,\ldots,\delta^I_{s-1}\}),& \quad \text{otherwise}
    \end{cases}
\end{align*}

\begin{figure}[h]
\centering
{\hfil \begin{tikzpicture}[scale=0.6]
    
 	\draw (-3.33,6.5) ellipse (1 and 0.3);
 	\draw (-7,1.5)[red] ellipse (0.7 and 2);
 	
 	\coordinate (T11) at (-4.33,6.5);
 	\coordinate (T10) at (-2.33,6.5);
 	\coordinate (T9) at (-1.66,6.5);
 	\coordinate (T8) at (-1,6.5);
 	\coordinate (T7) at (-0.33,6.5);
 	\coordinate (T6) at (0.33,6.5);
 	\coordinate (T5) at (1,6.5);
 	\coordinate (T4) at (1.66,6.5);
 	\coordinate (T3) at (2.33,6.5);
 	\coordinate (T2) at (3,6.5);
 	\coordinate (T1) at (3.66,6.5);
 	\coordinate (T0) at (4.33,6.5);
    \coordinate (R10) at (-3.33,5);
    \coordinate (R9) at (-3,4.5);
    \coordinate (R8) at (-2.66,4);
    \coordinate (R7) at (-2.33,3.5);
    \coordinate (R6) at (-2,3);
    \coordinate (R5) at (-1.66,2.5);
    \coordinate (R4) at (-1.33,2);
    \coordinate (R3) at (-1,1.5);
    \coordinate (R2) at (-0.66,1);
    \coordinate (R1) at (-0.33,0.5);
    \coordinate (R0) at (0,0);
    \coordinate (L11) at (-7,6.5);
    \coordinate (L10) at (-7,5);
    \coordinate (L9) at (-7,4.5);
    \coordinate (L8) at (-7,4);
    \coordinate (L7) at (-7,3.5);
    \coordinate (L6) at (-7,3);
    \coordinate (L5) at (-7,2.5);
    \coordinate (L4) at (-7,2);
    \coordinate (L3) at (-7,1.5);
    \coordinate (L2) at (-7,1);
    \coordinate (L1) at (-7,0.5);
    \coordinate (L0) at (-7,0);
    
    \node (q1) at (-3.33,6.8) [above,font=\small] {$A$};
	\node (q2) at (-2.3,8) [above,font=\small] {$\{x_1,\ldots,x_r\}$};
	\node (q3) at (2.3,8) [above,font=\small] {$\{x_{r+1},\ldots,x_s\}$};
	\node (q4) at (-8.4,3.25) [left,font=\small] {$\{\delta_1^{I},\ldots, \delta_{s-1}^{I}\}$};
	\node (q5) at (-7,-0.5) [below,font=\small] {$\{\delta_r^{I},\ldots,\delta_{s-1}^{I}\}$};
    
    \draw [decorate,
    decoration = {brace, amplitude=8pt}] (-4.2,7.5) --  (-0.4,7.5);
    \draw [decorate,
    decoration = {brace, amplitude=8pt}] (0.4,7.5) --  (4.2,7.5);
    \draw [decorate,
    decoration = {brace, amplitude=8pt}] (-8,0.1) --  (-8,6.4);
    
    \draw[line width=0.5] (R0)--(R1)--(R2)--(R3)--(R4)--(R5);
    \draw[line width=0.5] (R5)--(R6)--(R7)--(R8)--(R9)--(R10);
    \draw[line width=0.5] (T0)--(R0);
    \draw[line width=0.5] (T1)--(R1);
    \draw[line width=0.5] (T2)--(R2);
    \draw[line width=0.5] (T3)--(R3);
    \draw[line width=0.5] (T4)--(R4);
    \draw[line width=0.5] (T5)--(R5);
    \draw[line width=0.5] (T6)--(R6);
    \draw[line width=0.5] (T7)--(R7);
    \draw[line width=0.5] (T8)--(R8);
    \draw[line width=0.5] (T9)--(R9);
    \draw[line width=0.5] (T10)--(R10);
    \draw[line width=0.5] (T11)--(R6);
    \draw[line width=0.5] (L0)--(L11);
    \draw[line width=0.5][dashed] (L0)--(R0);
    \draw[line width=0.5][dashed] (L1)--(R1);
    \draw[line width=0.5][dashed] (L2)--(R2);
    \draw[line width=0.5][dashed] (L3)--(R3);
    \draw[line width=0.5][dashed] (L4)--(R4);
    \draw[line width=0.5][dashed] (L5)--(R5);
    \draw[line width=0.5][dashed] (L6)--(R6);
    
    \draw[fill][red] (R0) circle [radius=0.1];
 	\draw[fill][red] (R1) circle [radius=0.1];
 	\draw[fill][red] (R2) circle [radius=0.1];
 	\draw[fill][red] (R3) circle [radius=0.1];
 	\draw[fill][red] (R4) circle [radius=0.1];
 	\draw[fill][red] (R5) circle [radius=0.1];
 	\draw[fill][red] (R6) circle [radius=0.1];
 	\draw[fill][red] (L0) circle [radius=0.1];
 	\draw[fill][red] (L1) circle [radius=0.1];
 	\draw[fill][red] (L2) circle [radius=0.1];
 	\draw[fill][red] (L3) circle [radius=0.1];
 	\draw[fill][red] (L4) circle [radius=0.1];
 	\draw[fill][red] (L5) circle [radius=0.1];
 	\draw[fill][red] (L6) circle [radius=0.1];
    
    \end{tikzpicture}\hfil}
    \caption{A left comb of Type $4$ and its projections.}
    \label{fig:type4}
\end{figure}
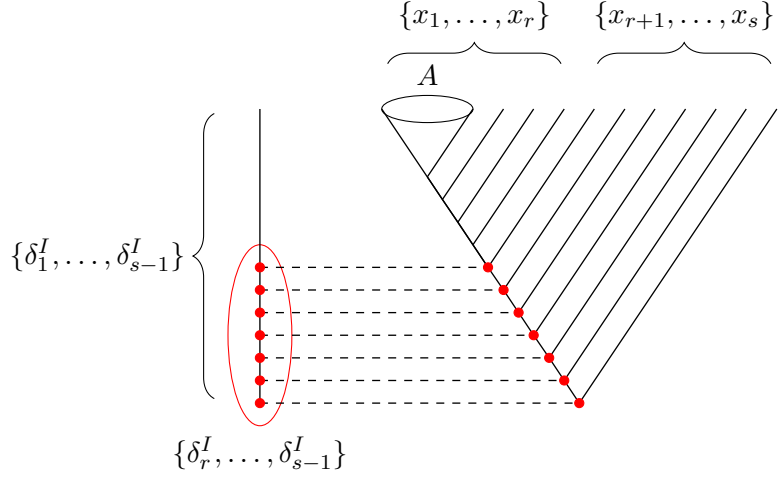

\vspace{0.2cm}

\noindent \underline{Type 5:} $I$ is left comb, $\ell=|A|> r$ and $|B|\geq r$.
\begin{align*}
    \chi_0(I)=\phi_{|B|}(\delta_{\ell}^I,\ldots,\delta_{s-1}^I)
\end{align*}

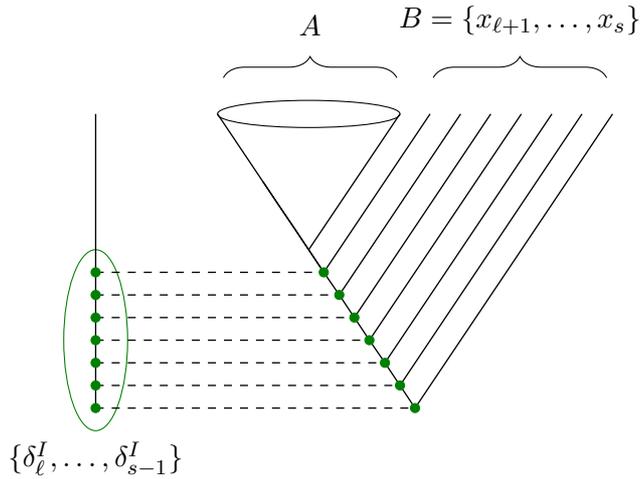
\begin{figure}[h]
\centering
{\hfil \begin{tikzpicture}[scale=0.6]
    
 	\draw (-2.33,6.5) ellipse (2 and 0.3);
 	\draw (-7,1.5)[green1] ellipse (0.7 and 2);
 	
 	\coordinate (T11) at (-4.33,6.5);
 	\coordinate (T10) at (-2.33,6.5);
 	\coordinate (T9) at (-1.66,6.5);
 	\coordinate (T8) at (-1,6.5);
 	\coordinate (T7) at (-0.33,6.5);
 	\coordinate (T6) at (0.33,6.5);
 	\coordinate (T5) at (1,6.5);
 	\coordinate (T4) at (1.66,6.5);
 	\coordinate (T3) at (2.33,6.5);
 	\coordinate (T2) at (3,6.5);
 	\coordinate (T1) at (3.66,6.5);
 	\coordinate (T0) at (4.33,6.5);
    \coordinate (R10) at (-3.33,5);
    \coordinate (R9) at (-3,4.5);
    \coordinate (R8) at (-2.66,4);
    \coordinate (R7) at (-2.33,3.5);
    \coordinate (R6) at (-2,3);
    \coordinate (R5) at (-1.66,2.5);
    \coordinate (R4) at (-1.33,2);
    \coordinate (R3) at (-1,1.5);
    \coordinate (R2) at (-0.66,1);
    \coordinate (R1) at (-0.33,0.5);
    \coordinate (R0) at (0,0);
    \coordinate (L11) at (-7,6.5);
    \coordinate (L10) at (-7,5);
    \coordinate (L9) at (-7,4.5);
    \coordinate (L8) at (-7,4);
    \coordinate (L7) at (-7,3.5);
    \coordinate (L6) at (-7,3);
    \coordinate (L5) at (-7,2.5);
    \coordinate (L4) at (-7,2);
    \coordinate (L3) at (-7,1.5);
    \coordinate (L2) at (-7,1);
    \coordinate (L1) at (-7,0.5);
    \coordinate (L0) at (-7,0);
    
	\node (q2) at (-2.3,8) [above,font=\small] {$A$};
	\node (q3) at (2.3,8) [above,font=\small] {$B=\{x_{\ell+1},\ldots,x_s\}$};
	\node (q5) at (-7,-0.5) [below,font=\small] {$\{\delta_\ell^{I},\ldots,\delta_{s-1}^{I}\}$};
    
    \draw [decorate,
    decoration = {brace, amplitude=8pt}] (-4.2,7.3) --  (-0.4,7.3);
    \draw [decorate,
    decoration = {brace, amplitude=8pt}] (0.4,7.3) --  (4.2,7.3);
    
    \draw[line width=0.5] (R0)--(R1)--(R2)--(R3)--(R4)--(R5);
    \draw[line width=0.5] (R5)--(R6)--(R7)--(R8)--(R9)--(R10);
    \draw[line width=0.5] (T0)--(R0);
    \draw[line width=0.5] (T1)--(R1);
    \draw[line width=0.5] (T2)--(R2);
    \draw[line width=0.5] (T3)--(R3);
    \draw[line width=0.5] (T4)--(R4);
    \draw[line width=0.5] (T5)--(R5);
    \draw[line width=0.5] (T6)--(R6);
    \draw[line width=0.5] (T7)--(R7);
    \draw[line width=0.5] (T11)--(R6);
    \draw[line width=0.5] (L0)--(L11);
    \draw[line width=0.5][dashed] (L0)--(R0);
    \draw[line width=0.5][dashed] (L1)--(R1);
    \draw[line width=0.5][dashed] (L2)--(R2);
    \draw[line width=0.5][dashed] (L3)--(R3);
    \draw[line width=0.5][dashed] (L4)--(R4);
    \draw[line width=0.5][dashed] (L5)--(R5);
    \draw[line width=0.5][dashed] (L6)--(R6);
    
    \draw[fill][green1] (R0) circle [radius=0.1];
 	\draw[fill][green1] (R1) circle [radius=0.1];
 	\draw[fill][green1] (R2) circle [radius=0.1];
 	\draw[fill][green1] (R3) circle [radius=0.1];
 	\draw[fill][green1] (R4) circle [radius=0.1];
 	\draw[fill][green1] (R5) circle [radius=0.1];
 	\draw[fill][green1] (R6) circle [radius=0.1];
 	\draw[fill][green1] (L0) circle [radius=0.1];
 	\draw[fill][green1] (L1) circle [radius=0.1];
 	\draw[fill][green1] (L2) circle [radius=0.1];
 	\draw[fill][green1] (L3) circle [radius=0.1];
 	\draw[fill][green1] (L4) circle [radius=0.1];
 	\draw[fill][green1] (L5) circle [radius=0.1];
 	\draw[fill][green1] (L6) circle [radius=0.1];
    
    \end{tikzpicture}\hfil}
    \caption{A left comb of Type $5$ and its projections.}
    \label{fig:type5}
\end{figure}

\vspace{0.2cm}

\noindent \underline{Type 6:} All other broken or left maximal combs.
\begin{align*}
    \chi_0(I)=0
\end{align*}

Finally, the auxiliary coloring $\chi_0$ define a coloring $\chi:[2^N]^{(k+r)}\rightarrow \{0,1\}$ as follows:
\begin{align*}
    \chi(X)=\sum_{I\in \cI_X}\chi_0(I) \pmod{2}
\end{align*}

\section{Coloring data}\label{sec:info}

Given an edge $X$ and a maximal comb $I \subseteq X$, one can determine the color $\chi_0(I)$ by looking at the type of the maximal comb $I$. Some of the types do not use information on the ancestors to determine its coloring. For instance, if $I$ is of type $2$, then its color will be always $0$. The projection of the ancestors $\delta_i^I$ has no influence in defining $\chi_0(I)$. However, if $I$ is of type $1$, then the color crucially depends on the projection of the ancestors. 

This observation suggests the following definition. Given an edge $X \in [2^N]^{(k+r)}$ and a maximal comb $I\subseteq X$, let the \textit{coloring data} $F(I)$ of $I$ be defined as the ordered set of ancestors whose projection determine the coloring $\chi_0(I)$. More explicitly, we can define directly the coloring data of $I$ by looking its types. We may assume here that $I=\{x_1,\ldots,x_s\}$ is a broken or left maximal comb.

\begin{itemize}
    \item Type 1,3 and 4: $F(I)=\{a(x_i,x_{i+1})\}_{1\leq i \leq s-1}$
    \item Type 2 and 6: $F(I)=\emptyset$
    \item Type 5: $F(I)=\{a(x_i,x_{i+1})\}_{\ell \leq i \leq s-1}$
\end{itemize} 

Our first observation is that maximal combs with same data have same color. We say that two combs have the same \textit{orientation} if they are of the same class (e.g., both are left combs).

\begin{proposition}\label{prop:equalinfo}
Let $X$, $X' \in [2^N]^{(k+r)}$ be two edges. If $I$ and $I'$ are maximal combs of same type and orientation in $X$ and $X'$, respectively, such that $F(I)=F(I')$, then $\chi_0(I)=\chi_0(I')$.
\end{proposition}

\begin{proof}
The proof basically consists of checking the consistency of our definition. If $F(I)=F(I')=\emptyset$, then $I$ and $I'$ are either of type $2$ or $6$. In both cases $\chi_0(I)=\chi_0(I')=0$. 

If $I=\{x_1,\ldots,x_s\}$ and $I'=\{x'_1,\ldots,x'_{s'}\}$ are of type $1$, $3$ or $4$, then since $a(I)=F(I)=F(I')=a(I')$ we obtain by Fact \ref{fact:shape} that $s=s'$ and $a(x_i,x_{i+1})=a(x'_i,x'_{i+1})$ for every $1\leq i \leq s$. Therefore $\delta(x_i,x_{i+1})=\delta(x'_i,x'_{i+1})$ for every $1\leq i \leq s$ and by the coloring defined in Section \ref{subsec:variant}, it follows that $\chi_0(I)=\chi_0(I')$.

The last case that we need to check is when $I=A\cup B=\{x_1,\ldots,x_{s}\}$ and $I'=A'\cup B'=\{x'_1,\ldots,x'_{s'}\}$ are of type $5$, where $|A|=\ell$ and $|A'|=\ell'$. As usual, we assume that $I$ and $I'$ are left combs. Since $a(\{x_\ell,\ldots,x_s\})=F(I)=F(I')=a(\{x'_{\ell'},\ldots,x'_{s'}\})$, it follows again by Fact \ref{fact:shape} that $s-\ell=|B|=|B'|=s'-\ell'$ and $a(x_i,x_{i+1})=a(x'_i,x'_{i+1})$ for $\ell\leq i\leq s-1$. Thus $\delta(x_i,x_{i+1})=\delta(x'_i,x'_{i+1})$ for $\ell\leq i \leq s-1$ and by the coloring of type 5 we obtain that $\chi_0(I)=\chi_0(I')$.
\end{proof}

Although maximal combs in the same edge do not need to be disjoint, the next result shows that they do not share the same coloring data.

\begin{proposition}\label{prop:uniqueinfo}
Let $X \in [2^N]^{(k+r)}$ be an edge. If $I=A\cup B$ and $I'=A'\cup B'$ are maximal combs in $X$, then $F(I)\cap F(I')=\emptyset$.
\end{proposition}

\begin{proof}
Suppose without loss of generality that $|I|\leq |I'|$. By Proposition \ref{prop:maximalcomb} either $I\cap I'=\emptyset$ or $I\subseteq A'$. If $I\cap I'=\emptyset$, then by the fact that $I, I'$ are closed we obtain that $a(I)\cap a(I')=\emptyset$. Since $F(I)\subseteq a(I)$ by definition, it follows that $F(I)\cap F(I')=\emptyset$.

Now suppose that $I\subseteq A'$. We may assume that $F(I), F(I')\neq \emptyset$ and consequently that $I,I'$ are of type $1$, $3$, $4$ or $5$. Since maximal combs of types $1$, $3$, $4$ and $5$ have size at least $r$, our assumption implies that $|I'|\geq |I|\geq r$. 

We claim that $|A'|>r$. Suppose that $|A'|=r$. Since $r\leq |I|\leq |A'|$ we obtain that $I=A'$ and $|I|=r$. The maximality of $I$ implies that it is either a left or right maximal comb (otherwise we could extend the comb to $I\cup B$). However, in this case $I$ is of type $2$. Thus $F(I)=\emptyset$, which contradicts our assumption on $I$. Therefore $|A'|>r$ and consequently $I'$ is of type $5$. Write $I'=\{x'_1,\ldots,x'_{s'}\}$ with $|A'|=\ell'$ and assume that $I'$ is a left comb. Then by definition
\begin{align*}
    F(I')=\{a(x'_i,x'_{i+1})\}_{\ell'\leq i \leq s'-1}.
\end{align*}
Since $I\subseteq A'$ we obtain that $F(I)\subseteq a(A')$. By the structure of a left comb we have that
\begin{align*}
    \delta(x'_1,x'_{\ell'})>\delta(x'_{\ell'},x'_{\ell'+1})>\delta(x'_{\ell'+1},x'_{\ell'+2})>\ldots>\delta(x'_{s'-1},x'_{s'}).
\end{align*}

\begin{figure}[h]
\centering
{\hfil \begin{tikzpicture}[scale=0.6]
    
 	\draw (-2,6.5) ellipse (2.33 and 0.5);
 	\draw[red] (-2.66,6.5) ellipse (1 and 0.3);
 	\draw[rotate=124][red] (1.5,0) ellipse (2 and 0.5);
 	
 	\coordinate (T11) at (-4.33,6.5);
 	\coordinate (T12) at (-3.66,6.5);
 	\coordinate (T10) at (-2.33,6.5);
 	\coordinate (T9) at (-1.66,6.5);
 	\coordinate (T8) at (-1,6.5);
 	\coordinate (T7) at (-0.33,6.5);
 	\coordinate (T6) at (0.33,6.5);
 	\coordinate (T5) at (1,6.5);
 	\coordinate (T4) at (1.66,6.5);
 	\coordinate (T3) at (2.33,6.5);
 	\coordinate (T2) at (3,6.5);
 	\coordinate (T1) at (3.66,6.5);
 	\coordinate (T0) at (4.33,6.5);
    \coordinate (R10) at (-3.33,5);
    \coordinate (R9) at (-3,4.5);
    \coordinate (R8) at (-2.66,4);
    \coordinate (R7) at (-2.33,3.5);
    \coordinate (R6) at (-2,3);
    \coordinate (R5) at (-1.66,2.5);
    \coordinate (R4) at (-1.33,2);
    \coordinate (R3) at (-1,1.5);
    \coordinate (R2) at (-0.66,1);
    \coordinate (R1) at (-0.33,0.5);
    \coordinate (R0) at (0,0);
    \coordinate (X) at (-2.66,5);
    
	\node (q2) at (-2,8) [above,font=\small] {$A$};
	\node (q3) at (2.65,8) [above,font=\small] {$B$};
	\node (q4) at (-4,5.5) [left,font=\small] {$a(I)$};
	\node (q3) at (0,-0.2) [below left,font=\small] {$\{a(x_i,x_{i+1})\}_{\ell'\leq i \leq s-1}$};
    
    \draw [decorate,
    decoration = {brace, amplitude=8pt}] (-4.2,7.5) --  (0.2,7.5);
    \draw [decorate,
    decoration = {brace, amplitude=8pt}] (1.1,7.5) --  (4.2,7.5);
    
    \draw[line width=0.5] (R0)--(R1)--(R2)--(R3)--(R4)--(R5);
    \draw[line width=0.5] (R5)--(R6)--(R7)--(R8)--(R9)--(R10);
    \draw[line width=0.5] (T0)--(R0);
    \draw[line width=0.5] (T1)--(R1);
    \draw[line width=0.5] (T2)--(R2);
    \draw[line width=0.5] (T3)--(R3);
    \draw[line width=0.5] (T4)--(R4);
    \draw[line width=0.5] (T5)--(R5);
    \draw[line width=0.5] (T6)--(R6);
    \draw[line width=0.5] (T11)--(R6);
    \draw[line width=0.5][red] (X)--(T9);
    \draw[line width=0.5][red] (X)--(T12);
    
    \draw[fill][red] (R0) circle [radius=0.1];
 	\draw[fill][red] (R1) circle [radius=0.1];
 	\draw[fill][red] (R2) circle [radius=0.1];
 	\draw[fill][red] (R3) circle [radius=0.1];
 	\draw[fill][red] (R4) circle [radius=0.1];
 	\draw[fill][red] (R5) circle [radius=0.1];
    
    \end{tikzpicture}\hfil}
    \caption{A picture of $I$ and $\{a(x_i,x_{i+1})\}_{\ell'\leq i \leq s-1}$}
    \label{fig:uniqueinfo}
\end{figure}
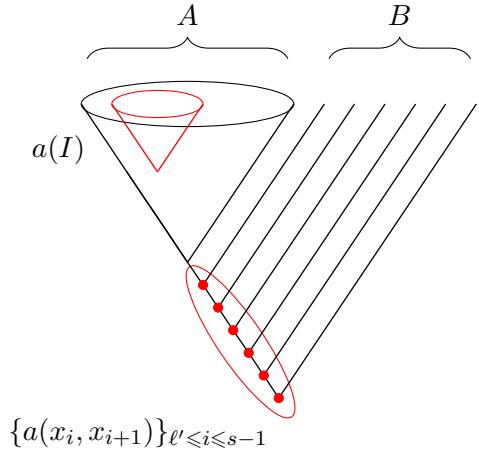

Thus $a(I)\cap \{a(x'_i,x'_{i+1})\}_{\ell'\leq i \leq s'-1}=\emptyset$ and consequently $F(I)\cap F(I')=\emptyset$.
\end{proof}

\section{Pre-processing}\label{sec:preprocessing}

As discussed in Subsection \ref{subsec:overview}, we now turn our focus to show that a simple daisy $H$ can be pre-processed in a smaller simple subdaisy $H'$ with the property that for every edge $X$ with petal petal $P$ we have that either $P$ is a closed interval in $X$ or is part of the ``teeth" of a maximal comb in $X$.

\begin{lemma}\label{lem:preprocessing}
For any simple $(r,m,k)$-daisy $H$ with vertex set $V(H)\subseteq [2^N]$, $K_0<M<K_1$, $|K_0\cup K_1|=k$ and $|M|=m$, there exists a subset $M'\subseteq M$ of size $|M'|=\frac{1}{2}k^{-1/2}m^{1/2}$ such that the simple $(r,\frac{1}{2}k^{-1/2}m^{1/2},k)$-daisy $H'=H[K_0\cup M'\cup K_1]$ satisfies one of the following:
\begin{enumerate}
    \item $M'$ is a closed interval in $V(H')$.
    \item There exists a maximal comb $I=A\cup B$ in $V(H')$ such that $M'\subseteq B$.
\end{enumerate}
\end{lemma}

\begin{figure}[h]
\centering
{\hfil \begin{tikzpicture}[scale=0.4]
    
 	\draw (-5.5,9.5) ellipse (0.833 and 0.3);
 	\draw (5.5,9.5) ellipse (0.833 and 0.3);
 	\draw (-3.33,9.5) ellipse (1 and 0.3);
 	\draw[red] (1.33,9.5) ellipse (2 and 0.4);
 	
 	\coordinate (T13) at (-4.66,9.5);
 	\coordinate (T14) at (-6.33,9.5);
 	\coordinate (T15) at (4.66,9.5);
 	\coordinate (T16) at (6.33,9.5);
 	\coordinate (T11) at (-4.33,9.5);
 	\coordinate (T12) at (-3.66,9.5);
 	\coordinate (T10) at (-2.33,9.5);
 	\coordinate (T9) at (-1.66,9.5);
 	\coordinate (T8) at (-1,9.5);
 	\coordinate (T7) at (-0.33,9.5);
 	\coordinate (T6) at (0.33,9.5);
 	\coordinate (T5) at (1,9.5);
 	\coordinate (T4) at (1.66,9.5);
 	\coordinate (T3) at (2.33,9.5);
 	\coordinate (T2) at (3,9.5);
 	\coordinate (T1) at (3.66,9.5);
 	\coordinate (T0) at (4.33,9.5);
    \coordinate (R10) at (-3.33,8);
    \coordinate (R9) at (-3,7.5);
    \coordinate (R8) at (-2.66,7);
    \coordinate (R7) at (-2.33,6.5);
    \coordinate (R6) at (-2,6);
    \coordinate (R5) at (-1.66,5.5);
    \coordinate (R4) at (-1.33,5);
    \coordinate (R3) at (-1,4.5);
    \coordinate (R2) at (-0.66,4);
    \coordinate (R1) at (-0.33,3.5);
    \coordinate (R0) at (0,3);
    \coordinate (R') at (1,1.5);
    \coordinate (R'') at (0,0);
    \coordinate (X) at (-5.5,8.25);
    \coordinate (Y) at (5.5,8.25);
    
    \node (q1) at (-3.65,10.7) [above,font=\small] {$K_0$};
	\node (q2) at (1.35,10.7) [above,font=\small] {$M'$};
	\node (q3) at (4.95,10.7) [above,font=\small] {$K_1$};
    
    \draw [decorate,
    decoration = {brace, amplitude=7pt}] (-6.2,10.2) --  (-1.1,10.2);
    \draw [decorate,
    decoration = {brace, amplitude=7pt}] (-0.5,10.2) --  (3.2,10.2);
    \draw [decorate,
    decoration = {brace, amplitude=7pt}] (3.7,10.2) --  (6.2,10.2);
    
    \draw[line width=0.5] (R0)--(R1)--(R2);
    \draw[line width=0.5][red] (R2)--(R3)--(R4)--(R5);
    \draw[line width=0.5][red] (R5)--(R6)--(R7);
    \draw[line width=0.5] (R7)--(R8)--(R9)--(R10);
    \draw[line width=0.5] (T0)--(R0);
    \draw[line width=0.5] (T1)--(R1);
    \draw[line width=0.5][red] (T2)--(R2);
    \draw[line width=0.5][red] (T3)--(R3);
    \draw[line width=0.5][red] (T4)--(R4);
    \draw[line width=0.5][red] (T5)--(R5);
    \draw[line width=0.5][red] (T6)--(R6);
    \draw[line width=0.5][red] (T7)--(R7);
    \draw[line width=0.5] (T8)--(R8);
    \draw[line width=0.5] (T9)--(R9);
    \draw[line width=0.5] (T10)--(R10);
    \draw[line width=0.5] (T11)--(R10);
    \draw[line width=0.5] (R'')--(R')--(R0);
    \draw[line width=0.5] (R'')--(T14);
    \draw[line width=0.5] (R')--(T16);
    \draw[line width=0.5] (X)--(T13);
    \draw[line width=0.5] (Y)--(T15);
    
    \draw (-11.83,9.5) ellipse (1.5 and 0.4);
 	\draw (-21,9.5) ellipse (2 and 0.4);
 	\draw[red] (-16.16,9.5) ellipse (2.5 and 0.4);
 	
 	\coordinate (S0) at (-10.33,9.5);
 	\coordinate (S1) at (-13.33,9.5);
 	\coordinate (S2) at (-13.66, 9.5);
 	\coordinate (S3) at (-18.66, 9.5);
 	\coordinate (S4) at (-19, 9.5);
 	\coordinate (S5) at (-23, 9.5);
 	\coordinate (Q0) at (-16.66, 0);
 	\coordinate (Q1) at (-11.83, 7.25);
 	\coordinate (Q2) at (-21, 6.5);
 	\coordinate (Q3) at (-18.33, 2.5);
 	\coordinate (Q4) at (-16.16, 5.75);
 	
 	\draw[line width=0.5] (S0)--(Q0)--(S5);
 	\draw[line width=0.5] (S1)--(Q1);
 	\draw[line width=0.5] (S4)--(Q2);
 	\draw[line width=0.5] (Q4)--(Q3);
 	\draw[line width=0.5][red] (S3)--(Q4);
 	\draw[line width=0.5][red] (S2)--(Q4);
 	
 	\draw [decorate,
    decoration = {brace, amplitude=7pt}] (-13.2,10.2) --  (-10.4,10.2);
    \draw [decorate,
    decoration = {brace, amplitude=7pt}] (-18.5,10.2) --  (-13.7,10.2);
    \draw [decorate,
    decoration = {brace, amplitude=7pt}] (-22.9,10.2) --  (-19.1,10.2);
 	
 	\node (q4) at (-21,10.7) [above,font=\small] {$K_0$};
	\node (q5) at (-16.16,10.7) [above,font=\small] {$M'$};
	\node (q6) at (-11.83,10.7) [above,font=\small] {$K_1$};
 	
    \end{tikzpicture}\hfil}
    \caption{An example of $H'$ satisfying statement (1) and (2).}
    \label{fig:preprocessing}
\end{figure}
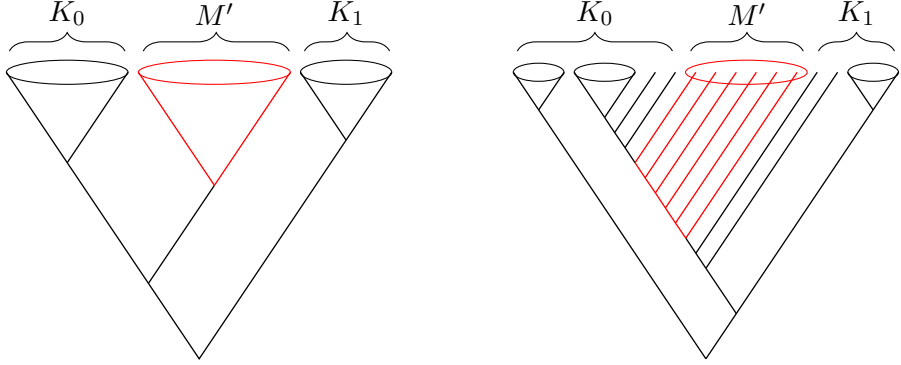

\begin{proof}
Let $V:=V(H)$. Given a closed interval $I\subseteq V$, by condition (ii) of Definition \ref{def:closed} there exists a vertex $u \in a(I)$ such that $I=V(u)$. Consider the partition of $I$ given by $I=I^L\cup I^R$, where $I^L=V_L(u)$ are the left descendants of $u$ and $I^R=V_R(u)$ are the right descendants of $u$. Let $u^L$ be the left child of $u$ and $u^R$ be the right child. Hence, $I^L=V(u^L)$ and $I^R=V(u^R)$ and consequently $I=I^L \cup I^R$ is a partition of a closed interval in $V$ into two non empty closed intervals in $V$.

\begin{figure}[b]
\centering
{\hfil \begin{tikzpicture}[scale=0.4]
    
    \draw[blue] (2,5.5) ellipse (1.66 and 0.3);
    \draw[red] (-2,5.5) ellipse (1.66 and 0.3);
    \draw (0,5.5) ellipse (8 and 0.5);
    
    \coordinate (T0) at (3.66,5.5);
    \coordinate (T1) at (0.33,5.5);
    \coordinate (T2) at (-0.33,5.5);
    \coordinate (T3) at (-3.66,5.5);
    \coordinate (R0) at (0,0);
    \coordinate (R1) at (2,3);
    \coordinate (R2) at (-2,3);

    \node (q0) at (0,0) [below right] {$u$};
    \node (q1) at (2,3) [below right] {$u^R$};
	\node (q2) at (-2,3) [below left] {$u^L$};
	\node (q3) at (8,5.5) [above right] {$V$};
	\node (q4) at (-2,5.9) [above] {$I^L$};
	\node (q5) at (2,5.9) [above] {$I^R$};
	\node (q6) at (0,7.8) [above] {$I=V(u)$};
    
    \draw [decorate,
    decoration = {brace, amplitude=8pt}] (-3.3,7.2) --  (3.3,7.2);
    
    \draw[line width=0.5] (R1)--(R0)--(R2);
    \draw[line width=0.5][blue] (T0)--(R1)--(T1);
    \draw[line width=0.5][red] (T2)--(R2)--(T3);
    
    \draw[fill] (R0) circle [radius=0.1];
 	\draw[fill][blue] (R1) circle [radius=0.1];
 	\draw[fill][red] (R2) circle [radius=0.1];
    
    \end{tikzpicture}\hfil}
    \caption{Partition of a closed interval into two other closed intervals.}
    \label{fig:Ipart}
\end{figure}
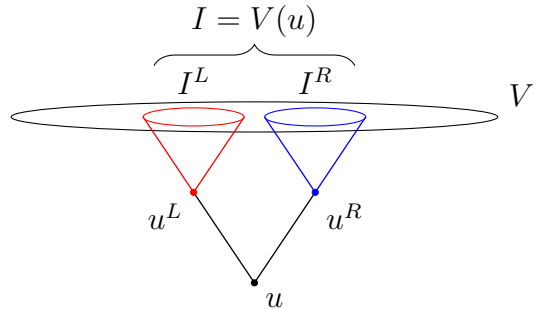

We will construct our set $M'$ iteratively. This is done in two stages. In the first stage we start with the closed interval $Y_0=V$ and proceed recursively as follows: For a closed interval $Y_i\subseteq V$, let $Y_i=Y_i^L\cup Y_i^R$ be the partition described above in two closed intervals. The choice of $Y_{i+1}$ is determined by the conditions below
\begin{enumerate}
\item[(P1)] Set $Y_{i+1}:=Y_i^L$ if $|Y_i^L\cap M|\geq |Y_i^R\cap M|$.
\item[(P2)] Set $Y_{i+1}:=Y_i^R$ if $|Y_i^L\cap M|<|Y_i^R\cap M|$.
\end{enumerate}

We stop the process whenever $Y_{i}\cap K_0=\emptyset$ or $Y_{i}\cap K_1=\emptyset$. Note that since $Y_i^L$ and $Y_i^R$ are non empty, at each iteration of the process the size of $|(K_0\cup K_1)\cap Y_i|$ reduces at least by one. Thus, in a finite amount of time the process terminates. Let $Y$ be the closed interval obtained in the end. We may assume without loss of generality that $Y\cap K_1=\emptyset$. Write $Y=K_Y\cup M_Y$, where $K_Y\subseteq K_0$ and $M_Y\subseteq M$. It is not hard to check by the construction that $|M_Y|\geq m/2$.

For the second stage, let $Z_0=Y$. Given a closed interval $Z_i \subseteq V$, let $Z_i=Z_i^L\cup Z_i^R$ be the partition into two non empty closed intervals. By definition we have that $Z_i^L<Z_i^R$. We say that a partition $Z_i^L\cup Z_i^R$ is of type $A$ if $Z_i^R\cap K_0=\emptyset$ and of type $B$ if $Z_i^R\cap K_0\neq \emptyset$. The choice of $Z_{i+1}$ will depend on the type of partition as follows: 

\vspace{0,2cm}

\noindent \underline{Type A:} $Z_i^R\cap K_0 = \emptyset$. 

\vspace{0.2cm}

\begin{enumerate}
    \item[(A1)] Set $Z_{i+1}:=Z_i^L$ if $|Z_i^R|<\frac{1}{2}k^{-1/2}m^{1/2}$.
    \item[(A2)] Set $Z_{i+1}:=Z_I^R$ if $|Z_i^R|\geq \frac{1}{2}k^{-1/2}m^{1/2}$ and stop the process.
\end{enumerate}

\vspace{0,2cm}

\noindent \underline{Type B:} $Z_i^R\cap K_0 \neq \emptyset$. 

\vspace{0.2cm}

\begin{enumerate}
    \item[(B)] Set $Z_{i+1}:=Z_i^R$.
\end{enumerate}

We terminate the process if we either reach condition (A2) or if $Z_{i+1}$ is a singleton. Since $|Z_{i+1}|<|Z_i|$, the process is finite. Let $Z$ be the closed interval obtained at the end. We split into two cases.

If the process terminates after some instance of condition (A2), then it means that $Z=Z_i^R$ is a closed interval in $V$ with $|Z|\geq \frac{1}{2}k^{-1/2}m^{1/2}$ for some index $i$. Because we are in a partition of type $A$ we also obtain that $Z\subseteq M$. Thus, if we set $M'=Z$ the simple subdaisy $H[K_0\cup M'\cup K_1]$ satisfies condition (1) of the statement.

Now suppose that the process terminates with $|Z|=1$. Then it means that for every partition of type A we had an instance of condition (A1). If $Z_{i+1}$ is a set obtained after condition (A1), then $|Z_{i+1}\cap M|> |Z_i\cap M|-\frac{1}{2}k^{-1/2}m^{1/2}$ and $|Z_{i+1}\cap K_0|=|Z_{i+1}\cap K_0|$. That is, condition (A1) removes less than $\frac{1}{2}k^{-1/2}m^{1/2}$ element of $M$ from $Z_i$ and no elements of $K_0$ from it. Moreover, if $Z_{i+1}$ is obtained after condition (B), then $|Z_{i+1}\cap M|=|Z_i\cap M|$ and $|Z_{i+1}\cap K_0|<|Z_i\cap K_0|$. That is, $M$ remains unaffected, but $K_0$ loses at least one element from $K_i$ to $K_{i+1}$.

Consider the sequence of operations applied to $Z_0$ in order to obtain $Z$. Since we start with a set $Z_0=Y$ with $|Z_0\cap M|=|M_Y|\geq m/2$, we obtain that during our process we had at least
\begin{align*}
    \frac{\frac{m}{2}}{\frac{1}{2}k^{-1/2}m^{1/2}}=k^{1/2}m^{1/2}
\end{align*}
instances of condition (A1) in the sequence. Similarly, since $|Z_0\cap K_0|=|K_Y|\leq k$, we obtain that we had at most $k$ instances of condition (B) in the sequence. Hence, by the pigeonhole principle there exists a sequence of consecutive applications of condition (A1) of length at least $m'=k^{1/2}m^{1/2}/(k+1)\geq\frac{1}{2}k^{-1/2}m^{1/2}$. 

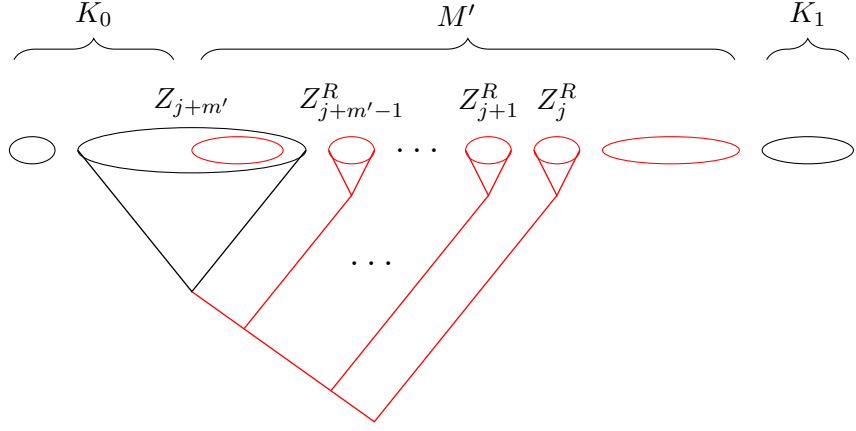
\begin{figure}[h]
\centering
{\hfil \begin{tikzpicture}[scale=0.6]
    
    \draw[red] (1.5,5) ellipse (0.5 and 0.3);
    \draw[red] (3,5) ellipse (0.5 and 0.3);
    \draw[red] (5.5,5) ellipse (1.5 and 0.3);
    \draw[red] (-1.5,5) ellipse (0.5 and 0.3);
    \draw[red] (-4,5) ellipse (1 and 0.3);
    \draw (-5,5) ellipse (2.5 and 0.5);
    \draw (8.5,5) ellipse (1 and 0.3);
    \draw (-8.5,5) ellipse (0.5 and 0.3);
    
    \coordinate (A1) at (9.5,5);
    \coordinate (B1) at (7.5,5);
    \coordinate (AM) at (7,5);
    \coordinate (BM) at (4,5);
    \coordinate (Aj1) at (3.5,5);
    \coordinate (Bj1) at (2.5,5);
    \coordinate (Cj1) at (3,4);
    \coordinate (Aj2) at (2,5);
    \coordinate (Bj2) at (1,5);
    \coordinate (Cj2) at (1.5,4);
    \coordinate (Aj3) at (-1,5);
    \coordinate (Bj3) at (-2,5);
    \coordinate (Cj3) at (-1.5,4);
    \coordinate (Ajm) at (-2.5,5);
    \coordinate (Bjm) at (-7.5,5);
    \coordinate (Cjm) at (-5,1.875);
    \coordinate (A0) at (-8,5);
    \coordinate (B0) at (-9,5);
    \coordinate (R0) at (-1,-1);
    \coordinate (R1) at (-1.952,-0.315);
    \coordinate (R2) at (-3.857, 1.056);

    \node (r) at (0,5) [font=\large] {$\dots$};
    \node (r2) at (-1,2.5) [font=\large] {$\dots$};
    \node (z1) at (3,5.4) [above, font=\small] {$Z_j^R$};
    \node (z2) at (1.5,5.4) [above, font=\small] {$Z_{j+1}^R$};
    \node (z3) at (-1.5,5.4) [above, font=\small] {$Z_{j+m'-1}^R$};
    \node (z4) at (-5,5.5) [above, font=\small] {$Z_{j+m'}$};
    \node (m) at (0.8,7.5) [above, font=\small] {$M'$};
    \node (k1) at (8.5,7.5) [above, font=\small] {$K_1$};
    \node (k0) at (-7.15,7.5) [above, font=\small] {$K_0$};
    
    \draw [decorate,
    decoration = {brace, amplitude=8pt}] (-4.8,7) --  (6.9,7);
    \draw [decorate,
    decoration = {brace, amplitude=8pt}] (7.6,7) --  (9.4,7);
    \draw [decorate,
    decoration = {brace, amplitude=8pt}] (-8.9,7) --  (-5.4,7);
    
    \draw[line width=0.5][red] (Aj1)--(Cj1)--(Bj1);
    \draw[line width=0.5][red] (Aj2)--(Cj2)--(Bj2);
    \draw[line width=0.5][red] (Aj3)--(Cj3)--(Bj3);
    \draw[line width=0.5] (Ajm)--(Cjm)--(Bjm);
    \draw[line width=0.5][red] (Cjm)--(R0)--(Cj1);
    \draw[line width=0.5][red] (Cj2)--(R1);
    \draw[line width=0.5][red] (Cj3)--(R2);
    
    
    \end{tikzpicture}\hfil}
    \caption{Sequence of closed intervals $Z_j^R,\ldots, Z_{j+m'-1}^R$.}
    \label{fig:Zjs}
\end{figure}

Let $Z_j, Z_{j+1},\ldots, Z_{j+m'}$ be the closed intervals involved in the sequence. That is, $Z_{i+1}$ is obtained from $Z_i$ by a condition (A1) for every $j\leq i \leq j+m'-1$. By the algorithm, we obtain closed intervals $Z_j^R,\ldots,Z_{j+m'-1}^R \subseteq M$ all of them with size less than $\frac{1}{2}k^{-1/2}m^{1/2}$. For every $j\leq i \leq j+m'-1$, choose a point $z_i\in Z_i^R$. 

Set $M'=\{z_j,\ldots,z_{j+m'-1}\}$. We claim that $M'$ is a set satisfying condition (2) of the statement. Let $H'=H[K_0\cup M'\cup K_1]$ and $V'=V(H')$. To see that condition (2) is satisfied we just need to find a maximal comb $I=A\cup B \subseteq V'$ such that $M'\subseteq B$. Let $K'=K_0\cap Z_{j+m'} $ and consider the interval $I'=K' \cup M'$ in $V'$. By construction, the intervals $K'$ and $K'\cup \{z_{j+i},\ldots,z_{j+m'-1}\}$ are closed in $V'$ for every $0\leq i\leq m'-1$. Therefore, by condition (a3*) of Definition \ref{def:comb}, the interval $I'=A'\cup B'$ is a left comb and $M' \subseteq B'$. Since every comb can be extended to a maximal one, there exists a maximal left comb $I=A\cup B$ with $A=A'$ and $B'\subseteq B$ such that $M'\subseteq B$ and we are done.
\end{proof}

One of the main consequences of our pre-processing is that it allows us to identify certain closed and non-closed intervals in an arbitrary edge of $H'$. To be more precise, given an edge $X \in E(H')$ with petal $P$ and $V'=V(H')$, let 
\begin{align*}
    \cC_{V',M'}&=\{I:\: \text{$I$ is an interval in $V'$ and either $M'\subseteq I$ or $M'\cap I=\emptyset$}\}  \\
    \cC_{X,P}&=\{I:\: \text{$I$ is an interval in $X$ and either $P\subseteq I$ or $P\cap I=\emptyset$}\}
\end{align*}
be the set of intervals in $V'$ and $X$ such that the intervals either contain or are disjoint of $M'$ and $P$, respectively. The next proposition shows that there is a one-to-one correspondence between $\cC_{V',M'}$ and $\cC_{X,P}$ preserving the property of being closed.

\begin{proposition}\label{prop:preserveclosed}
For a given edge $X\in E(H')$ with petal $P$, there exists a bijection $\Psi: \cC_{V',M'}\rightarrow \cC_{X,P}$ given by
\begin{align*}
    \Psi(I)=I\cap X
\end{align*}
such that $I$ is a closed interval in $V'$ if and only if $\Psi(I)$ is a closed interval in $X$.
\end{proposition}

\begin{proof}
If $I\in \cC_{V',M'}$ is such that $I\cap M'=\emptyset$, then either $I\subseteq K_0$ or $I\subseteq K_1$. Since $X=K_0\cup P\cup K_1$ for some $P\in M'^{(r)}$, we obtain that $\Psi(I)=I\cap X=I$. This shows that $\Psi$ is a bijection from the intervals of $V'$ disjoint of $M'$ to the intervals of $X$ disjoint of $P$. 

Now suppose that $I \in \cC_{V',M'}$ is such that $M'\subseteq I$. Then $I$ can be written as $I=K_I\cup M'$ with $K_I\subseteq K_0\cup K_1$. Thus $\Psi(I)=I\cap X=K_I \cup P$. Since $K_I\neq K_{I'}$ for $I\neq I'$, we obtain that $\Psi$ is an injection from the intervals of $V'$ containing $M'$ to the intervals of $X$ containing $P$. To check surjectivity, just notice that $K_I\cup P$ is an interval if and only if $K_I\cup M'$ is an interval.

It remains to prove that $I$ is closed if and only if $\Psi(I)$ is closed. Throughout the rest of the proof, for a set $S\subseteq V$ we define
\begin{align*}
    x_S=\min(S), \quad y_S=\max(S), \quad u_S=a(x_S,y_S).
\end{align*}
Note that the backwards direction is straightforward from the definition of being closed.

\begin{proposition}\label{prop:backclosed}
If $I$ is closed in $V'$, then $I\cap X$ is closed in $X$.
\end{proposition}

\begin{proof}
Suppose by contradiction that $I\cap X$ is not closed in $X$. Then by condition ($\star\star$) of Definition \ref{def:closed}, there exists $y \in X\setminus I$ such that $u_{I\cap X}$ is an ancestor of $y$. Since $I\cap X \subseteq I$, we have that $u_I$ is an ancestor of $u_{I\cap X}$. Therefore, $y\in X\setminus I \subseteq V'\setminus I$ is an ancestor of $u_I$ which contradicts the fact that $I$ is closed in $V'$.
\end{proof}

The following observation will be useful for the rest of the proof.

\begin{fact}\label{fact:closedanc}
Let $W=V(u_W)$ be a closed interval in $V$. If $x$ and $y$ are two vertices such that $x\in W$ and $y\notin W$, then $a(x,y)=a(y,u_W)$.   
\end{fact}

In particular, Fact \ref{fact:closedanc} applied to $W=M'$ says that an element $y\notin M'$ have the same common ancestor with any $x\in M'$. We split the proof of the forward implication depending on the structure of $H'$ given by Lemma \ref{lem:preprocessing}.

\vspace{0.2cm}

\noindent \underline{Case 1:} $M'$ is a closed interval in $V'$.

The proof of Case 1 is slightly different depending on the location of the interval $I$ in $V'$.

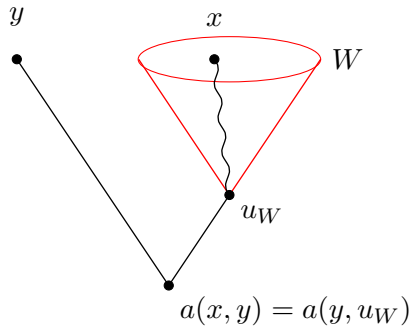
\begin{figure}[h]
\centering
{\hfil \begin{tikzpicture}[scale=0.6]
    
    \draw[red] (1.33,5) ellipse (2 and 0.5);
    
    \coordinate (R) at (0,0);
    \coordinate (Q) at (1.33,2);
    \coordinate (Y) at (-3.33,5);
    \coordinate (A) at (-0.66,5);
    \coordinate (B) at (3.33,5);
    \coordinate (X) at (1,5);
    
    \node (x) at (1,5.5) [above, font=\small] {$x$};
    \node (y) at (-3.33,5.5) [above, font=\small] {$y$};
    \node (q) at (1.33,2) [below right, font=\small] {$u_W$};
    \node (r) at (0,0) [below right, font=\small] {$a(x,y)=a(y,u_W)$};
    \node (w) at (3.33,5) [right, font=\small] {$W$};
    
    
    \draw[line width=0.5] (Y)--(R)--(Q);
    \draw[line width=0.5][red] (A)--(Q)--(B);
    \draw[line width=0.5, decoration={complete sines,amplitude=2, segment length=11},
        decorate,
        rounded corners=0.7] (Q)--(X);
        
    
    \draw[fill] (X) circle [radius=0.1];
 	\draw[fill] (Y) circle [radius=0.1];
 	\draw[fill] (Q) circle [radius=0.1];
 	\draw[fill] (R) circle [radius=0.1];
    
    \end{tikzpicture}\hfil}
    \caption{A picture of Fact \ref{fact:closedanc}}
    \label{fig:Zjs}
\end{figure}

\vspace{0.2cm}

\noindent \underline{Case 1.1:} $I \in \cC_{V',M'}$ such that $I\cap M'=\emptyset$.

As seen before, we have that $\Psi(I)=I$. By condition ($\star\star$) of Definition \ref{def:closed} there is no vertex $x\in X\setminus I$ such that $u_I$ is an ancestor of $x$. If there is a descendant of $u_I$ in $V'\setminus I$, then the descendant is in the set $V'\setminus X=M'\setminus P$. Since $I\cap M'=\emptyset$, Fact \ref{fact:closedanc}, applied to the closed interval $M'$, implies that for every $y\in I'$ and $x\in M'$ we have $a(x,y)=a(y,u_{M'})$. Thus, if $u_I$ is an ancestor of some $x \in M'$, then $u_I$ is an ancestor of $u_{M'}$. This implies that $u_I$ is an ancestor for the entire set $M'$ and in particular of $P$, which contradicts the fact that $I$ is closed in $X$. Therefore, $I$ is a closed interval in $V'$.

\vspace{0.2cm}

\noindent \underline{Case 1.2:} $I \in \cC_{V',M'}$ such that $M'\subseteq I$.

Suppose that $I=K_I\cup M'$ is a interval in $V'$ containing $M'$. We need to prove that $I=K_I\cup M'$ is closed in $V'$ if $\Psi(I)=I\cap X=K_I\cup P$ is closed in $X$. If $K_I=\emptyset$, then $I=M'$ which is by assumption closed in $V'$. Otherwise, we claim that $u_I=u_{\Psi(I)}$. That is $I$ and $\Psi(I)$ have the same common ancestor.

The assumption that $K_I\neq \emptyset$ gives us that either $x_I<\min(M')$ or $y_I>\max(M')$. Assume without loss of generality that $y_I>\max(M')$. Thus, $y_I\in K_I$ and we have that $y_I=y_{\Psi(I)}=\max(K_I)$. If $x_I \notin M'$, then similarly we have $x_I=x_{\Psi(I)}$ and consequently $u_I=a(x_I,y_I)=a(x_{\Psi(I)},y_{\Psi(I)})=u_{\Psi(I)}$. Now if $x_I \in M'$, then $x_I\notin K_I$. This implies that $x_{\Psi(I)} \in P\subseteq M'$. Since both $x_I,x_{\Psi(I)} \in M'$ and $y_I=y_{\Psi(I)}\notin M'$, by Fact \ref{fact:closedanc} we obtain that $u_I=a(x_I,y_I)=a(u_{M'},y_I)=a(u_{M'},y_{\Psi(I)})=a(x_{\Psi(I)},y_{\Psi(I)})=u_{\Psi(I)}$. Hence, $I=K_I\cup M'$ and $\Psi(I)=K_I\cup P$ have the same common ancestor.

To finish the proof note that by condition ($\star\star$) of Definition \ref{def:closed} there are no descendants of $u_{\Psi(I)}$ in $X\setminus \Psi(I)$. Since $u_I=u_{\Psi(I)}$ and $V\setminus I'=K\setminus K_I=X\setminus \Psi(I)$, we conclude that there are no descendants of $u_I$ in $V'\setminus I$ and consequently $I$ is closed in $V'$.

\vspace{0.2cm}

\noindent \underline{Case 2:} $M'\subseteq Q$ for some maximal comb $Q=A^Q\cup B^Q$ in $V'$ with $M'\subseteq B^Q$.

We may assume without loss of generality that $Q$ is a maximal left comb. Let $K_0^Q=K_0\cap Q$ and $K_1^Q=K_1\cap Q$. Clearly $Q=K_0^Q\cup M'\cup K_1^Q$ with $K_0^Q<M'<K_1^Q$. Moreover, $Q=A^Q\cup B^Q$ with $A^Q<B^Q$ and $M'\subseteq B^Q$. Thus, $A^Q\subseteq K_0^Q$ and by condition (a3*) of Definition \ref{def:comb}, we obtain that $K_0^Q$ is closed in $V'$. As in the first case, we split into two cases depending on the type of the interval.

\vspace{0.2cm}

\noindent \underline{Case 2.1:} $I \in \cC_{V',M'}$ such that $I\cap M'=\emptyset$.

Suppose that $I$ is a closed interval in $X$. We claim that $V'(u_I)\cap M' = \emptyset$, i.e., the descendants of $u_I$ are disjoint of $M'$. Applying Proposition \ref{prop:closed} to the closed interval $V'(u_I)$ and maximal comb $Q$ gives us that either $V'(u_I)\cap Q=\emptyset$, $V'(u_I)\subseteq Q$ or $Q\subseteq V'(u_I)$. If $V'(u_I)\cap Q=\emptyset$, then we immediately obtain that $V'(u_I)\cap M'=\emptyset$, since $M'\subseteq Q$. If $Q\subseteq V'(u_I)$, then $M'\subseteq V'(u_I)$ and consequently $P=M'\cap X\subseteq V'(u_I)\cap X=X(u_I)$. This implies that $X(u_I)\neq I$, which contradicts $I$ being closed in $X$.

Thus, we may assume that $V'(u_I)\subseteq Q$ and $M' \not\subseteq V'(u_I)$. Then, by Proposition \ref{prop:maximalcomb}, we have that $V'(u_I)=A^{V'(u_i)}\cup B^{V'(u_I)}$ where either $V'(u_I) \subseteq A^Q$ or $A^{V'(u_i)}=A^Q$ and $B^{V'(u_I)}\subseteq B^Q$. For the first case, note that $A^Q\cap M'=\emptyset$ and therefore $V'(u_I)\cap M'=\emptyset$. For the second case, note that since $M'\not\subseteq V'(u_I)$, then $M'\not\subseteq B^{V'(u_I)}$. This implies that $V'(u_I)\subseteq K_0\cup M'$. Together with the fact that $I\cap M'=\emptyset$ and $I\subseteq V'(u_I)$, we obtain that $I\subseteq K_0^Q$. Since $K_0^Q$ is closed in $V'$, we have that the common ancestor $u_{K_0^Q}=a(\min(K_0^Q),\max(K_0^Q))$ is an ancestor of the entire $I$ and therefore of $u_I$. Hence, $V'(u_I)\subseteq K_0^Q$, which implies that $V'(u_I)\cap M'=\emptyset$. The fact that $I$ is closed in $V'$ now follows because $V'(u_I)=X(u_I)=I$.

\begin{figure}[h]
\centering
{\hfil \begin{tikzpicture}[scale=0.5]
    
 	\draw (-5.5,9.5) ellipse (0.833 and 0.3);
 	\draw (5.5,9.5) ellipse (0.833 and 0.3);
 	\draw (-3.33,9.5) ellipse (1 and 0.3);
 	\draw[red] (1.33,9.5) ellipse (2 and 0.4);
 	
 	\coordinate (T13) at (-4.66,9.5);
 	\coordinate (T14) at (-6.33,9.5);
 	\coordinate (T15) at (4.66,9.5);
 	\coordinate (T16) at (6.33,9.5);
 	\coordinate (T11) at (-4.33,9.5);
 	\coordinate (T12) at (-3.66,9.5);
 	\coordinate (T10) at (-2.33,9.5);
 	\coordinate (T9) at (-1.66,9.5);
 	\coordinate (T8) at (-1,9.5);
 	\coordinate (T7) at (-0.33,9.5);
 	\coordinate (T6) at (0.33,9.5);
 	\coordinate (T5) at (1,9.5);
 	\coordinate (T4) at (1.66,9.5);
 	\coordinate (T3) at (2.33,9.5);
 	\coordinate (T2) at (3,9.5);
 	\coordinate (T1) at (3.66,9.5);
 	\coordinate (T0) at (4.33,9.5);
    \coordinate (R10) at (-3.33,8);
    \coordinate (R9) at (-3,7.5);
    \coordinate (R8) at (-2.66,7);
    \coordinate (R7) at (-2.33,6.5);
    \coordinate (R6) at (-2,6);
    \coordinate (R5) at (-1.66,5.5);
    \coordinate (R4) at (-1.33,5);
    \coordinate (R3) at (-1,4.5);
    \coordinate (R2) at (-0.66,4);
    \coordinate (R1) at (-0.33,3.5);
    \coordinate (R0) at (0,3);
    \coordinate (R') at (1,1.5);
    
    \node (q1) at (-2.65,10.7) [above,font=\small] {$K_0^Q$};
	\node (q2) at (1.35,10.7) [above,font=\small] {$M'$};
	\node (q3) at (3.95,10.7) [above,font=\small] {$K_1^Q$};
	\node (q4) at (-3.65,12.5) [above,font=\small] {$K_0$};
	\node (q5) at (4.9,12.5) [above,font=\small] {$K_1$};
    
    \draw [decorate,
    decoration = {brace, amplitude=7pt}] (-4.2,10.2) --  (-1.1,10.2);
    \draw [decorate,
    decoration = {brace, amplitude=7pt}] (-0.5,10.2) --  (3.2,10.2);
    \draw [decorate,
    decoration = {brace, amplitude=4pt}] (3.6,10.2) --  (4.3,10.2);
    \draw [decorate,
    decoration = {brace, amplitude=7pt}] (-6.2,12) --  (-1.1,12);
    \draw [decorate,
    decoration = {brace, amplitude=7pt}] (3.6,12) --  (6.2,12);
    
    \draw[line width=0.5] (R0)--(R1)--(R2);
    \draw[line width=0.5][red] (R2)--(R3)--(R4)--(R5);
    \draw[line width=0.5][red] (R5)--(R6)--(R7);
    \draw[line width=0.5] (R7)--(R8)--(R9)--(R10);
    \draw[line width=0.5] (T0)--(R0);
    \draw[line width=0.5] (T1)--(R1);
    \draw[line width=0.5][red] (T2)--(R2);
    \draw[line width=0.5][red] (T3)--(R3);
    \draw[line width=0.5][red] (T4)--(R4);
    \draw[line width=0.5][red] (T5)--(R5);
    \draw[line width=0.5][red] (T6)--(R6);
    \draw[line width=0.5][red] (T7)--(R7);
    \draw[line width=0.5] (T8)--(R8);
    \draw[line width=0.5] (T9)--(R9);
    \draw[line width=0.5] (T10)--(R10);
    \draw[line width=0.5] (T11)--(R10);

    \end{tikzpicture}\hfil}
    \caption{Maximal comb $Q$ and sets $K_0^Q$ and $K_1^Q$.}
    \label{fig:maxQ}
\end{figure}
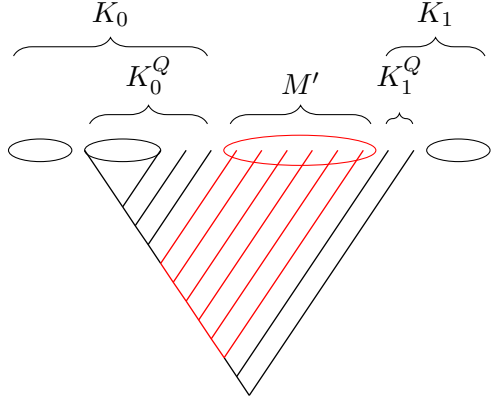

\vspace{0.2cm}

\noindent \underline{Case 2.2:} $I \in \cC_{V',M'}$ such that $M'\subseteq I$.

Let $I=K_0^I\cup M'\cup K_1^I$ be an interval in $V'$ containing $M'$ with $K_0^I\subseteq K_0$ and $K_1^I\subseteq K_1$. Suppose that $\psi(I)=K_0^I\cup P\cup K_1^I$ is closed in $X$. Since $K_0^Q\cup M'$ is closed, by the same argument of Case 1.2 (by considering $K_0^Q\cup M'$ instead of $M'$), we can show that if $x_I<\min(K_0^Q)$ or $y_I > \max(M')$, then $u_I=a(x_I,y_I)=a(x_{\Psi(I)},y_{\Psi(I)})=u_{\Psi(I)}$ and consequently $I$ is closed in $V'$.

Now suppose that $\min(K_0^Q)\leq x_I \leq \min(M')$ and $y_I=\max(M')$. Since both $M'$ and $P$ are not closed intervals in their respective ground sets, we have that $x_I\neq \min(M')$ and consequently $x_{\Psi(I)}=x_I$ and $y_{\Psi(I)}=\max(P)$. Hence, in this case, $K_0^I\subseteq K_0^Q$ and $K_I=\emptyset$, which implies that $I=K_0^I\cup M'$ and $\Psi(I)=K_0^I\cup P$. Because $Q$ is a maximal left comb with $M'\subseteq Q$, then both sets $K_0^Q$ and $K_0^Q\cup M'$ are closed in $V'$. Therefore, by Proposition \ref{prop:backclosed} the intervals $K_0^Q$ and $K_0^Q\cup P$ are closed in $X$. Fact \ref{fact:closedanc} applied to $K_0^Q$ gives us that $a(z,y_{\Psi(I)})=a(z',y_{\Psi(I)})$ for every $z,z' \in K_0^Q$. This implies that $u_{\Psi(I)}=a(x_{\Psi(I)},y_{\Psi(I)})=a(\min(K_0^Q),y_{\Psi(I)})$, i.e., $K_0^Q\cup P$ and $\Psi(I)=K_0^I\cup P$ have $u_{\Psi(I)}$ as the same common ancestor. Since $K_0^Q\cup P$ and $K_0^I\cup P$ are both closed in $X$, we obtain that $K_0^Q=K_0^I$. Thus $I=K_0^Q\cup M'$, which is closed in $V'$.
\end{proof}

The next result shows that we can always find in an edge the location of the maximal comb with coloring data containing $a(P)$. This will be extremely important, since the comb will be the only maximal comb such that coloring data changes while we run through different edges of $H'$.

\begin{proposition}\label{prop:uniquecomb}
Let $H'$ be a fixed pre-processed daisy obtained by Lemma \ref{lem:preprocessing}. There exists a unique interval $J \subseteq [k+r]$ such that for every edge $X=K_0\cup P\cup K_1=\{x_1,\ldots,x_{k+r}\}$ in $H'$, the interval $X_J=\{x_j\}_{j\in J}$ is a maximal comb of type depending only on $H'$ with 
\begin{align*}
    a(P)\subseteq F(X_J).
\end{align*}
Moreover, writing $X_J=A^{X_J}\cup B^{X_J}$ we have one of the following:
\begin{enumerate}
    \item If $H'$ satisfies statement $(1)$ of Lemma \ref{lem:preprocessing}, then $A^{X_J}\subseteq P$ and $X_J$ is the smallest maximal comb containing $P$ with non-empty coloring data.
    \item If $H'$ satisfies statement $(2)$ of Lemma \ref{lem:preprocessing}, then $X_J=I\cap X$, where $I=A\cup B$ is the maximal comb in $V'$ such that $M'\subseteq B$, and $X_J$ satisfies $P\subseteq B^{X_J}$.
\end{enumerate}
\end{proposition}

\begin{proof}
The idea of the proof is to identify certain maximal combs in $V'$ with maximal combs in an edge $X$. Because the structure of those maximal combs in $V'$ only depends on $H'$, we will obtain the same for the corresponding combs in $X$. Proposition \ref{prop:preserveclosed} will be useful here, since by condition (a3*) and (b3*) of Definition \ref{def:comb} a comb can be defined by looking at certain closed subintervals. The proof is split into cases depending on the structure of the tree $T_{V'}$

\vspace{0.2cm}

\noindent \underline{Case 1:} $M'$ is a closed interval in $V'$.

We will construct a maximal comb in $X$ by looking at a maximal comb in $V'$ containing $M'$. Write $K_0=\{x_1,\ldots,x_{k_0}\}$, $M'=\{y_1,\ldots, y_{m'}\}$ and $K_1=\{z_1,\ldots,z_{k_1}\}$. There are two possibilities here:

\vspace{0.2cm}

\noindent \underline{Case 1.1:} Either $M'\cup \{z_1\}$ is a closed interval in $V'$ or $M'\cup\{x_{k_0}\}$ is a closed interval in $V'$.

Suppose without loss of generality that $M'\cup\{z_1\}$ is closed in $V'$. In this case $M'\cup \{z_1\}$ is a left comb. Let $M'\cup\{z_1,\ldots,z_t\}$ be the maximal left comb obtained by extending $M'\cup \{z_1\}$. We will assume during the entire proof that $t<k_1$. For $t=k_1$, the same proof work by removing any claims and sets involving $z_{t+1}$. By condition (a3*) of Definition \ref{def:comb} and Definition \ref{def:maximalcomb}, $M'\cup\{z_1,\ldots,z_t\}$ being a maximal left comb is the same as saying that the intervals $M'$ and $M'\cup\{z_1,\ldots,z_i\}$ are closed for every $1\leq i \leq t$, but the interval $M'\cup \{z_1,\ldots,z_{t+1}\}$ is not closed.

Set $J=\{k_0+1,\ldots,k_0+r+t\}$. Let $X$ be an edge of $H'$ with petal $P$. We claim that $X_J$ is a maximal left comb in $X$ with $A^{X_J}\subseteq P$. To see that consider the intervals
\begin{align*}
    J_i=\{k_0+1,\ldots,k_0+r+i\}, \quad 0\leq i \leq t+1. 
\end{align*}
In particular $J_t=J$. Note that $X_{J_0}=M'\cap X=P$ and $X_{J_i}=(M'\cap\{z_1,\ldots,z_i\})\cap X$ for $1\leq i \leq t+1$. Thus, by applying Proposition \ref{prop:preserveclosed} with $I=M'$ and $I=M'\cup\{z_1,\ldots,z_i\}$, we obtain that $X_{J_i}$ is closed in $X$ for $0\leq i \leq t$ and $X_{J_{t+1}}$ is not closed in $X$. Hence, by condition (a3*) of Definition \ref{def:comb} and Definition \ref{def:maximalcomb}, we have that $X_J=X_{J_t}$ is a maximal left comb. Since $P=X_{J_0}\subseteq X_{J_1}\subseteq \ldots \subseteq X_{J_t}=X_J$ are all closed intervals, we have that $A^{X_J}\subseteq P$. Thus, $|A^{X_J}|\leq |P| =r$ and we have that either $X_J$ is a maximal comb of type 3 or type 4 depending on the size of $|X_J|=r+t$. Because $t$ is a parameter that depends on the size of the maximal comb in $V'$, i.e., on the structure of $H'$, we conclude that the type of $X_J$ is independent of our choice of edge $X$.

It remains to show that $a(P)\subseteq F(X_J)$ and $X_J$ is the smallest maximal comb containing $P$ with non-empty data coloring. For the first, note that $F(X_J)=a(X_J)$ because $X_J$ is of type $3$ or $4$. Thus, $a(P)\subseteq a(X_J)=F(X_J)$. For the latter, note that the only potential maximal comb smaller than $X_J$ containing $P$ is $P$ itself. However, if $P$ is a maximal comb, then it is a comb of type $2$ and therefore $F(P)=\emptyset$. Hence, $X_J$ is the smallest maximal comb containing $P$ with non-empty coloring data.

\vspace{0.2cm}

\noindent \underline{Case 1.2:} Both $M'\cup \{z_1\}$ and $M'\cup\{x_{k_0}\}$ are not closed in $V'$.

By Definition \ref{def:maximalcomb}, $M'$ is a maximal comb. Set $J=\{k_0+1,\ldots,k_0+r\}$. Note that $X_J=P$. By Proposition \ref{prop:preserveclosed}, the set $P=M'\cap X$ is closed in $X$ and $P\cup \{z_1\}=(M'\cup\{z_1\})\cap X$ and $P\cup\{x_{k_0}\}=(M'\cup \{x_{k_0}\})\cap X$ are not closed in $X$. Thus, $P$ is a maximal comb in $X$. It is clear that $A^{X_J}\subseteq X_J=P$. Since $|P|=r$ and $P\cup\{z_1\}$, $P\cup\{x_{k_0}\}$ are not closed, we have that $X_J=P$ is of type $1$. Therefore, the type of $X_J$ does not depend on $X$. Moreover, the fact that $X_J$ is of type 1 gives us that $a(P)=a(X_J)=F(X_J)$. The minimality of $X_J$ is immediate from the fact that all combs with non-empty data has size at least $r$.

\vspace{0.2cm}

\noindent \underline{Case 2:} $M'\subseteq B$ for a maximal comb $I=A\cup B$ in $V'$.

Suppose without loss of generality that $I=A\cup B$ is a maximal left comb. Let $A=\{x_{k_0-p-\ell+1},\ldots,x_{k_0-p}\}$, $B\cap K_0=\{x_{k_0-p+1},\ldots,x_{k_0}\}$ and $B\cap K_1=\{z_1,\ldots,z_t\}$. Set $J=\{k_0-p-\ell+1,\ldots,k_0+r+t\}$. Let $X$ be an edge of $H'$ with petal $P$. Clearly, $X_J=I\cap X$. We claim that $X_J$ is a maximal left comb with $P\subseteq B^{X_J}$. By Definition \ref{def:comb} and \ref{def:maximalcomb} we have that $A$ is a closed interval in $V'$, $A\setminus\{x_{k_0-p}\}$ and $I\cup\{z_{t+1}\}$ are not closed in $V'$ and
\begin{align*}
\delta(x_{k_0-p},x_{k_0-p+1})&>\ldots>\delta(x_{k_0-1},x_{k_0})>\delta(x_{k_0},y_1)>\delta(y_1,y_2)>\ldots>\delta(y_{m'-1},y_{m'})\\
&> \delta(y_{m'},z_1)>\delta(z_1,z_2)>\ldots>\delta(z_{t-1},z_t).
\end{align*}

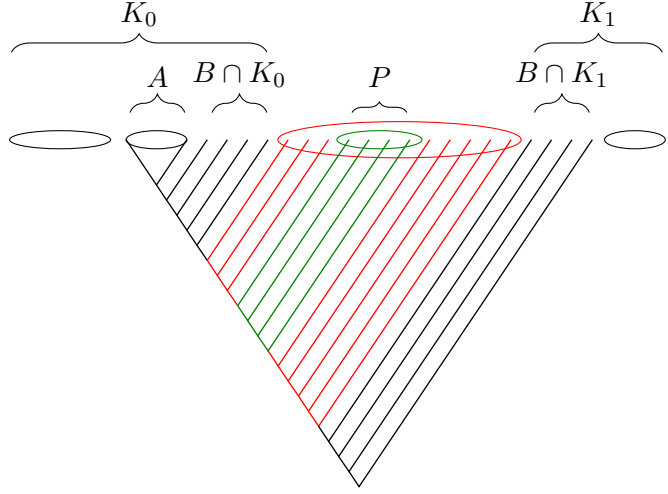
\begin{figure}[h]
\centering
{\hfil \begin{tikzpicture}[scale=0.4]
    
 	\draw (-8.5,9.5) ellipse (1.66 and 0.3);
 	\draw (-5.33,9.5) ellipse (1 and 0.3);
 	\draw (10.4,9.5) ellipse (1 and 0.3);
 	\draw[red] (2.66,9.5) ellipse (4 and 0.6);
 	\draw[green1] (2,9.5) ellipse (1.4 and 0.3);
 	
 	\coordinate (T0) at (-6.33,9.5);
 	\coordinate (T1) at (-4.33,9.5);
 	\coordinate (T2) at (-3.66,9.5);
 	\coordinate (T3) at (-3,9.5);
 	\coordinate (T4) at (-2.33,9.5);
 	\coordinate (T5) at (-1.66,9.5);
 	\coordinate (T6) at (-1,9.5);
 	\coordinate (T7) at (-0.33,9.5);
 	\coordinate (T8) at (0.33,9.5);
 	\coordinate (T9) at (1,9.5);
 	\coordinate (T10) at (1.66,9.5);
 	\coordinate (T11) at (2.33,9.5);
 	\coordinate (T12) at (3,9.5);
 	\coordinate (T13) at (3.66,9.5);
 	\coordinate (T14) at (4.33,9.5);
 	\coordinate (T15) at (5,9.5);
 	\coordinate (T16) at (5.66,9.5);
 	\coordinate (T17) at (6.33,9.5);
 	\coordinate (T18) at (7,9.5);
 	\coordinate (T19) at (7.66,9.5);
 	\coordinate (T20) at (8.33,9.5);
 	\coordinate (T21) at (9,9.5);
    
    \coordinate (R1) at (-5.33,8);
    \coordinate (R2) at (-5,7.5);
    \coordinate (R3) at (-4.66,7);
    \coordinate (R4) at (-4.33,6.5);
    \coordinate (R5) at (-4,6);
    \coordinate (R6) at (-3.66,5.5);
    \coordinate (R7) at (-3.33,5);
    \coordinate (R8) at (-3,4.5);
    \coordinate (R9) at (-2.66,4);
    \coordinate (R10) at (-2.33,3.5);
    \coordinate (R11) at (-2,3);
    \coordinate (R12) at (-1.66,2.5);
    \coordinate (R13) at (-1.33,2);
    \coordinate (R14) at (-1,1.5);
    \coordinate (R15) at (-0.66,1);
    \coordinate (R16) at (-0.33,0.5);
    \coordinate (R17) at (0,0);
    \coordinate (R18) at (0.33,-0.5);
    \coordinate (R19) at (0.66,-1);
    \coordinate (R20) at (1,-1.5);
    \coordinate (R21) at (1.33,-2);

    \node (a) at (-5.3,10.9) [above,font=\small] {$A$};
	\node (b0) at (-2.6,10.9) [above,font=\small] {$B\cap K_0$};
	\node (p) at (2,10.9) [above,font=\small] {$P$};
	\node (b1) at (8,10.9) [above,font=\small] {$B\cap K_1$};
	\node (k0) at (-5.9,12.9) [above,font=\small] {$K_0$};
	\node (k1) at (9.2,12.9) [above,font=\small] {$K_1$};
    
    \draw [decorate,
    decoration = {brace, amplitude=7pt}] (-6.2,10.4) --  (-4.4,10.4);
    \draw [decorate,
    decoration = {brace, amplitude=7pt}] (-3.5,10.4) --  (-1.7,10.4);
    \draw [decorate,
    decoration = {brace, amplitude=4pt}] (1.1,10.4) --  (2.9,10.4);
    \draw [decorate,
    decoration = {brace, amplitude=7pt}] (7.1,10.4) --  (8.9,10.4);
    \draw [decorate,
    decoration = {brace, amplitude=7pt}] (-10.1,12.4) --  (-1.7,12.4);
    \draw [decorate,
    decoration = {brace, amplitude=7pt}] (7.1,12.4) --  (11.3,12.4);
    
    \draw[line width=0.5] (T1)--(R1)--(T0);
    \draw[line width=0.5] (T2)--(R2)--(R1);
    \draw[line width=0.5] (T3)--(R3)--(R2);
    \draw[line width=0.5] (T4)--(R4)--(R3);
    \draw[line width=0.5] (T5)--(R5)--(R4);
    \draw[line width=0.5] (R5)--(R6);
    \draw[line width=0.5][red] (T6)--(R6);
    \draw[line width=0.5][red] (T7)--(R7)--(R6);
    \draw[line width=0.5][red] (T8)--(R8)--(R7);
    \draw[line width=0.5][red] (R8)--(R9);
    \draw[line width=0.5][green1] (T9)--(R9);
    \draw[line width=0.5][green1] (T10)--(R10)--(R9);
    \draw[line width=0.5][green1] (T11)--(R11)--(R10);
    \draw[line width=0.5][green1] (T12)--(R12)--(R11);
    \draw[line width=0.5][red] (T13)--(R13)--(R12);
    \draw[line width=0.5][red] (T14)--(R14)--(R13);
    \draw[line width=0.5][red] (T15)--(R15)--(R14);
    \draw[line width=0.5][red] (T16)--(R16)--(R15);
    \draw[line width=0.5][red] (T17)--(R17)--(R16);
    \draw[line width=0.5] (T18)--(R18)--(R17);
    \draw[line width=0.5] (T19)--(R19)--(R18);
    \draw[line width=0.5] (T20)--(R20)--(R19);
    \draw[line width=0.5] (T21)--(R21)--(R20);

    \end{tikzpicture}\hfil}
    \caption{Case $2$ of Proposition \ref{prop:uniquecomb}}
    \label{fig:maxQ}
\end{figure}

Let $P=\{y_{i_1},\ldots,y_{i_r}\}$. By Proposition \ref{prop:preserveclosed}, the set $A=A\cap X$ is closed in $X$ and the sets $A\setminus \{x_{k_0-p}\}=(A\setminus \{x_{k_0-p}\})\cap X$ and $X_J\cup \{z_{t_1}\}=(I\cup \{z_{t+1}\})\cap X$ are not closed in $X$. Moreover, since $P\subseteq M'$, we have that
\begin{align*}
\delta(x_{k_0-p},x_{k_0-p+1})&>\ldots>\delta(x_{k_0-1},x_{k_0})>\delta(x_{k_0},y_{i_1})>\delta(y_{i_1},y_{i_2})>\ldots>\delta(y_{i_{r-1}},y_{i_r})\\
&> \delta(y_{i_r},z_1)>\delta(z_1,z_2)>\ldots>\delta(z_{t-1},z_t).
\end{align*}
Thus, by Definition \ref{def:comb} and Definition \ref{def:maximalcomb}, the interval $X_J$ is a maximal left comb in $X$ with $A^{X_J}=A$ and $|B^{X_J}|=r+t+p$. Since $A\subseteq K_0$, we obtain that $P\subseteq B^{X_J}$. Note that to determine the type of $X_J$ we need to know the sizes of $X_J$, $A^{X_J}$ and $B^{X_J}$. None of this parameters depends on the choice of $X$. Hence, the type of $X_J$ is independent of $X$. Finally, because $|B^{X_J}|\geq |P|\geq r$, we obtain that $|X_J|\geq r+1$ and consequently the comb $X_J$ is of type $3$, $4$ or $5$. If it is of type $3$ or $4$, then $a(P)\subseteq a(X_J)=F(X_J)$. If it is of type $5$, then $a(P)\subseteq a(\max(A^{X_J})\cup B^{X_J})=F(X_J)$.
\end{proof}

To finish the section we prove that the maximal comb determined by the set $J$ is the comb that essentially determines the color of the entire edge.

\begin{proposition}\label{prop:simplecolor}
Let $X=\{x_1,\ldots,x_{k+r}\}$, $X'=\{x'_1,\ldots,x'_{k+r}\}$ be two edges in $H'$ and let $X_J=\{x_j\}_{j\in J}$, $X'_J=\{x'_j\}_{j\in J}$. If $\chi(X)=\chi(X')$, then $\chi_0(X_J)=\chi_0(X'_J)$.
\end{proposition}

\begin{proof}
Let $P, P'\subseteq M'$ be the petals of $X$ and $X'$, respectively. By definition, $\chi(X)=\chi(X')$ implies that
\begin{align*}
    \sum_{I\in \cI_X}\chi_0(I)=\sum_{I'\in \cI_{X'}}\chi_0(I') \pmod{2}.
\end{align*}
By the definition of coloring data, if $F(I)=\emptyset$, then $\chi_0(I)=0$. Thus, we may rewrite the equality above as
\begin{align}\label{eq1}
    \sum_{\substack{I\in\cI_X\\F(I)\neq\emptyset}}\chi_0(I)=\sum_{\substack{I'\in \cI_{X'}\\F(I')\neq\emptyset}} \chi_0(I) \pmod{2}.
\end{align}

We claim that if $I=A\cup B$ is a maximal comb of $X$ with $F(I)\neq\emptyset$, then either $I\cap P=\emptyset$ or $P\subseteq I$. Note that, in the coloring defined in Subsection \ref{subsec:variant}, whenever $F(I)=\emptyset$, we have that $|I|\geq |P|=r$. By Lemma \ref{lem:preprocessing}, the daisy $H'$ satisfies one of the following conditions: Either $M'$ is a closed interval in $V':=V(H')$ or there exists a maximal comb $Q=A^Q\cup B^Q$ such that $M'\subseteq B^Q$. If $M'$ is closed in $V'$, then by Proposition \ref{prop:preserveclosed} the petal $P=M'\cap X$ is a closed interval in $X$. Thus, Proposition \ref{prop:closed} applied to the closed intervals $I$ and $P$ gives the desired result that either $I\cap P=\emptyset$ or $P\subseteq I$. Now suppose that we are in Condition (2) of Lemma \ref{lem:preprocessing}. By Proposition \ref{prop:uniquecomb}, we have that $P\subseteq B^{X_J}$, where $B^{X_J}$ is the ``teeth" part of the comb $X_J=A^{X_J}\cup B^{X_J}$. Thus, Proposition \ref{prop:maximalcomb} applied to the maximal combs $I$ and $X_J$ implies that $I\cap X_J=\emptyset$, $I\subseteq A^{X_J}$ or $X_J\subseteq A \subseteq I$. In the first two cases we obtain $I\cap P=\emptyset$, while in the latter we have $P\subseteq I$.

The idea of the proof of Proposition \ref{prop:simplecolor} is to show that there exists a bijection between $\{I\in \cI_X:\: F(I)\neq \emptyset\}$ and $\{I'\in \cI_{X'}:\: F(I')\neq \emptyset\}$ such that $X_J$ is sent to $X'_J$ and every $I\neq X_J$ is sent to an $I'\neq X'_J$ with $\chi_0(I)=\chi_0(I')$. Hence, after some cancellation, we obtain from equation (\ref{eq1}) that $\chi_0(X_J)=\chi_0(X'_J)$. Based on the last paragraph, we construct such a bijection by splitting $\{I\in \cI_X:\: F(I)\neq \emptyset\}$ into two parts:

\vspace{0.2cm}

\noindent \underline{Case 1:} $I\in \cI_X$ is a maximal comb of $X$ with $F(I)\neq \emptyset$ and $I\cap P=\emptyset$.

We claim that $I\in \cI_{X'}$ is a maximal comb in $X'$ of the same type and consequently $\chi_0(I)$ is the same in $X$ and $X'$. Assume without loss of generality that $I\subseteq K_0$. Let $x$ be the element preceding $\min(I)$ in $X$ (In the case that such $x$ does not exists, we simply take $x=\min(I)$). Let $y$ be the element after $\max(I)$ in $X$. Similarly, define $x'$ as the element before $\min(I)$ in $X'$ and $y'$ as the element after $\max(I)$ in $X'$. Since $I\subseteq K_0$, clearly $x=x'$. However, $y$ and $y'$ are not necessarily the same. By conditions (a3*) and (b3*) of Definition \ref{def:comb} and Definition \ref{def:maximalcomb}, to prove that $I\in \cI_{X'}$ is enough to check that $I\cup\{x\}$, $L\subseteq I$, $I\cup\{y\}$ are closed intervals in $X$ if and only if $I\cup \{x'\}$, $L\subseteq I$ and $I\cup\{y'\}$ are closed intervals in $X'$, respectively. Since $F(I)=\neq \emptyset$, we have that $I$ is of type 1, 3, 4 or 5. Note that one can distinguish between this types by determining the size of the ``handle'' and `teeth''of $I$. Thus, by checking the properties above, we also obtain that $I$ have the same type in $X$ and $X'$.

For an interval $L\subseteq I$, by Proposition \ref{prop:preserveclosed} we have that $L=L\cap X$ is a closed interval in $X$ if and only if it is a closed interval in $V'$. Another application of Proposition \ref{prop:preserveclosed} gives us that $L=L\cap X'$ is a closed interval if it is closed in $V'$. Hence, $L$ is closed in $X$ if and only if it is closed in $X'$. Similarly, the same argument works for $I\cup\{x\}$ and $I\cup \{x'\}$, because $x=x'\notin M'$. Moreover, if $y\in K_0$, then $y'=y\notin M'$ and we also obtain that $I\cup \{y\}$ is closed in $X$ if and only if $I\cup\{y'\}$ is closed in $X'$. Hence, the only case remaining is when $y\notin K_0$, i,e, $y=\min(P)$ and $y'=\min(P')$. 

We split the argument into two cases depending on the structure given by Lemma \ref{lem:preprocessing}. Suppose that $M'$ is closed in $V'$ and let $u=a(\min(M'),\max(M'))$ be the common ancestor of $M'$. By Fact \ref{fact:closedanc}, we have that $a(z,y)=a(z,y')=a(z,u)$ for every $z\in I$. Therefore, the entire set $M'$ is descendant of the common ancestors of $I\cup\{y\}$ and $I\cup \{y'\}$, which by condition (ii*) of Definition \ref{def:closed} implies that both sets are not closed. Now suppose that $M'\subseteq B^Q$ for some maximal comb $Q=A^Q\cup B^Q$. Since $I$ is a closed interval in $X$, then by Proposition \ref{prop:preserveclosed} it is a closed interval in $V'$. Thus, by Proposition \ref{prop:closed}, applied to $I$ and the maximal comb $Q$, one of the following three possibilities holds: $I\cap Q=\emptyset$, $I\subseteq Q$ or $Q\subseteq I$. Clearly, the last possibility cannot hold, since $P\subseteq Q$ and $I\cap P=\emptyset$. Suppose that $I\cap Q=\emptyset$. By Proposition \ref{prop:preserveclosed}, the interval $Q\cap X$ is a closed interval in $X$. Since $(I\cup \{y\})\cap (Q\cap X)=\{y\}\neq \emptyset$ and $\{y\}\neq Q\cap X$, we obtain by Proposition \ref{prop:closed} that $I\cup\{y\}$ is not closed in $X$. Similarly, $I\cup\{y'\}$ is not closed in $X'$.

Now we handle with the case that $I\subseteq Q$. By Proposition \ref{prop:maximalcomb}, either $I\subseteq A^Q$ or $I=A\cup B$ is a comb with $A=A^Q$ and $B\subseteq B^Q$. Since $I\subseteq K_0$, we have that $Q$ is a maximal left comb. Let $K_0^Q=K_0\cap Q$, $B^Q=\{z_1,\ldots,z_b\}$ and let $Q_{z_i}=A^Q\cup\{z_1,\ldots,z_i\}$ be the subcomb of $Q$ ending in $z_i$. By condition (a3*) of Definition \ref{def:comb}, we have that $A^Q$ and $Q_z$ are closed in $V'$ for every $z\in B^Q$. Moreover, note that $\min(I)\in A^Q$. Let $v=a(\min(A^Q),\max(A^Q)$ and $w=a(\min(I),y)$ be the common ancestor of $A^Q$ and $I\cup\{y\}$, respectively. By Fact \ref{fact:closedanc} applied to $A^Q$, we have that $a(\min(I),x)=a(v,x)=a(\min(A^Q),x)$ for every $x\in M'$. Thus, $X(w)=V'(w)\cap X=Q_y\cap X=K_0^Q\cup\{y\}$, which implies that $I\cup\{y\}$ is a closed interval in $X$ if and only if $I=K_0^Q$. Similarly, $I\cup \{y'\}$ is a closed interval in $X'$ if and only if $I=K_0^Q$. Hence, $I\cup\{y\}$ is closed in $X$ if and only if $I\cup\{y'\}$ is closed in $X'$.

\vspace{0.2cm}

\noindent \underline{Case 2:} $I\in \cI_X$ is a maximal comb of $X$ with $F(I)\neq \emptyset$, $P\subseteq I$ and $I\neq X_J$.

In this case, by Proposition \ref{prop:uniqueinfo} and Proposition \ref{prop:uniquecomb}, we have that $F(I)\cap a(P)=\emptyset$ and consequently $F(I)\neq a(I)$. Thus, by our coloring, we obtain that $I$ is of type $5$, i.e., $I=A\cup B$ is a maximal left/right comb with $|A|>r$ and $|B|\geq r$. We may assume that $I$ is a maximal left comb. Hence, $F(I)=a(\{\max(A)\}\cup B)$ and the fact that $F(I)\cap a(P)=\emptyset$ implies that $P\subseteq A$. Write $A=K_A\cup P$ with $K_A\subseteq K_0\cup K_1$. We claim that $I'=K_A\cup P'\cup B$ is a maximal comb of type $5$ with set of ``teeth" $B'=B$ and ``handle'' $A'=K_A\cup P'$.

Write $B=\{y_1,\ldots,y_t\}$. Let $x=\max(A)$, $x'=\max(A')$ and let $z$ be the element coming after $B$ in $V'$ (In case that such element does note exist, we take $z=\max(B)$). By condition (a3*) of Definition \ref{def:comb} and Definition \ref{def:maximalcomb} to prove that $I=A'\cup B'$ is a maximal comb of type $5$ with $A'=K_A\cup P$ and $B'=B$ it is enough to prove that $A'$ and $A'\cup \{y_1,\ldots,y_i\}$ are closed in $X'$ for every $1\leq i \leq t$, $A'\setminus\{x'\}$ is not closed in $X'$ and $A'\cup \{y_1,\ldots, y_t,z\}$ is closed if and only if $A\cup\{y_1,\ldots,y_t,z\}$ is closed in $X$. Applying Proposition \ref{prop:preserveclosed} with $X$ and $V'$ and then $V'$ and $X'$ gives us that $A$, $A\cup\{y_1,\ldots,y_i\}$ and $A\cup\{y_1,\ldots,y_t,z\}$ are closed in $X$ if and only if $A'$, $A'\cup \{y_1,\ldots,y_i\}$ and $A'\cup\{y_1,\ldots,y_t,z\}$ are closed in $X'$, respectively. Since $I$ is a maximal comb in $X$, this implies that $A'$, $A'\cup\{y_1,\ldots,y_i\}$ are closed in $X'$ for $1\leq i \leq t$. If $x\notin P$, then $x=x'$ and by the same argument $A'\setminus \{x'\}$ is not closed in $X'$.

It remains to deal with the case that $x\in P$, i.e., $x=\max(P)$ and $x'=\max(P')$. The proof is split into two cases depending on the structure of $H'$ given by Lemma \ref{lem:preprocessing}. If $M'$ is a closed interval in $V'$, then by Proposition \ref{prop:preserveclosed} the interval $P'$ is closed in $X'$. The intersection $A'\setminus \{x'\}$ is proper since $|A'|=|A|>r=|P'|$ and $x'\in P'$. Therefore, by Proposition \ref{prop:closed}, we have that $A'\setminus\{x'\}$ is not closed in $X'$.

Now suppose that $M'\subseteq B^Q$ for some maximal comb $Q=A^Q\cup B^Q$. By Proposition \ref{prop:uniquecomb}, the maximal combs $X_J$ and $X'_J$ satisfies $P\subseteq B^{X_J}$ and $P'\subseteq B^{X'_J}$. Applying Proposition \ref{prop:maximalcomb} to the maximal combs $I$ and $X_J$ gives us that either $I\subseteq A^{X_J}$ or $X_J\subseteq A$. Since $P\cap A^{X_J}=\emptyset$, it follows that $X_J\subseteq A$. Because $x=\max(A)\in P$, we have that $\max(K_A)<\min(P)$. This implies that $X_J$ is a maximal left comb. Hence, by Proposition \ref{prop:uniquecomb} both $Q$ and $X'_J$ are maximal left combs.

We claim that $A'\setminus\{x'\} \cap X'_J$ is a proper intersection. Since $|X'_J|=|X_J|\leq |A|=|A'|$, by Proposition \ref{prop:closed} applied to $A'$ and $X'_J$, we have that $X'_J\subseteq A'$. Note that we already proved for $1\leq i \leq t$ that $A'$ and $A'\cup\{y_1,\ldots,y_i\}$ are closed in $X'$. Hence, by the maximality of $X'_J$ we have that $A'\neq X'_J$ (otherwise we could extend to the left comb $X'_J\cup B$). Thus, $X'_J$ is strictly contained in $A'$, which implies that $K_A\setminus X'_J\neq \emptyset$. This concludes that $A'\setminus\{x'\}\cap X'_J$ is proper and by Proposition \ref{prop:closed} the interval $A\setminus\{x\}$ is not closed in $X'$.

Therefore, the interval $I'=A'\cup B'$ is a maximal left comb in $X'$ of type $5$ with $A'=K_A\cup P'$ and $B'=B$. It is not difficult to check (by Proposition \ref{prop:preserveclosed}) that the correspondence between $I$ and $I'$ is a bijection. Moreover, since $A=K_A\cup P$ is closed in $X$, we obtain by Proposition \ref{prop:preserveclosed} that $K_A\cup M'$ is closed in $V'$. It follows by Fact \ref{fact:closedanc} that $a(x,y_1)=a(x',y_1)$ and consequently thar $F(I)=a(B\cup\{x\})=a(B'\cup\{x'\})=F(I')$. Hence, by Proposition \ref{prop:equalinfo} we have $\chi_0(I)=\chi_0(I')$.
\end{proof}

\section{Main proof}\label{sec:proof}

The proof of Theorem \ref{thm:daisyramsey} follows by a simple induction of the following stepping up theorem.

\begin{theorem}\label{th:daisystepup}
Let $m\geq 100kr^2$, $N=\min_{0\leq j\leq k}\{ D_{r-1}^{\smp}(\frac{1}{5}k^{-1/2}m^{1/2},j)\}$ be integers and let $\{\phi_i\}_{r-1\leq i \leq k+r-1}$ be a family of colorings $\phi_i:[N]^{(i)}\rightarrow \{0,1\}$ without a monochromatic copy of a simple $(r-1,\frac{1}{5}k^{-1/2}m^{1/2},i-r+1)$-daisy. Then, the coloring $\chi:[2^N]^{(k+r)}\rightarrow \{0,1\}$ described in Subsection \ref{subsec:variant} does not contain a monochromatic simple $(r,m,k)$-daisy.
\end{theorem}

\begin{proof}
Suppose by contradiction that there exists a monochromatic simple $(r,m,k)$-daisy $H$ in $[2^N]^{(k+r)}$ with kernel $K=K_0\cup K_1$ of size $k$, universe of petals $M$ of size $m$ and $K_0<M<K_1$. By Lemma \ref{lem:preprocessing}, we obtain a monochromatic simple $(r,\frac{1}{2}k^{-1/2}m^{1/2},k)$-daisy $H'$ with same kernel and universe of petals $M'=\{y_1,\ldots,y_{m'}\} \subseteq M$ of size $m'=\frac{1}{2}k^{-1/2}m^{1/2}$ satisfying that either $M'$ is a closed interval in $V'=V(H')$ or $M'$ is part of the ``teeth" of a maximal comb $I=A\cup B$, i.e., $M'\subseteq B$.

Note that every edge $X\in H'$ can be written in the form $X=K_0\cup P\cup K_1$ where $P\in (M')^{(r)}$ is a petal of $H'$. Since $H'$ is monochromatic, we have that
\begin{align*}
    \chi(X)=\sum_{I\subseteq \cI_X} \chi_0(I) \pmod{2}
\end{align*}
is constant, for every $X \in H'$. Thus, by Propositions \ref{prop:uniquecomb} and \ref{prop:simplecolor}, there exists a unique interval $J\subseteq [k+r]$ such that for every $X \in E(H')$ the interval $X_{J}=\{x_j\}_{j\in J}$ is maximal comb with color $\chi_0(X_J)$ constant. 

As in the proof give in Subsection \ref{subsec:ehrstep}, our goal is to use the fact that the combs $X_J$ are monochromatic with respect to $\chi_0$ to find a large $1$-comb. Let $t=|J|-r$ and let $G$ be the simple $(r,\frac{1}{2}k^{-1/2}m^{1/2},t)$-daisy constructed by taking as edges the combs $X_J$ for every edge $X \in H'$. To be more precise, let $K_J$ be the subset of $t$ vertices of $K_0\cup K_1$ in the interval $J$. Note that every comb $X_J$ can be partitioned into $X_J=K_J\cup P$, where $P\subseteq M'$ is the petal of $X$. We define $G$ as the simple $(r,\frac{1}{2}k^{-1/2}m^{1/2},t)$-daisy given by
\begin{align*}
    V(G)&=K_J\cup M'\\
    G&=\{X_{J}:\: X\in H'\}  
\end{align*}
As discussed in the last paragraph the $(t+r)$-graph $G$ is monochromatic under the coloring $\chi_0$. The following lemma is a variant of Proposition \ref{prop:findcomb} for simple daisies.

\begin{proposition}\label{prop: findcombdaisy} 
If $M'$ is a closed interval in $V(H')$ and $G$ is monochromatic with respect to the coloring $\chi_0$, then there exists an interval $M''\subseteq M'$ of size $|M''|\geq (|M'|-r+6)/2$ such that $M''$ is a $1$-comb in $V'$.
\end{proposition}

\begin{proof}
By Proposition \ref{prop:uniquecomb}, all the edges $X_J$ of $G$ are combs of the same type. Thus we may assume without loss of generality that $X_J$ is either a broken comb or a $\ell$-left comb in $X$. Since $M'$ is closed, by the same proposition we obtain that $A^{X_J}\subseteq P$ and consequently $K_J \subseteq B^{X_J}$ for every edge $X \in H'$. Therefore, we either have $K_J=\emptyset$ (and $X_J$ is a broken comb) or $P< K_J$ for every $X$, which implies that $K_J\subseteq K_1$, i.e., $M'< K_J$. Moreover, if $X_J$ is an $\ell$-left comb, then $A^{X_J}\subseteq P$ implies that $\ell\leq r$. This implies that $X_J$ is either of type $1$, $3$ or $4$.

We split the proof into two cases according to the size of $t=|K_J|$. Write $M'=\{y_1,\ldots,y_{m'}\}$, $K_J=\{y_{m'+1},\ldots,y_{m'+t}\}$ (if $K_J\neq \emptyset$) and $\delta_i^G=\delta(y_i,y_{i+1})$ for $1\leq i \leq m'+t-1$.

\vspace{0.2cm}

\noindent \underline{Case 1:} $0\leq t\leq r-2$.

Since $|X_J|=r+t\leq 2r-2$, we obtain that $X_J$ is either of type $1$ or $3$. The proof follow the same lines of the proof of Proposition \ref{prop:findcomb}. Write $P=\{y_{i_1},\ldots,y_{i_r}\}\subseteq M'$ for indices $1\leq i_1<\ldots<i_r\leq m'$. Suppose without loss of generality that $G$ is monochromatic of color $0$, i.e., $\chi_0(X_J)=0$ for every $X_J \in G$. Thus, by the definition of $\chi_0$ for combs of type $1$ and $3$, we do not have that $\delta_{i_{r-3}}^G<\delta_{i_{r-2}}^G>\delta_{i_{r-1}}^G$. In particular, because $X_J$ is arbitrary, this implies that there are no indices $r-3\leq p<q<s \leq m'-1$ such that $\delta_p^G<\delta_q^G>\delta_s^G$. That is, the sequence $\{\delta_i^G\}_{i=r-3}^{m'-1}$ has no local maximum.

Now the same argument as in Proposition \ref{prop:findcomb} gives that there exists an interval $M''=\{y_p,\ldots,y_q\}\subseteq M'$ such that $\{\delta_i^G\}_{i=p}^{q-1}$ is monotone and $|M''|\geq (|M'|-r+6)/2$. By the definition given in Example \ref{ex:1comb}, it follows that $M''$ is a $1$-comb.

\vspace{0,2cm}

\noindent \underline{Case 2:} $t\geq r-1$.

In this case $X_J$ is a left comb of type 4 for every $X_J\in G$, since $|X_J|=|P|+|K_J|=r+t\geq 2r-1$. Suppose without loss of generality that $G$ is monochromatic of color $0$ and that $\phi_t(\{\delta_{m'}^G,\ldots,\delta_{m'+t-1}^G\})=0$. Let $u=a(\min(M'),\max(M'))$. Fact \ref{fact:closedanc} applied to $M'$ gives us that $\delta(z,y_{m'+1})=\delta(u,y_{m'+1})$ for every $z\in M'$. In particular, this implies that $\delta(z,y_{m'+1})=\delta_{m'}^G$ for every $z\in M'$.

Write $P=\{y_{i_1},\ldots,y_{i_r}\}\subseteq M'$ with $1\leq i_1<\ldots<i_r\leq m'$ and $X_J=P\cup K_J=\{y_{i_1},\ldots,y_{i_r},y_{m'},\ldots,y_{m'+t-1}\}$. Since $\chi_0(X_J)=0$, $\delta(y_{i_j},y_{m'+1})=\delta_{m'}^G$ for every $1\leq j\leq r$ and $\phi_t(\{\delta_{m'}^G,\ldots,\delta_{m'+t-1}^G\})+1=1$ we obtain by the definition of $\chi_0$ for combs of type $4$ that the inequality
$\delta_{i_{r-3}}^G<\delta_{i_{r-2}}^G>\delta_{i_{r-1}}^G$ cannot hold. Because $X_J$ is arbitrary, we have that there are no indices $r-3\leq p<q<s \leq m'-1$ such that $\delta_p^G<\delta_q^G>\delta_s^G$. Hence, similarly as in Case $1$ we find an interval $M''\subseteq M'$ of size at least $(|M'|-r+6)/2$ such that $M''$ is a $1$-comb in $V'$.
\end{proof}

To finish the proof of Theorem \ref{thm:daisyramsey} we are going to show now that if $G$ is monochromatic with respect to $\chi_0$, then there exists a monochromatic simple $(r-1,\frac{1}{5}k^{-1/2}m^{1/2},j)$-daisy in $\delta(G)\subseteq [N]$ with respect to some coloring $\phi_{j+r-1}$. The proof is split into several cases depending on the structure of $H'$ given by Lemma \ref{lem:preprocessing} and on the possible types of $X_J$.

\vspace{0.2cm}

\noindent \underline{Case 1:} $M'$ is a closed interval in $V'$.

As usual, we may assume that an edge of $G$ is either a broken comb or a left comb. By Proposition \ref{prop: findcombdaisy}, there exists an interval $M''\subseteq M'$ of size $h=(|M'|-r+6)/2$ such that $M''$ is a $1$-comb. Consider the coloring $\chi_0$ over the monochromatic subdaisy $G':=G[K_J\cup M'']\subseteq G$. As in the proof of Proposition \ref{prop: findcombdaisy}, we have that either $X_J$ is a broken comb and $t=|K_J|=0$ or $X_J$ is an $\ell$-left comb with $\ell\leq r$ and $M'<K_J$. Write $M''=\{y_{i_1},\ldots,y_{i_h}\}$ with $1\leq i_1<\ldots<i_h<m'$, $K_J=\{y_{m'+1},\ldots,y_{m'+t}\}$ (if $K_J\neq \emptyset$) and $\delta_i^G=\delta(y_i,y_{i+1})$. 

\begin{figure}[h]
\centering
{\hfil \begin{tikzpicture}[scale=0.6]
    
 	\draw (-8,4.75)[green1] ellipse (0.4 and 1.1);
 	\draw (-8,1) ellipse (0.7 and 1.2);
 	\draw (-8,4.75)[red] ellipse (0.7 and 2.5);
 	
 	\coordinate (T1) at (-5,7.5);
 	\coordinate (T2) at (-4.33,7.5);
 	\coordinate (T3) at (-3.66,7.5);
 	\coordinate (T4) at (-3,7.5);
 	\coordinate (T5) at (-2.33,7.5);
 	\coordinate (T6) at (-1.66,7.5);
 	\coordinate (T7) at (-1,7.5);
 	\coordinate (T8) at (-0.33,7.5);
 	\coordinate (T9) at (0.33,7.5);
 	\coordinate (T10) at (1,7.5);
 	\coordinate (T11) at (1.66,7.5);
 	\coordinate (T12) at (2.33,7.5);
 	\coordinate (T13) at (3,7.5);
 	\coordinate (T14) at (3.66,7.5);
 	\coordinate (T15) at (4.33,7.5);
 	\coordinate (T16) at (5,7.5);

    \coordinate (R16) at (0,0);
    \coordinate (R15) at (-0.33,0.5);
    \coordinate (R14) at (-0.66,1);
    \coordinate (R13) at (-1,1.5);
    \coordinate (R12) at (-1.33,2);
    \coordinate (R1) at (-1.66,2.5);
    \coordinate (R2) at (-1.33,3);
    \coordinate (R3) at (-1,3.5);
    \coordinate (R4) at (-0.66,4);
    \coordinate (R5) at (-0.33,4.5);
    \coordinate (R6) at (0,5);
    \coordinate (R7) at (0.33,5.5);
    \coordinate (R8) at (0.66,6);
    \coordinate (R9) at (1,6.5);
    \coordinate (R10) at (1.33,7);

    \coordinate (L15) at (-8,7.5);
    \coordinate (L14) at (-8,7);
    \coordinate (L13) at (-8,6.5);
    \coordinate (L12) at (-8,6);
    \coordinate (L11) at (-8,5.5);
    \coordinate (L10) at (-8,5);
    \coordinate (L9) at (-8,4.5);
    \coordinate (L8) at (-8,4);
    \coordinate (L7) at (-8,3.5);
    \coordinate (L6) at (-8,3);
    \coordinate (L5) at (-8,2.5);
    \coordinate (L4) at (-8,2);
    \coordinate (L3) at (-8,1.5);
    \coordinate (L2) at (-8,1);
    \coordinate (L1) at (-8,0.5);
    \coordinate (L0) at (-8,0);
    
    \node (p) at (-2,8.5) [above,font=\small] {$P$};
	\node (p') at (-9,5) [left,font=\small] {$\delta(P)$};
	\node (m') at (-8.2,7) [above left,font=\small] {$\delta(M'')$};
	\node (k') at (-9,0.5) [left,font=\small] {$K_D$};
	\node (m) at (-1.7,9.8) [above,font=\small] {$M''$};
	\node (k) at (3.7,8.5) [above,font=\small] {$K_J$};
    
    \draw [decorate,
    decoration = {brace, amplitude=8pt}] (-2.9,8) --  (-1.1,8);
    \draw [decorate,
    decoration = {brace, amplitude=8pt}] (-4.9,9.3) --  (1.5,9.3);
    \draw [decorate,
    decoration = {brace, amplitude=8pt}] (2.4,8) --  (5,8);
    
    \draw[line width=0.5] (R15)--(R16)--(T16);
    \draw[line width=0.5] (R14)--(R15)--(T15);
    \draw[line width=0.5] (R13)--(R14)--(T14);
    \draw[line width=0.5] (R12)--(R13)--(T13);
    \draw[line width=0.5] (R12)--(T12);
    \draw[line width=0.5] (R1)--(R12);
    \draw[line width=0.5][red] (T1)--(R1)--(R2);
    \draw[line width=0.5][red] (T2)--(R2)--(R3);
    \draw[line width=0.5][red] (T3)--(R3)--(R4);
    \draw[line width=0.5][green1] (T4)--(R4)--(R5);
    \draw[line width=0.5][green1] (T5)--(R5)--(R6);
    \draw[line width=0.5][green1] (T6)--(R6)--(R7);
    \draw[line width=0.5][green1] (T7)--(R7);
    \draw[line width=0.5][red] (R7)--(R8);
    \draw[line width=0.5][red] (T8)--(R8)--(R9);
    \draw[line width=0.5][red] (T9)--(R9)--(R10);
    \draw[line width=0.5][red] (T10)--(R10)--(T11);
    \draw[line width=0.5] (L0)--(L15);
    
    \draw[line width=0.1][dashed] (L8)--(R4);
    \draw[line width=0.1][dashed] (L9)--(R5);
    \draw[line width=0.1][dashed] (L10)--(R6);
    \draw[line width=0.1][dashed] (L11)--(R7);
    
    \draw[fill][green1] (T4) circle [radius=0.1];
 	\draw[fill][green1] (T5) circle [radius=0.1];
 	\draw[fill][green1] (T6) circle [radius=0.1];
 	\draw[fill][green1] (T7) circle [radius=0.1];
 	\draw[fill][green1] (R4) circle [radius=0.1];
 	\draw[fill][green1] (R5) circle [radius=0.1];
 	\draw[fill][green1] (R6) circle [radius=0.1];
 	\draw[fill][green1] (R7) circle [radius=0.1];
 	\draw[fill][green1] (L8) circle [radius=0.1];
 	\draw[fill][green1] (L9) circle [radius=0.1];
 	\draw[fill][green1] (L10) circle [radius=0.1];
 	\draw[fill][green1] (L11) circle [radius=0.1];

    \end{tikzpicture}\hfil}
    \caption{Case 1 of Theorem \ref{th:daisystepup}.}
    \label{fig:case1}
\end{figure}
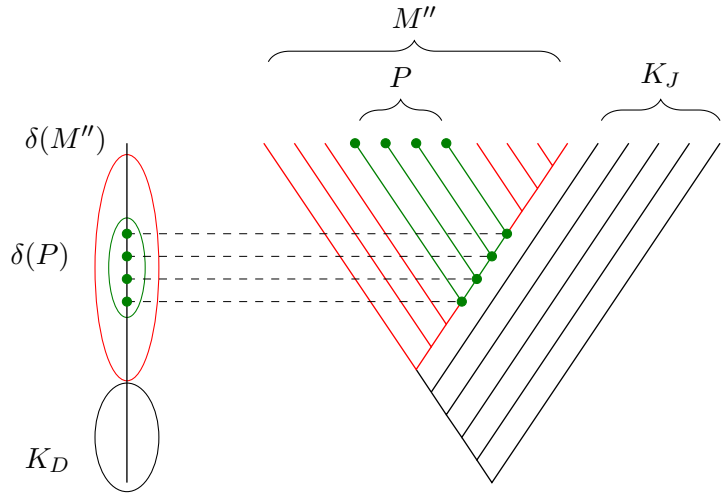

Let $X_J=P\cup K_J=\{x_1,\ldots,x_r\}\cup \{y_{m'+1},\ldots,y_{m'+t}\}$ be an arbitrary edge from $G'$ with $P\subseteq (M'')^{(r)}$. Note that since $M''$ is a $1$-comb, then $\delta(x_{r-3},x_{r-2})$, $\delta(x_{r-2},x_{r-1})$, $\delta(x_{r-1},x_r)$ forms a monotone sequence. Moreover, as discussed in Proposition \ref{prop: findcombdaisy}, the comb $X_J$ is of type $1$, $3$ or $4$. Thus, by the coloring defined in Subsection \ref{subsec:variant}, we have $\chi_0(X_J)=\phi_{r+t-1}(\delta(X_J))$, i.e., the color of $X_J$ is determined by its full projection on the the levels $[N]$.

Let $u=a(\min(M'),\max(M'))$. Note that since $M'$ is closed, by Fact \ref{fact:closedanc} we have that $a(x_r,y_{m'+1})=a(u,y_{m'+1})=a(y_{m'},y_{m'+1})$. Consequently, we have that $\delta(x_r,y_{m'+1})=\delta^G_{m'}$, which implies that $\delta(X_J)=\{\delta(x_1,x_2),\ldots,\delta(x_{r-1},x_r)\}\cup \{\delta_{m'}^G,\ldots,\delta_{m'+t-1}^G\}$. Therefore the projection of all the edges of $X_J$ forms a simple $(r-1,h-1,t)$-daisy $D\subseteq [N]$ with universe of petals $\delta(M'')$ an kernel $K_D=\{\delta_{m'}^G,\ldots,\delta_{m'+t-1}^G\}$ satisfying $K_D<\delta(M'')$. By the fact that $G'$ is monochromatic with respect to $\chi_0$, we have that $D$ is a monochromatic simple $(r-1,h-1,t)$-daisy with respect to the coloring $\phi_{r+t-1}$. This leads to a contradiction since $h-1\geq (m'-r+4)/2\geq \frac{1}{5}k^{-1/2}m^{1/2}$ for $m\geq 100kr^2$ and $\phi_{r+t-1}$ has no monochromatic simple $(r-1,\frac{1}{5}k^{-1/2}m^{1/2},t)$-daisy.  

\vspace{0.2cm}

\noindent \underline{Case 2:} There exists a maximal comb $I=A\cup B$ in $V(H')$ such that $M'\subseteq B$

We may assume without loss of generality that $I=A\cup B$ is a left comb. Write $A=\{y_1,\ldots,y_{\ell}\}$, $B_0=B\cap K_0=\{y_{\ell+1},\ldots,y_{\ell+p}\}$, $M'=\{y_{\ell+p+1},\ldots,y_{\ell+p+m'}\}$ and $B_1=B\cap K_1=\{y_{\ell+p+m'+1},\ldots,y_{\ell+p+m'+t}\}$. By Proposition \ref{prop:uniquecomb}, $V(G)=I$ and all the edges $X_J=A^{X_J}\cup B^{X_J} \in G$ are maximal left comb of same type with $A^{X_J}=A$, $B^{X_J}=B_0\cup P\cup B_1$ and $A<B_0<P<B_1$. In particular, this implies that $|B^{X_J}|\geq r$ and $X_J$ is of type $3$, $4$ or $5$. We split the cases depending on the type of $X_J$. Let $\delta_i^G=\delta(y_i,y_{i+1})$ for $1\leq i \leq \ell+p+m'+t-1$. For an arbitrary edge $X_J\in G$, write $X_J=\{x_1,\ldots,x_{\ell+p+r+t}\}$ with $x_i=y_i$ for $1\leq i \leq \ell+p$, $P=\{x_{\ell+p+1},\ldots,x_{\ell+p+r}\}\subseteq M'$ and $x_{\ell+p+r+i}=y_{\ell+p+m'+i}$ for $1\leq i \leq t$ and let $\delta_i^{X_J}=\delta(x_i,x_{i+1})$ for $1\leq i \leq \ell+p+r+t-1$.

\begin{figure}[h]
\centering
{\hfil \begin{tikzpicture}[scale=0.4]
    
 	\draw (-8,9.5) ellipse (1.66 and 0.3);
 	\draw[red] (0.33,9.5) ellipse (3.7 and 0.5);
 	
 	\coordinate (T0) at (-9.66,9.5);
 	\coordinate (T1) at (-6.33,9.5);
 	\coordinate (T2) at (-5.66,9.5);
 	\coordinate (T3) at (-5,9.5);
 	\coordinate (T4) at (-4.33,9.5);
 	\coordinate (T5) at (-3.66,9.5);
 	\coordinate (T6) at (-3,9.5);
 	\coordinate (T7) at (-2.33,9.5);
 	\coordinate (T8) at (-1.66,9.5);
 	\coordinate (T9) at (-1,9.5);
 	\coordinate (T10) at (-0.33,9.5);
 	\coordinate (T11) at (0.33,9.5);
 	\coordinate (T12) at (1,9.5);
 	\coordinate (T13) at (1.66,9.5);
 	\coordinate (T14) at (2.33,9.5);
 	\coordinate (T15) at (3,9.5);
 	\coordinate (T16) at (3.66,9.5);
 	\coordinate (T17) at (4.33,9.5);
 	\coordinate (T18) at (5,9.5);
 	\coordinate (T19) at (5.66,9.5);
 	\coordinate (T20) at (6.33,9.5);
    \coordinate (R1) at (-8,7);
    \coordinate (R2) at (-7.66,6.5);
    \coordinate (R3) at (-7.33,6);
    \coordinate (R4) at (-7,5.5);
    \coordinate (R5) at (-6.66,5);
    \coordinate (R6) at (-6.33,4.5);
    \coordinate (R7) at (-6,4);
    \coordinate (R8) at (-5.66,3.5);
    \coordinate (R9) at (-5.33,3);
    \coordinate (R10) at (-5,2.5);
    \coordinate (R11) at (-4.66,2);
    \coordinate (R12) at (-4.33,1.5);
    \coordinate (R13) at (-4,1);
    \coordinate (R14) at (-3.66,0.5);
    \coordinate (R15) at (-3.33,0);
    \coordinate (R16) at (-3,-0.5);
    \coordinate (R17) at (-2.66,-1);
    \coordinate (R18) at (-2.33,-1.5);
    \coordinate (R19) at (-2,-2);
    \coordinate (R20) at (-1.66,-2.5);

    \node (a) at (-7.95,10.7) [above,font=\small] {$A$};
	\node (b0) at (-4.6,10.7) [above,font=\small] {$B_0$};
	\node (m) at (0.2,10.7) [above,font=\small] {$M'$};
	\node (b1) at (5.3,10.7) [above,font=\small] {$B_1$};

    \draw [decorate,
    decoration = {brace, amplitude=7pt}] (-9.5,10.2) --  (-6.4,10.2);
    \draw [decorate,
    decoration = {brace, amplitude=7pt}] (-5.5,10.2) --  (-3.7,10.2);
    \draw [decorate,
    decoration = {brace, amplitude=7pt}] (-2.9,10.2) --  (3.5,10.2);
    \draw [decorate,
    decoration = {brace, amplitude=7pt}] (4.4,10.2) --  (6.2,10.2);
    
    \draw[line width=0.5] (T0)--(R1)--(T1);
    \draw[line width=0.5] (R1)--(R2)--(T2);
    \draw[line width=0.5] (R2)--(R3)--(T3);
    \draw[line width=0.5] (R3)--(R4)--(T4);
    \draw[line width=0.5] (R4)--(R5)--(T5);
    \draw[line width=0.5][red] (R5)--(R6)--(T6);
    \draw[line width=0.5][red] (R6)--(R7)--(T7);
    \draw[line width=0.5][red] (R7)--(R8)--(T8);
    \draw[line width=0.5][red] (R8)--(R9)--(T9);
    \draw[line width=0.5][red] (R9)--(R10)--(T10);
    \draw[line width=0.5][red] (R10)--(R11)--(T11);
    \draw[line width=0.5][red] (R11)--(R12)--(T12);
    \draw[line width=0.5][red] (R12)--(R13)--(T13);
    \draw[line width=0.5][red] (R13)--(R14)--(T14);
    \draw[line width=0.5][red] (R14)--(R15)--(T15);
    \draw[line width=0.5][red] (R15)--(R16)--(T16);
    \draw[line width=0.5] (R16)--(R17)--(T17);
    \draw[line width=0.5] (R17)--(R18)--(T18);
    \draw[line width=0.5] (R18)--(R19)--(T19);
    \draw[line width=0.5] (R19)--(R20)--(T20);
 	
    \end{tikzpicture}\hfil}
    \caption{Auxiliary tree of $G$ in Case $2$.}
    \label{fig:treeG}
\end{figure}
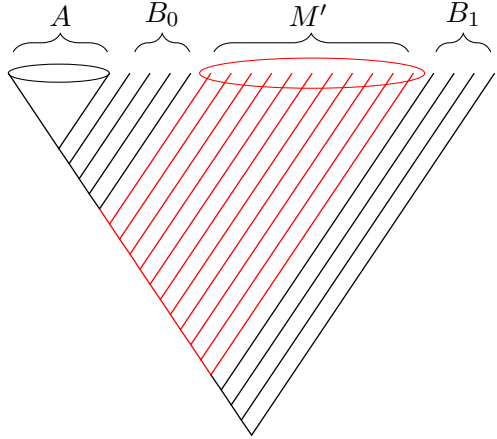

\vspace{0.2cm}

\noindent \underline{Case 2.1:} $X_J$ is of type $3$.

Recall that if $X_J$ is of type $3$, then $|A^{X_J}|\leq r$ and $r\leq |X_J|=|A^{X_J}|+|B^{X_J}|\leq 2r-2$. Because $|B^{X_J}|\geq r$, we obtain that $|A^{X_J}|\leq r-2$. This implies that $\{x_{r-1},x_{r}\}\subseteq B^{X_J}$ and consequently $\delta_{r-3}^{X_J}>\delta_{r-2}^{X_J}>\delta_{r-1}^{X_J}$. Therefore, by the fact that $|\delta(X_J)|\geq |\delta(\{x_\ell,\ldots,x_{\ell+p+r+t}\})|=p+r+t\geq r$, we obtain that $\chi_0(X_J)=\phi_{|\delta(X_J)|}(\delta(X_J))$.

Note that
\begin{align*}
    \delta(X_J)&=\{\delta_1^G, \ldots, \delta_{\ell+p-1}^G\}\cup \{\delta_{\ell+p}^{X_J},\ldots,\delta_{\ell+p+r-1}^{X_J}\}\cup\{\delta_{\ell+p+m'}^G,\ldots,\delta_{\ell+p+m'+t-1}^G\}\\
    &=\delta(A\cup B_0)\cup \{\delta_{\ell+p}^{X_J},\ldots,\delta_{\ell+p+r-1}^{X_J}\}\cup\delta(\{y_{\ell+p+m'}\}\cup B_1).
\end{align*}
Hence, the projection of the edges of $G$ is a simple $(r,m',|\delta(A\cup B_0)|+t)$-daisy with kernel $\delta(A\cup B_0)\cup \delta(\{y_{\ell+p+m'}\}\cup B_1)$. Since $G$ is monochromatic with respect to $\chi_0$, we obtain that $\delta(G)\subseteq [N]$ is monochromatic with respect to the projection coloring, which is a contradiction because any simple $(r,m',|\delta(A\cup B_0)|+t)$-daisy contains a simple $(r-1,m'-1,|\delta(A\cup B_0)|+t+1)$-subdaisy and $m'-1\geq \frac{1}{5}k^{-1/2}m^{1/2}$.

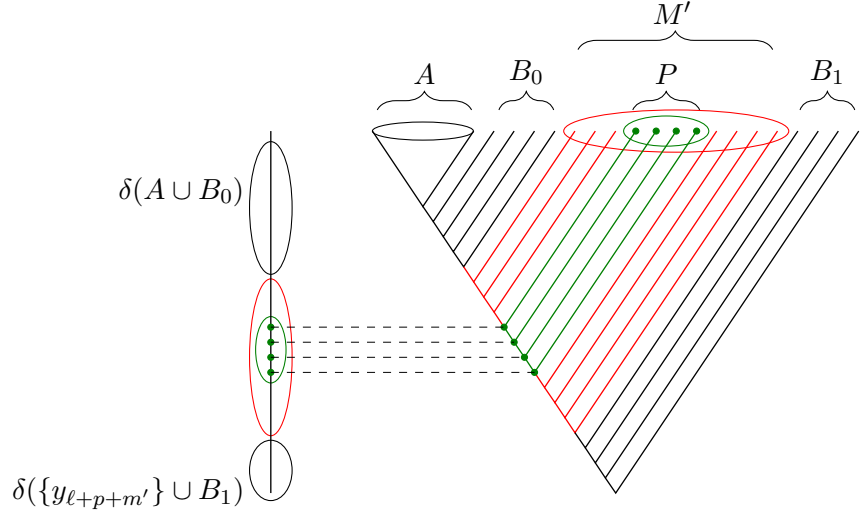
\begin{figure}[h]
\centering
{\hfil \begin{tikzpicture}[scale=0.4]
    
 	\draw (-8,9.5) ellipse (1.66 and 0.3);
 	\draw[red] (0.33,9.5) ellipse (3.7 and 0.7);
 	\draw[green1] (0,9.5) ellipse (1.4 and 0.5);
 	\draw[green1] (-13,2.25) ellipse (0.5 and 1.1);
 	\draw[red] (-13,2) ellipse (0.7 and 2.6);
 	\draw (-13,6.95) ellipse (0.7 and 2.2);
 	\draw (-13,-1.75) ellipse (0.7 and 1);
 	
 	\coordinate (T0) at (-9.66,9.5);
 	\coordinate (T1) at (-6.33,9.5);
 	\coordinate (T2) at (-5.66,9.5);
 	\coordinate (T3) at (-5,9.5);
 	\coordinate (T4) at (-4.33,9.5);
 	\coordinate (T5) at (-3.66,9.5);
 	\coordinate (T6) at (-3,9.5);
 	\coordinate (T7) at (-2.33,9.5);
 	\coordinate (T8) at (-1.66,9.5);
 	\coordinate (T9) at (-1,9.5);
 	\coordinate (T10) at (-0.33,9.5);
 	\coordinate (T11) at (0.33,9.5);
 	\coordinate (T12) at (1,9.5);
 	\coordinate (T13) at (1.66,9.5);
 	\coordinate (T14) at (2.33,9.5);
 	\coordinate (T15) at (3,9.5);
 	\coordinate (T16) at (3.66,9.5);
 	\coordinate (T17) at (4.33,9.5);
 	\coordinate (T18) at (5,9.5);
 	\coordinate (T19) at (5.66,9.5);
 	\coordinate (T20) at (6.33,9.5);
    \coordinate (R1) at (-8,7);
    \coordinate (R2) at (-7.66,6.5);
    \coordinate (R3) at (-7.33,6);
    \coordinate (R4) at (-7,5.5);
    \coordinate (R5) at (-6.66,5);
    \coordinate (R6) at (-6.33,4.5);
    \coordinate (R7) at (-6,4);
    \coordinate (R8) at (-5.66,3.5);
    \coordinate (R9) at (-5.33,3);
    \coordinate (R10) at (-5,2.5);
    \coordinate (R11) at (-4.66,2);
    \coordinate (R12) at (-4.33,1.5);
    \coordinate (R13) at (-4,1);
    \coordinate (R14) at (-3.66,0.5);
    \coordinate (R15) at (-3.33,0);
    \coordinate (R16) at (-3,-0.5);
    \coordinate (R17) at (-2.66,-1);
    \coordinate (R18) at (-2.33,-1.5);
    \coordinate (R19) at (-2,-2);
    \coordinate (R20) at (-1.66,-2.5);
    
    \coordinate (L) at (-13,-2.5);
    \coordinate (L') at (-13, 9.5);
    \coordinate (L9) at (-13,3);
    \coordinate (L10) at (-13,2.5);
    \coordinate (L11) at (-13,2);
    \coordinate (L12) at (-13,1.5);

    \node (a) at (-7.95,10.7) [above,font=\small] {$A$};
	\node (b0) at (-4.6,10.7) [above,font=\small] {$B_0$};
	\node (m) at (0.2,12.7) [above,font=\small] {$M'$};
	\node (b1) at (5.3,10.7) [above,font=\small] {$B_1$};
	\node (p) at (0,10.7) [above,font=\small] {$P$};
	\node (d0) at (-13.5,-2.5) [left,font=\small] {$\delta(\{y_{\ell+p+m'}\}\cup B_1)$};
	\node (d1) at (-13.5,7.5) [left,font=\small] {$\delta(A\cup B_0)$};

    \draw [decorate,
    decoration = {brace, amplitude=7pt}] (-9.5,10.2) --  (-6.4,10.2);
    \draw [decorate,
    decoration = {brace, amplitude=7pt}] (-5.5,10.2) --  (-3.7,10.2);
    \draw [decorate,
    decoration = {brace, amplitude=7pt}] (-2.9,12.2) --  (3.5,12.2);
    \draw [decorate,
    decoration = {brace, amplitude=7pt}] (4.4,10.2) --  (6.2,10.2);
    \draw [decorate,
    decoration = {brace, amplitude=7pt}] (-1.1,10.2) --  (1.1,10.2);
    
    \draw[line width=0.5] (T0)--(R1)--(T1);
    \draw[line width=0.5] (R1)--(R2)--(T2);
    \draw[line width=0.5] (R2)--(R3)--(T3);
    \draw[line width=0.5] (R3)--(R4)--(T4);
    \draw[line width=0.5] (R4)--(R5)--(T5);
    \draw[line width=0.5][red] (R5)--(R6)--(T6);
    \draw[line width=0.5][red] (R6)--(R7)--(T7);
    \draw[line width=0.5][red] (R7)--(R8)--(T8);
    \draw[line width=0.5][red] (R8)--(R9);
    \draw[line width=0.5][green1] (R9)--(T9);
    \draw[line width=0.5][green1] (R9)--(R10)--(T10);
    \draw[line width=0.5][green1] (R10)--(R11)--(T11);
    \draw[line width=0.5][green1] (R11)--(R12)--(T12);
    \draw[line width=0.5][red] (R12)--(R13)--(T13);
    \draw[line width=0.5][red] (R13)--(R14)--(T14);
    \draw[line width=0.5][red] (R14)--(R15)--(T15);
    \draw[line width=0.5][red] (R15)--(R16)--(T16);
    \draw[line width=0.5] (R16)--(R17)--(T17);
    \draw[line width=0.5] (R17)--(R18)--(T18);
    \draw[line width=0.5] (R18)--(R19)--(T19);
    \draw[line width=0.5] (R19)--(R20)--(T20);
    \draw[line width=0.5] (L)--(L');
 	
 	\draw[fill][green1] (T9) circle [radius=0.1];
 	\draw[fill][green1] (T10) circle [radius=0.1];
 	\draw[fill][green1] (T11) circle [radius=0.1];
 	\draw[fill][green1] (T12) circle [radius=0.1];
 	\draw[fill][green1] (R9) circle [radius=0.1];
 	\draw[fill][green1] (R10) circle [radius=0.1];
 	\draw[fill][green1] (R11) circle [radius=0.1];
 	\draw[fill][green1] (R12) circle [radius=0.1];
 	\draw[fill][green1] (L9) circle [radius=0.1];
 	\draw[fill][green1] (L10) circle [radius=0.1];
 	\draw[fill][green1] (L11) circle [radius=0.1];
 	\draw[fill][green1] (L12) circle [radius=0.1];
 	
 	\draw[line width=0.1][dashed] (L9)--(R9);
    \draw[line width=0.1][dashed] (L10)--(R10);
    \draw[line width=0.1][dashed] (L11)--(R11);
    \draw[line width=0.1][dashed] (L12)--(R12);
 	
    \end{tikzpicture}\hfil}
    \caption{Case $2.1$ of Theorem \ref{th:daisystepup}}
    \label{fig:case21}
\end{figure}

\vspace{0.2cm}

\noindent \underline{Case 2.2:} $X_J$ is of type $4$.

If $X_J$ is of type 4, then $|A^{X_J}|\leq r$ and $|X_J|=|A^{X_J}|+|B^{X_J}|\geq 2r-1$. We split the proof into two subcases depending on the sequence formed by $\{\delta_{r-3}^{X_J},\delta_{r-2}^{X_J},\delta_{r-1}^{X_J}\}$:

\vspace{0.2cm}

\noindent \underline{Case 2.2.a:} Either $\delta_{r-3}^{X_J}>\delta_{r-2}^{X_J}<\delta_{r-1}^{X_J}$ or $\delta_{r-3}^{X_J}<\delta_{r-2}^{X_J}>\delta_{r-1}^{X_J}$.

Suppose without loss of generality that $\delta_{r-3}^{X_J}>\delta_{r-2}^{X_J}<\delta_{r-1}^{X_J}$. Hence, by the coloring definition, we have $\chi_0(X_J)=\phi_{\ell+p+t}(\{\delta_r^{X_J},\ldots,\delta_{\ell+p+r+t-1}^{X_J}\})$. Thus, we just need to look at the projection $\{\delta_r^{X_J},\ldots,\delta_{\ell+p+r+t-1}^{X_J}\}$ for every $X_J \in G$. Note that $\delta_{r-3}^{X_J}>\delta_{r-2}^{X_J}<\delta_{r-1}^{X_J}$ implies that $|A^{X_J}|\geq r-1$. Indeed, by the same argument made in Case 2.1, if $|A^{X_J}|\leq r-2$, then $\delta_{r-3}^{X_J}>\delta_{r-2}^{X_J}>\delta_{r-1}^{X_J}$, which is a contradiction. So, it follows that $r-1\leq |A^{X_J}|=\ell\leq r$.

\begin{figure}[h]
\centering
{\hfil \begin{tikzpicture}[scale=0.4]
    
 	\draw (-7,9.5) ellipse (2.66 and 0.5);
 	\draw[red] (0,9.5) ellipse (4 and 0.7);
 	\draw[green1] (0,9.5) ellipse (1.4 and 0.5);
 	\draw[green1] (-13,2.25) ellipse (0.5 and 1.1);
 	\draw[red] (-13,2) ellipse (0.7 and 2.6);
 	\draw (-13,-1.75) ellipse (0.7 and 1);
 	
 	\coordinate (T0) at (-9.66,9.5);
 	\coordinate (T1) at (-6.33,9.5);
 	\coordinate (T2) at (-5.66,9.5);
 	\coordinate (T3) at (-5,9.5);
 	\coordinate (T4) at (-4.33,9.5);
 	\coordinate (T5) at (-3.66,9.5);
 	\coordinate (T6) at (-3,9.5);
 	\coordinate (T7) at (-2.33,9.5);
 	\coordinate (T8) at (-1.66,9.5);
 	\coordinate (T9) at (-1,9.5);
 	\coordinate (T10) at (-0.33,9.5);
 	\coordinate (T11) at (0.33,9.5);
 	\coordinate (T12) at (1,9.5);
 	\coordinate (T13) at (1.66,9.5);
 	\coordinate (T14) at (2.33,9.5);
 	\coordinate (T15) at (3,9.5);
 	\coordinate (T16) at (3.66,9.5);
 	\coordinate (T17) at (4.33,9.5);
 	\coordinate (T18) at (5,9.5);
 	\coordinate (T19) at (5.66,9.5);
 	\coordinate (T20) at (6.33,9.5);
    \coordinate (R1) at (-8,7);
    \coordinate (R2) at (-7.66,6.5);
    \coordinate (R3) at (-7.33,6);
    \coordinate (R4) at (-7,5.5);
    \coordinate (R5) at (-6.66,5);
    \coordinate (R6) at (-6.33,4.5);
    \coordinate (R7) at (-6,4);
    \coordinate (R8) at (-5.66,3.5);
    \coordinate (R9) at (-5.33,3);
    \coordinate (R10) at (-5,2.5);
    \coordinate (R11) at (-4.66,2);
    \coordinate (R12) at (-4.33,1.5);
    \coordinate (R13) at (-4,1);
    \coordinate (R14) at (-3.66,0.5);
    \coordinate (R15) at (-3.33,0);
    \coordinate (R16) at (-3,-0.5);
    \coordinate (R17) at (-2.66,-1);
    \coordinate (R18) at (-2.33,-1.5);
    \coordinate (R19) at (-2,-2);
    \coordinate (R20) at (-1.66,-2.5);
    
    \coordinate (L) at (-13,-2.5);
    \coordinate (L') at (-13, 9.5);
    \coordinate (L9) at (-13,3);
    \coordinate (L10) at (-13,2.5);
    \coordinate (L11) at (-13,2);
    \coordinate (L12) at (-13,1.5);

    \node (a) at (-9.66,9.5) [above left,font=\small] {$A$};
	\node (r) at (-6.5,10.8) [above,font=\small] {$r$};
	\node (m) at (0.2,12.7) [above,font=\small] {$M'$};
	\node (b1) at (5.3,10.7) [above,font=\small] {$B_1$};
	\node (p) at (0,10.7) [above,font=\small] {$P$};
	\node (m') at (-13.5,4.6) [left,font=\small] {$\delta(M')$};
	\node (d) at (-13.5,-2.5) [left,font=\small] {$\{\delta_{r+m'-1}^G,\ldots,\delta_{r+m'+t-2}^G\}$};
	\node (p) at (-13.5,1.1) [left,font=\small] {$\delta(P)$};

    \draw [decorate,
    decoration = {brace, amplitude=7pt}] (-9.5,10.2) --  (-3.6,10.2);
    \draw [decorate,
    decoration = {brace, amplitude=7pt}] (-2.9,12.2) --  (3.5,12.2);
    \draw [decorate,
    decoration = {brace, amplitude=7pt}] (4.4,10.2) --  (6.2,10.2);
    \draw [decorate,
    decoration = {brace, amplitude=7pt}] (-1.1,10.2) --  (1.1,10.2);
    
    \draw[line width=0.5] (T0)--(R1);
    \draw[line width=0.5] (R1)--(R2);
    \draw[line width=0.5] (R2)--(R3);
    \draw[line width=0.5] (R3)--(R4)--(T4);
    \draw[line width=0.5] (R4)--(R5);
    \draw[line width=0.5][red] (T5)--(R5)--(R6)--(T6);
    \draw[line width=0.5][red] (R6)--(R7)--(T7);
    \draw[line width=0.5][red] (R7)--(R8)--(T8);
    \draw[line width=0.5][red] (R8)--(R9);
    \draw[line width=0.5][green1] (R9)--(T9);
    \draw[line width=0.5][green1] (R9)--(R10)--(T10);
    \draw[line width=0.5][green1] (R10)--(R11)--(T11);
    \draw[line width=0.5][green1] (R11)--(R12)--(T12);
    \draw[line width=0.5][red] (R12)--(R13)--(T13);
    \draw[line width=0.5][red] (R13)--(R14)--(T14);
    \draw[line width=0.5][red] (R14)--(R15)--(T15);
    \draw[line width=0.5][red] (R15)--(R16)--(T16);
    \draw[line width=0.5] (R16)--(R17)--(T17);
    \draw[line width=0.5] (R17)--(R18)--(T18);
    \draw[line width=0.5] (R18)--(R19)--(T19);
    \draw[line width=0.5] (R19)--(R20)--(T20);
    \draw[line width=0.5] (L)--(L');
 	
 	\draw[fill][green1] (T9) circle [radius=0.1];
 	\draw[fill][green1] (T10) circle [radius=0.1];
 	\draw[fill][green1] (T11) circle [radius=0.1];
 	\draw[fill][green1] (T12) circle [radius=0.1];
 	\draw[fill][green1] (R9) circle [radius=0.1];
 	\draw[fill][green1] (R10) circle [radius=0.1];
 	\draw[fill][green1] (R11) circle [radius=0.1];
 	\draw[fill][green1] (R12) circle [radius=0.1];
 	\draw[fill][green1] (L9) circle [radius=0.1];
 	\draw[fill][green1] (L10) circle [radius=0.1];
 	\draw[fill][green1] (L11) circle [radius=0.1];
 	\draw[fill][green1] (L12) circle [radius=0.1];
 	
 	\draw[line width=0.1][dashed] (L9)--(R9);
    \draw[line width=0.1][dashed] (L10)--(R10);
    \draw[line width=0.1][dashed] (L11)--(R11);
    \draw[line width=0.1][dashed] (L12)--(R12);
 	
    \end{tikzpicture}\hfil}
    \caption{Case $2.2$ of Theorem \ref{th:daisystepup} when $|A|=r-1$.}
    \label{fig:case22}
\end{figure}
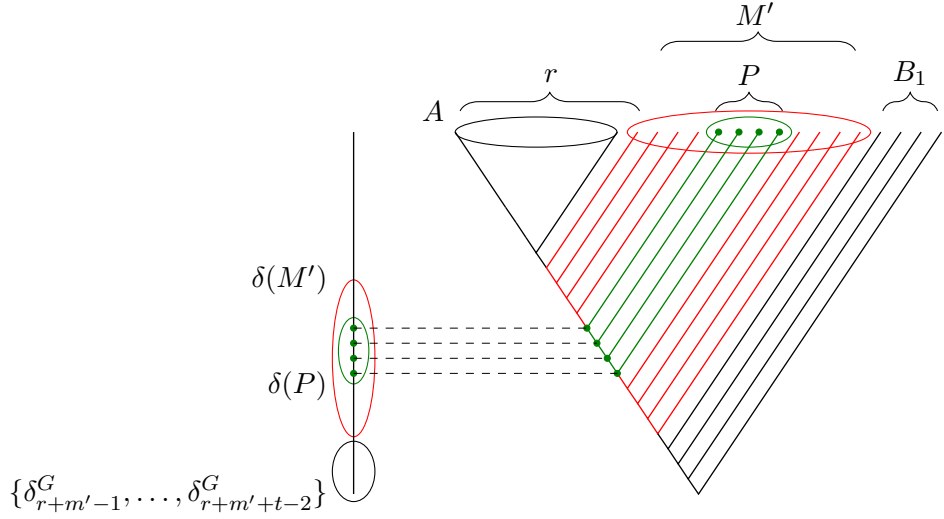

Suppose that $|A^{X_J}|=r-1$ and $B_0=\emptyset$, i.e., $\ell=r-1$, $p=0$ and $M'=\{y_r,\ldots,y_{r+m'-1}\}$. Then the projection of the relevant part of an edge $X_J$ can be written as
\begin{align*}
    \{\delta_r^{X_J},\ldots,\delta_{\ell+p+r+t-1}^{X_J}\}&=\{\delta_r^{X_J}. \ldots, \delta_{2r-2}^{X_J}\}\cup \{\delta_{r+m'-1}^G,\ldots,\delta_{r+m'+t-2}^G\}\\
    &=\delta(P)\cup \{\delta_{r+m'-1}^G,\ldots,\delta_{r+m'+t-2}^G\},
\end{align*}
since $\delta_{2r-1+i}^{X_J}=\delta_{\ell+p+r+i}^{X_J}=\delta_{\ell+p+m'+i}^G=\delta_{2r+m'-1+i}^G$ for $0\leq i \leq t-1$. Therefore, the projection of the edges $X_J$ is a simple $(r-1,m'-1,t)$-daisy with kernel $\{\delta_{r+m'-1}^G,\ldots,\delta_{r+m'+t-2}^G\}$. Because $G$ is monochromatic under $\chi_0$, the projection is also monochromatic under $\phi_{r+t-1}$, which is a contradiction.

Now suppose that $|A^{X_J}\cup B_0|=\ell+p\geq r$. The relevant projection of $X_J$ in this case would be
\begin{align*}
    \{\delta_r^{X_J},\ldots,\delta_{\ell+p+r+t-1}^{X_J}\}=\{\delta_r^G,\ldots, \delta_{\ell+p-1}^G\}&\cup\{\delta_{\ell+p}^{X_J},\ldots,\delta_{\ell+p+r-1}^{X_J}\}\\
    &\cup\{\delta_{\ell+p+m'}^G,\ldots,\delta_{\ell+p+m'+t-1}^G\},
\end{align*}
where the set $\{\delta_r^G,\ldots,\delta_{\ell+p-1}^G\}$ is empty for $\ell+p=r$. Since $\{\delta_{\ell+p}^{X_J},\ldots,\delta_{\ell+p+r-1}^{X_J}\}=\delta(\{y_{\ell+p}\}\cup P)$, we obtain that the projection of all edges $X_J$ is a simple $(r,m', \ell+p+t-r)$-daisy with kernel $\{\delta_r^G,\ldots,\delta_{\ell+p-1}^G\} \cup \{\delta_{\ell+p+m'}^G,\ldots,\delta_{\ell+p+m'+t-1}^G\}$. Because every simple $(r,m',\ell+p+t-r)$-daisy contains an $(r-1,m'-1,\ell+p+t-r+1)$-daisy and the projection is monochromatic with respect to $\phi_{\ell+p+t-1}$, we obtain a monochromatic simple $(r-1,\frac{1}{5}k^{-1/2}m^{1/2},\ell+p+t-r+1)$-daisy, which is a contradiction.
 
\vspace{0.2cm}

\noindent \underline{Case 2.2.b:} Either  $\delta_{r-3}^{X_J}<\delta_{r-2}^{X_J}<\delta_{r-1}^{X_J}$ or $\delta_{r-3}^{X_J}>\delta_{r-2}^{X_J}>\delta_{r-1}^{X_J}$.

In this case we obtain that $\chi_0(X_J)=\phi_{|\delta(X_J)|}(\delta(X_J))$, i.e., the coloring of $\chi_0$ is just the coloring of the projection of $X_J$. The proof now follows similarly as in Case 2.1.

\vspace{0.2cm}

\noindent \underline{Case 2.3:} $X_J$ is of type $5$.

\begin{figure}[h]
\centering
{\hfil \begin{tikzpicture}[scale=0.4]
    
 	\draw (-8,9.5) ellipse (1.66 and 0.3);
 	\draw[red] (0.33,9.5) ellipse (3.7 and 0.7);
 	\draw[green1] (0,9.5) ellipse (1.4 and 0.5);
 	\draw[green1] (-13,2.25) ellipse (0.5 and 1.1);
 	\draw[red] (-13,2) ellipse (0.7 and 2.6);
 	\draw (-13,5.65) ellipse (0.7 and 0.9);
 	\draw (-13,-1.75) ellipse (0.7 and 1);
 	
 	\coordinate (T0) at (-9.66,9.5);
 	\coordinate (T1) at (-6.33,9.5);
 	\coordinate (T2) at (-5.66,9.5);
 	\coordinate (T3) at (-5,9.5);
 	\coordinate (T4) at (-4.33,9.5);
 	\coordinate (T5) at (-3.66,9.5);
 	\coordinate (T6) at (-3,9.5);
 	\coordinate (T7) at (-2.33,9.5);
 	\coordinate (T8) at (-1.66,9.5);
 	\coordinate (T9) at (-1,9.5);
 	\coordinate (T10) at (-0.33,9.5);
 	\coordinate (T11) at (0.33,9.5);
 	\coordinate (T12) at (1,9.5);
 	\coordinate (T13) at (1.66,9.5);
 	\coordinate (T14) at (2.33,9.5);
 	\coordinate (T15) at (3,9.5);
 	\coordinate (T16) at (3.66,9.5);
 	\coordinate (T17) at (4.33,9.5);
 	\coordinate (T18) at (5,9.5);
 	\coordinate (T19) at (5.66,9.5);
 	\coordinate (T20) at (6.33,9.5);
    \coordinate (R1) at (-8,7);
    \coordinate (R2) at (-7.66,6.5);
    \coordinate (R3) at (-7.33,6);
    \coordinate (R4) at (-7,5.5);
    \coordinate (R5) at (-6.66,5);
    \coordinate (R6) at (-6.33,4.5);
    \coordinate (R7) at (-6,4);
    \coordinate (R8) at (-5.66,3.5);
    \coordinate (R9) at (-5.33,3);
    \coordinate (R10) at (-5,2.5);
    \coordinate (R11) at (-4.66,2);
    \coordinate (R12) at (-4.33,1.5);
    \coordinate (R13) at (-4,1);
    \coordinate (R14) at (-3.66,0.5);
    \coordinate (R15) at (-3.33,0);
    \coordinate (R16) at (-3,-0.5);
    \coordinate (R17) at (-2.66,-1);
    \coordinate (R18) at (-2.33,-1.5);
    \coordinate (R19) at (-2,-2);
    \coordinate (R20) at (-1.66,-2.5);
    
    \coordinate (L) at (-13,-2.5);
    \coordinate (L') at (-13, 9.5);
    \coordinate (L9) at (-13,3);
    \coordinate (L10) at (-13,2.5);
    \coordinate (L11) at (-13,2);
    \coordinate (L12) at (-13,1.5);

    \node (a) at (-7.95,10.7) [above,font=\small] {$A$};
	\node (b0) at (-4.6,10.7) [above,font=\small] {$B_0$};
	\node (m) at (0.2,12.7) [above,font=\small] {$M'$};
	\node (b1) at (5.3,10.7) [above,font=\small] {$B_1$};
	\node (p) at (0,10.7) [above,font=\small] {$P$};
	\node (d0) at (-13.5,5.7) [left,font=\small] {$\{\delta_{\ell}^G,\ldots, \delta_{\ell+p-1}^G\}$};
	\node (d1) at (-13.5,-2) [left,font=\small] {$\{\delta_{\ell+p+m'}^G,\ldots, \delta_{\ell+p+m'+t-1}^G\}$};
	\node (p') at (-13.5,2.7) [left,font=\small] {$\delta(\{y_{\ell+p}\}\cup P)$};
	\node (m') at (-13.5,0) [left,font=\small] {$\delta(\{y_{\ell+p}\}\cup M')$};

    \draw [decorate,
    decoration = {brace, amplitude=7pt}] (-9.5,10.2) --  (-6.4,10.2);
    \draw [decorate,
    decoration = {brace, amplitude=7pt}] (-5.5,10.2) --  (-3.7,10.2);
    \draw [decorate,
    decoration = {brace, amplitude=7pt}] (-2.9,12.2) --  (3.5,12.2);
    \draw [decorate,
    decoration = {brace, amplitude=7pt}] (4.4,10.2) --  (6.2,10.2);
    \draw [decorate,
    decoration = {brace, amplitude=7pt}] (-1.1,10.2) --  (1.1,10.2);
    
    \draw[line width=0.5] (T0)--(R1)--(T1);
    \draw[line width=0.5] (R1)--(R2)--(T2);
    \draw[line width=0.5] (R2)--(R3)--(T3);
    \draw[line width=0.5] (R3)--(R4)--(T4);
    \draw[line width=0.5] (R4)--(R5)--(T5);
    \draw[line width=0.5][red] (R5)--(R6)--(T6);
    \draw[line width=0.5][red] (R6)--(R7)--(T7);
    \draw[line width=0.5][red] (R7)--(R8)--(T8);
    \draw[line width=0.5][red] (R8)--(R9);
    \draw[line width=0.5][green1] (R9)--(T9);
    \draw[line width=0.5][green1] (R9)--(R10)--(T10);
    \draw[line width=0.5][green1] (R10)--(R11)--(T11);
    \draw[line width=0.5][green1] (R11)--(R12)--(T12);
    \draw[line width=0.5][red] (R12)--(R13)--(T13);
    \draw[line width=0.5][red] (R13)--(R14)--(T14);
    \draw[line width=0.5][red] (R14)--(R15)--(T15);
    \draw[line width=0.5][red] (R15)--(R16)--(T16);
    \draw[line width=0.5] (R16)--(R17)--(T17);
    \draw[line width=0.5] (R17)--(R18)--(T18);
    \draw[line width=0.5] (R18)--(R19)--(T19);
    \draw[line width=0.5] (R19)--(R20)--(T20);
    \draw[line width=0.5] (L)--(L');
 	
 	\draw[fill][green1] (T9) circle [radius=0.1];
 	\draw[fill][green1] (T10) circle [radius=0.1];
 	\draw[fill][green1] (T11) circle [radius=0.1];
 	\draw[fill][green1] (T12) circle [radius=0.1];
 	\draw[fill][green1] (R9) circle [radius=0.1];
 	\draw[fill][green1] (R10) circle [radius=0.1];
 	\draw[fill][green1] (R11) circle [radius=0.1];
 	\draw[fill][green1] (R12) circle [radius=0.1];
 	\draw[fill][green1] (L9) circle [radius=0.1];
 	\draw[fill][green1] (L10) circle [radius=0.1];
 	\draw[fill][green1] (L11) circle [radius=0.1];
 	\draw[fill][green1] (L12) circle [radius=0.1];
 	
 	\draw[line width=0.1][dashed] (L9)--(R9);
    \draw[line width=0.1][dashed] (L10)--(R10);
    \draw[line width=0.1][dashed] (L11)--(R11);
    \draw[line width=0.1][dashed] (L12)--(R12);
 	
    \end{tikzpicture}\hfil}
    \caption{Case $2.3$ of Theorem \ref{th:daisystepup}}
    \label{fig:case23}
\end{figure}
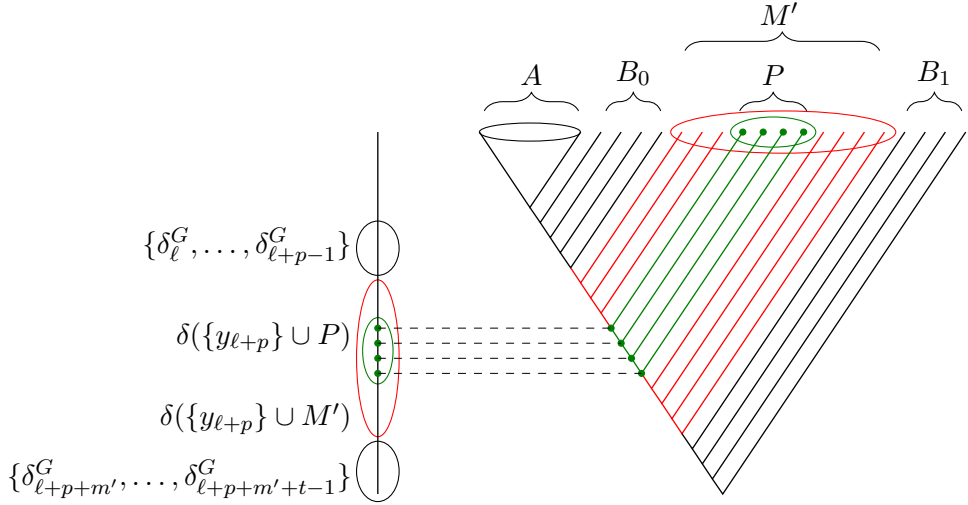

If $X_J$ is of type $5$, then $|A^{X_J}|>r$ and $|B^{X_J}|=p+r+t\geq r$. By the coloring definition, we have $\chi_0(X_J)=\phi_{p+r+t}(\{\delta_{\ell}^{X_J},\ldots,\delta_{\ell+p+r+t-1}^{X_J}\})$. The projection here can be rewritten as
\begin{align*}
    \{\delta_{\ell}^{X_J},\ldots,\delta_{\ell+p+r+t-1}^{X_J}\}=\{\delta_{\ell}^G,\ldots,\delta_{\ell+p-1}^G\}&\cup \{\delta(\{y_{\ell+p}\}\cup P)
    \\&\cup \{\delta_{\ell+p+m'}^G,\ldots, \delta_{\ell+p+m'+t-1}^G\}.
\end{align*}

Thus, the relevant projection over all edges $X_J$ is a simple $(r,m',p+t)$-daisy with kernel $\{\delta_{\ell}^G,\ldots,\delta_{\ell+p-1}^G\}\cup \{\delta_{\ell+p+m'}^G,\ldots, \delta_{\ell+p+m'+t-1}^G\}$. Therefore, by the same argument did in the previous cases, we reach a contradiction since there is no monochromatic simple $(r-1,\frac{1}{5}k^{-1/2}m^{1/2},p+t+1)$-daisy in the coloring $\phi_{p+r+t}$.
\end{proof}

\begin{proof}[Proof of Theorem \ref{thm:daisyramsey}]
We will prove by induction on the size of $r$ that there exists an absolute positive constants $c$ and $c'$ not depending on $k$ and $r$ such that
\begin{align*}
    D_r^{\smp}(m,k)=t_{r-2}(c'(5\sqrt{k})^{2^{5-r}-4}m^{2^{4-r}})\geq t_{r-2}(ck^{2^{4-r}-2}m^{2^{4-r}})
\end{align*} 
holds for $k\geq 1$ and $m\geq (25k)^{2^{r}-1}$. For $r=3$, the result follows by the next proposition given in \cite{PRS}.

\begin{proposition}[\cite{PRS}, Proposition 1.2]\label{prop:r=3}
There exists a positive constant $c'$ not depending on $k$ such that
\begin{align*}
    D_3(m,k)\geq 2^{c'm^2}
\end{align*}
holds for $m>3$.
\end{proposition}

Now suppose that $r\geq 4$ and that for any integer $\ell<r$ the induction hypothesis is satisfied, i.e.,
\begin{align*}
    D_{\ell}^{\smp} (m,k)\geq t_{\ell-2}(c'(5\sqrt{k})^{2^{5-\ell}-4}m^{2^{4-\ell}})
\end{align*}
for $m\geq (25k)^{2^{\ell}-1}$ and $k\geq 1$. Let $N=\min_{0\leq i \leq k-1}{D_{r-1}^{\smp}(\frac{1}{5}k^{-1/2}m^{1/2},i)}$. For $i=0$, by equation (\ref{eq:ramsey}) we have that
\begin{align*}
    D_{r-1}^{\smp}\left(\frac{1}{5}k^{-1/2}m^{1/2},0\right)\geq R_{r-1}\left(\frac{1}{5}k^{-1/2}m^{1/2}\right)\geq t_{r-2}(c_1k^{-1}m)
\end{align*}
for a positive constant $c_1$. Since $m\geq (25k)^{2^r-1}$, we obtain that $\frac{1}{5}k^{-1/2}m^{1/2}\geq (25k)^{2^{r-1}-1}$. Thus, by induction hypothesis we also have that
\begin{align*}
    D_{r-1}^{\smp}\left(\frac{1}{5}k^{-1/2}m^{1/2},i\right)\geq t_{r-3}\left(c'(5\sqrt{i})^{2^{6-r}-4}\left(\frac{1}{5}k^{-1/2}m^{1/2}\right)^{2^{5-r}}\right),
\end{align*}
for $i\geq 1$. Therefore,
\begin{align*}
N &\geq \min\left\{t_{r-2}(c_1k^{-1}m),\min_{1\leq i \leq k}\left\{t_{r-3}\left(c'(5\sqrt{i})^{2^{6-r}-4}\left(\frac{1}{5}k^{-1/2}m^{1/2}\right)^{2^{5-r}}\right)\right\}\right\}\\
&\geq  t_{r-3}(c'(5\sqrt{k})^{2^{5-r}-4}m^{2^{4-r}}).
\end{align*}
Finally, Theorem \ref{th:daisystepup}, applied to $m\geq (25k)^{2^r-1}\geq 100kr^2$, gives us that
\begin{align*}
    D_r^{\smp}(m,k)\geq 2^N\geq t^{r-2}(c'(5\sqrt{k})^{2^{5-r}-4}m^{2^{4-r}}).
\end{align*}
\end{proof}

\section*{Acknowledgments}
The author thanks Pavel Pudl\'{a}k and Vojtech R\"{o}dl for their comments on earlier versions of the manuscript.

\bibliography{literature}

\end{document}